\documentclass{amsart}

\usepackage{amssymb}
\usepackage{amsmath}
\usepackage{amsthm}
\usepackage{amsfonts}

\usepackage{float}
\usepackage[small]{caption}
\usepackage[lofdepth, lotdepth, captionskip=10pt, labelfont=rm]{subfig}

\usepackage{pstricks}
\usepackage{epsfig}
\usepackage{pst-node}
\usepackage{enumitem}

\usepackage[linktocpage=true]{hyperref} 

\usepackage[alphabetic, msc-links, nobysame, lite]{amsrefs}

\newcommand{\inv}{^{-1}}

\newcommand{\bbC}{\mathbb{C}}
\newcommand{\bbD}{\mathbb{D}}
\newcommand{\bbH}{\mathbb{H}}
\newcommand{\bbR}{\mathbb{R}}
\newcommand{\bbS}{\mathbb{S}}
\newcommand{\bbT}{\mathbb{T}}

\newcommand{\calP}{\mathcal{P}}

\newcommand{\jdIntersection}{\partial K \cap \partial \tilde K}
\newcommand{\jdUnion}{\partial K \cup \partial \tilde K}

\newcommand{\fNone}{F_\varnothing}
\newcommand{\fBoth}{F_\cap}

\newcommand{\cpc}{\operatorname{CPC}}

\theoremstyle{plain}

\newtheorem{lemma}{Lemma}[section]
\newtheorem{fact}[lemma]{Fact}
\newtheorem{observation}[lemma]{Observation}

\newtheorem{theorem}[lemma]{Theorem}
\newtheorem{lemmafree}{Lemma}

\theoremstyle{definition}

\newtheorem{definition}[lemma]{Definition}
\newtheorem*{definition*}{Definition}
\newtheorem{remark}[lemma]{Remark}
\newtheorem*{remark*}{Remark}
\newtheorem{additivity}[lemma]{Index Additivity Lemma}
\newtheorem{cil}[lemma]{Circle Index Lemma}
\newtheorem{incompat}[lemma]{Incompatibility Theorem}
\newtheorem{threep}[lemma]{Three Point Prescription Lemma}

\begin{document}

\title[Fixed-point index and torus parametrization]{Fixed-point index, the Incompatibility Theorem, and torus parametrization}
\author{Andrey M.\ Mishchenko}
\thanks{The author was partially supported by NSF grants DMS-0456940, DMS-0555750, DMS-0801029, DMS-1101373.}
\subjclass[2010]{54H25 (primary), 52C26 (secondary)}
\email{mishchea@gmail.com}
\date{\today}

\begin{abstract}
The fixed-point index of a homeomorphism of Jordan curves measures the number of fixed-points, with multiplicity, of the extension of the homeomorphism to the full Jordan domains in question. The now-classical Circle Index Lemma says that the fixed-point index of a positive-orientation-preserving homeomorphism of round circles is always non-negative. We begin by proving a generalization of this lemma, to accommodate Jordan curves bounding domains which do not disconnect each other. We then apply this generalization to give a new proof of Schramm's Incompatibility Theorem, which was used by Schramm to give the first proof of the rigidity of circle packings filling the complex and hyperbolic planes. As an example application, we include outlines of proofs of these circle packing theorems.

We then introduce a new tool,
the so-called torus para\-metri\-zation,
for working with fixed-point index, which allows some problems concerning this quantity to be approached combinatorially. We apply torus parametrization to give the first purely topological proof of the following lemma: given two positively oriented Jordan curves, one may essentially prescribe the images of three points of one of the curves in the other, and obtain an orientation-preserving homeomorphism between the curves, having non-negative fixed-point index, which respects this prescription. This lemma is essential to our proof of the Incompatibility Theorem.
\end{abstract}

\maketitle

\tableofcontents

\section{Introduction}
\label{chap:intro fpi}

This article is concerned with a topological quantity, the so-called \emph{fixed-point index} of a homeomorphism of Jordan curves, which has proven useful in the study of various areas of complex analysis. We begin with its definition:

\begin{definition}
A \emph{Jordan curve} is a homeomorphic image of a topological circle $\bbS^1$ in the complex plane $\bbC$. A \emph{Jordan domain} is a bounded open set in $\bbC$ with Jordan curve boundary. We use the term \emph{closed Jordan domain} or \emph{compact Jordan domain} to refer to the closure of a Jordan domain. We define the \emph{positive orientation} on a Jordan curve as usual. That is, if $K$ is a closed Jordan domain, then as we traverse $\partial K$ in what we call the \emph{positive} direction, the interior of $K$ stays to the left.

Let $K$ and $\tilde K$ be closed Jordan domains. Let $\phi:\partial K \to \partial \tilde K$ be a homeomorphism of Jordan curves which is fixed-point-free and orientation-preserving. We call such a homeomorphism \emph{indexable}. Then $\{\phi(z) - z\}_{z\in \partial K}$ is a closed curve in the plane which misses the origin. It has a natural orientation induced by traversing $\partial K$ positively. Then we define the \emph{fixed-point index} of $\phi$, denoted $\eta(\phi)$, to be the winding number of $\{\phi(z) - z\}_{z\in \partial K}$ around the origin.
\end{definition}

Two examples are shown in Figures \ref{fig:ex disjoint then winding 0} and \ref{fig:ex neg winding forced}. We remark that the fixed-point index depends crucially on the choice of homeomorphism, and also on the way that the sets $K$ and $\tilde K$ are juxtaposed. It is a worthwhile exercise to construct an indexable homeomorphism $\partial K\to \partial \tilde K$, for $K$ and $\tilde K$ as in Figure \ref{fig:ex neg winding forced}, having fixed-point index unequal to $-1$.\medskip

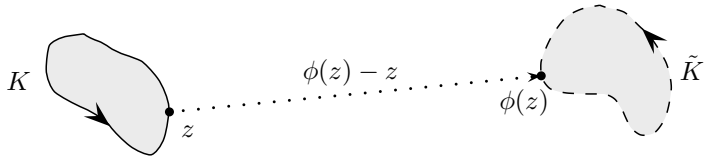
\begin{figure}
\centering
% Generated with LaTeXDraw 2.0.8
% Sat Dec 17 16:29:54 EST 2011
% \usepackage[usenames,dvipsnames]{pstricks}
% \usepackage{epsfig}
% \usepackage{pst-grad} % For gradients
% \usepackage{pst-plot} % For axes
\scalebox{1} % Change this value to rescale the drawing.
{
\begin{pspicture}(0,-1.02)(10.129063,1.015)
\definecolor{color31b}{rgb}{0.9215686274509803,0.9215686274509803,0.9215686274509803}
\usefont{T1}{ppl}{m}{n}
\rput(0.28453124,-0.015){$K$}
\usefont{T1}{ppl}{m}{n}
\rput(7.004531,-0.295){$\phi(z)$}
\usefont{T1}{ppl}{m}{n}
\rput(2.5345314,-0.675){$z$}
\usefont{T1}{ppl}{m}{n}
\rput(4.6945314,0.105){$\phi(z) - z$}
\pscustom[linewidth=0.02,fillstyle=solid,fillcolor=color31b]
{
\newpath
\moveto(1.15,0.615)
\lineto(0.94,0.585)
\curveto(0.835,0.57)(0.71,0.505)(0.69,0.455)
\curveto(0.67,0.405)(0.65,0.25)(0.65,0.145)
\curveto(0.65,0.04)(0.71,-0.115)(0.77,-0.165)
\curveto(0.83,-0.215)(1.02,-0.32)(1.15,-0.375)
\curveto(1.28,-0.43)(1.455,-0.55)(1.5,-0.615)
\curveto(1.545,-0.68)(1.675,-0.81)(1.76,-0.875)
\curveto(1.845,-0.94)(1.985,-1.0)(2.04,-0.995)
\curveto(2.095,-0.99)(2.185,-0.875)(2.22,-0.765)
\curveto(2.255,-0.655)(2.285,-0.45)(2.28,-0.355)
\curveto(2.275,-0.26)(2.2,-0.07)(2.13,0.025)
\curveto(2.06,0.12)(1.95,0.235)(1.91,0.255)
\curveto(1.87,0.275)(1.765,0.315)(1.7,0.335)
\curveto(1.635,0.355)(1.53,0.405)(1.49,0.435)
\curveto(1.45,0.465)(1.38,0.505)(1.35,0.515)
\curveto(1.32,0.525)(1.27,0.545)(1.25,0.555)
\curveto(1.23,0.565)(1.2,0.575)(1.15,0.615)
}
\psline[linestyle=none,linewidth=0.1,arrows=->](1.28,-0.43)(1.5,-0.615)

\pscustom[linewidth=0.02,fillstyle=solid,fillcolor=color31b,linestyle=dashed,dash=0.16cm 0.16cm]
{
\newpath
\moveto(8.01,0.995)
\lineto(7.84,0.965)
\curveto(7.755,0.95)(7.615,0.885)(7.56,0.835)
\curveto(7.505,0.785)(7.4,0.64)(7.35,0.545)
\curveto(7.3,0.45)(7.24,0.3)(7.23,0.245)
\curveto(7.22,0.19)(7.225,0.085)(7.24,0.035)
\curveto(7.255,-0.015)(7.31,-0.09)(7.35,-0.115)
\curveto(7.39,-0.14)(7.515,-0.17)(7.6,-0.175)
\curveto(7.685,-0.18)(7.815,-0.185)(7.86,-0.185)
\curveto(7.905,-0.185)(8.0,-0.185)(8.05,-0.185)
\curveto(8.1,-0.185)(8.185,-0.225)(8.22,-0.265)
\curveto(8.255,-0.305)(8.31,-0.395)(8.33,-0.445)
\curveto(8.35,-0.495)(8.39,-0.58)(8.41,-0.615)
\curveto(8.43,-0.65)(8.495,-0.705)(8.54,-0.725)
\curveto(8.585,-0.745)(8.68,-0.735)(8.73,-0.705)
\curveto(8.78,-0.675)(8.86,-0.56)(8.89,-0.475)
\curveto(8.92,-0.39)(8.955,-0.21)(8.96,-0.115)
\curveto(8.965,-0.02)(8.925,0.165)(8.88,0.255)
\curveto(8.835,0.345)(8.75,0.485)(8.71,0.535)
\curveto(8.67,0.585)(8.6,0.66)(8.57,0.685)
\curveto(8.54,0.71)(8.48,0.77)(8.45,0.805)
\curveto(8.42,0.84)(8.36,0.89)(8.33,0.905)
\curveto(8.3,0.92)(8.235,0.94)(8.2,0.945)
\curveto(8.165,0.95)(8.115,0.955)(8.01,0.995)
}
\psline[linestyle=none,linewidth=0.1,arrows=->](8.67,0.585)(8.57,0.685)

\psline[linewidth=0.04cm,fillcolor=color31b,linestyle=dotted,dotsep=0.16cm,arrowsize=0.05291667cm 2.0,arrowlength=1.4,arrowinset=0.4]{->}(2.29,-0.425)(7.23,0.055)
\psdots[dotsize=0.12](2.29,-0.425)
\psdots[dotsize=0.12](7.23,0.055)
\usefont{T1}{ppl}{m}{n}
\rput(9.234531,0.165){$\tilde K$}
\end{pspicture} 
}

\caption[Two closed Jordan domains $K$ and $\tilde K$ so that any indexable homeomorphism $\phi:\partial K \to \partial \tilde K$ satisfies $\eta(f) = 0$]{
{\bf Two closed Jordan domains $K$ and $\tilde K$ so that any indexable homeomorphism $\phi:\partial K \to \partial \tilde K$ satisfies $\eta(\phi) = 0$.}  The arrows on $\partial K$ and $\partial \tilde K$ indicate the positive orientations on these Jordan curves. In this case $\phi$ is indexable so long as it is orientation-preserving; the fixed-point-free condition is automatic because $\partial K$ and $\partial \tilde K$ do not meet. The dashed arrow represents a vector of the form $\phi(z) - z$. The vector $\phi(z) - z$ must always point ``to the right,'' so the curve $\{\phi(z) - z\}_{z\in \partial K}$ has winding number $0$ around the origin, thus $\eta(\phi) = 0$.
}
\label{fig:ex disjoint then winding 0}
\end{figure}

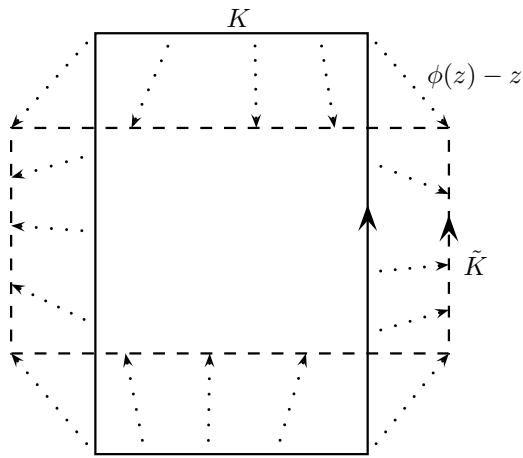
\begin{figure}
\centering
% Generated with LaTeXDraw 2.0.8
% Sat Dec 17 16:44:01 EST 2011
% \usepackage[usenames,dvipsnames]{pstricks}
% \usepackage{epsfig}
% \usepackage{pst-grad} % For gradients
% \usepackage{pst-plot} % For axes
\scalebox{1} % Change this value to rescale the drawing.
{
\begin{pspicture}(0,-3.0367188)(6.9990625,3.0567188)
\definecolor{color96b}{rgb}{0.9215686274509803,0.9215686274509803,0.9215686274509803}
\usefont{T1}{ppl}{m}{n}
\rput(6.174531,-0.46671876){$\tilde K$}
\usefont{T1}{ppl}{m}{n}
\rput(3.0245314,2.7932813){$K$}
\psframe[linewidth=0.03,linestyle=dashed,dash=0.16cm 0.16cm,dimen=middle](5.82,1.3232813)(0.0,-1.6767187)
\psframe[linewidth=0.03,dimen=middle](4.74,2.5832813)(1.12,-3.0167189)
\psline[linewidth=0.04cm,fillcolor=color96b,linestyle=dotted,dotsep=0.16cm,arrowsize=0.05291667cm 2.0,arrowlength=1.4,arrowinset=0.4]{<-}(5.82,1.3232813)(4.74,2.5832813)
\psline[linewidth=0.04cm,fillcolor=color96b,linestyle=dotted,dotsep=0.16cm,arrowsize=0.05291667cm 2.0,arrowlength=1.4,arrowinset=0.4]{<-}(4.3,1.3232813)(4.1,2.5832813)
\psline[linewidth=0.04cm,fillcolor=color96b,linestyle=dotted,dotsep=0.16cm,arrowsize=0.05291667cm 2.0,arrowlength=1.4,arrowinset=0.4]{<-}(3.26,1.3232813)(3.24,2.5832813)
\psline[linewidth=0.04cm,fillcolor=color96b,linestyle=dotted,dotsep=0.16cm,arrowsize=0.05291667cm 2.0,arrowlength=1.4,arrowinset=0.4]{<-}(1.6,1.3232813)(2.16,2.5832813)
\psline[linewidth=0.04cm,fillcolor=color96b,linestyle=dotted,dotsep=0.16cm,arrowsize=0.05291667cm 2.0,arrowlength=1.4,arrowinset=0.4]{<-}(0.0,1.3232813)(1.12,2.5832813)
\psline[linewidth=0.04cm,fillcolor=color96b,linestyle=dotted,dotsep=0.16cm,arrowsize=0.05291667cm 2.0,arrowlength=1.4,arrowinset=0.4]{<-}(0.0,0.66328126)(1.12,0.98328125)
\psline[linewidth=0.04cm,fillcolor=color96b,linestyle=dotted,dotsep=0.16cm,arrowsize=0.05291667cm 2.0,arrowlength=1.4,arrowinset=0.4]{<-}(0.0,0.02328125)(1.12,-0.05671875)
\psline[linewidth=0.04cm,fillcolor=color96b,linestyle=dotted,dotsep=0.16cm,arrowsize=0.05291667cm 2.0,arrowlength=1.4,arrowinset=0.4]{<-}(0.0,-0.75671875)(1.12,-1.2967187)
\psline[linewidth=0.04cm,fillcolor=color96b,linestyle=dotted,dotsep=0.16cm,arrowsize=0.05291667cm 2.0,arrowlength=1.4,arrowinset=0.4]{<-}(0.0,-1.6767187)(1.12,-3.0167189)
\psline[linewidth=0.04cm,fillcolor=color96b,linestyle=dotted,dotsep=0.16cm,arrowsize=0.05291667cm 2.0,arrowlength=1.4,arrowinset=0.4]{<-}(1.52,-1.6767187)(1.78,-3.0167189)
\psline[linewidth=0.04cm,fillcolor=color96b,linestyle=dotted,dotsep=0.16cm,arrowsize=0.05291667cm 2.0,arrowlength=1.4,arrowinset=0.4]{<-}(2.64,-1.6767187)(2.62,-3.0167189)
\psline[linewidth=0.04cm,fillcolor=color96b,linestyle=dotted,dotsep=0.16cm,arrowsize=0.05291667cm 2.0,arrowlength=1.4,arrowinset=0.4]{<-}(3.92,-1.6767187)(3.52,-3.0167189)
\psline[linewidth=0.04cm,fillcolor=color96b,linestyle=dotted,dotsep=0.16cm,arrowsize=0.05291667cm 2.0,arrowlength=1.4,arrowinset=0.4]{<-}(5.82,-1.6767187)(4.74,-3.0167189)
\psline[linewidth=0.04cm,fillcolor=color96b,linestyle=dotted,dotsep=0.16cm,arrowsize=0.05291667cm 2.0,arrowlength=1.4,arrowinset=0.4]{<-}(5.82,-1.0967188)(4.74,-1.4167187)
\psline[linewidth=0.04cm,fillcolor=color96b,linestyle=dotted,dotsep=0.16cm,arrowsize=0.05291667cm 2.0,arrowlength=1.4,arrowinset=0.4]{<-}(5.82,-0.49671876)(4.74,-0.5967187)
\psline[linewidth=0.04cm,fillcolor=color96b,linestyle=dotted,dotsep=0.16cm,arrowsize=0.05291667cm 2.0,arrowlength=1.4,arrowinset=0.4]{<-}(5.82,0.44328126)(4.74,0.88328123)

\psline[linewidth=0.1cm,linestyle=none]{->}(4.74,0.00328125)(4.74,0.30328125)
\psline[linewidth=0.1cm,linestyle=none]{->}(5.82,-0.03671875)(5.82,0.16328125)
\usefont{T1}{ptm}{m}{n}
\rput(6.164531,2.0132813){$\phi(z) - z$}
\end{pspicture} 
}
\caption[An indexable homeomorphism $\phi:\partial K \to \partial \tilde K$ so that $\eta(\phi) = -1$]
{
\label{fig:ex neg winding forced}
{\bf An indexable homeomorphism $\phi:\partial K \to \partial \tilde K$ so that $\eta(\phi) = -1$.}  Suppose we insist that $\phi$ identifies the respective corners as shown. Then tracing the path of the dashed vectors $\phi(z) - z$ as $z$ traverses $\partial K$ positively, we see that $\phi(z) - z$ must wind once clockwise around the origin, thus $\eta(\phi)=-1$.
}
\end{figure}

Fixed-point index has found applications for example in the theories of circle packing \cite{MR1207210}, Koebe uniformization \cite{MR1207210}, and Sierpinski carpets \cite{MR2900233}*{Section 12}. In all of these settings, it has been applied to prove powerful existence, rigidity, and uniformization statements. Most recently, the current author has used fixed-point index, including torus parametrization, to prove rigidity statements for collections of possibly-overlapping round disks, see \citelist{\cite{mishchenko-dissertation} \cite{mishchenko-rigidity-2012}}.\medskip

The fixed-point index measures the following topological quantity: suppose that $K$ and $\tilde K$ are closed Jordan domains, and $\Phi : K \to \tilde K$ is a homeomorphism, having finitely many fixed points, which restricts to an indexable homeomorphism $\partial\Phi : \partial K \to \partial \tilde K$. There is a well-understood notion of the \emph{multiplicity} of a fixed point of $\Phi$. Then the fixed-point index $\eta(\partial\Phi)$ counts the number of fixed-points of $\Phi$, with multiplicity. For more discussion along these lines, see \cite{MR1207210}*{Section 2}.\medskip

In this article, we describe a new technique for working with fixed-point index, which we call the \emph{torus parametrization} of a pair of Jordan curves, defined in Section \ref{chap:torus}. We apply torus parametrization to give a new, elementary proof of the following fundamental lemma:

\begin{threep}
\label{threep}
\label{3p}
\label{lem:3p}
\label{lem:three points}
\label{tppl}
Let $K$ and $\tilde K$ be compact Jordan domains in transverse position, with boundaries oriented positively. Let $z_1,z_2,z_3\in \partial K\setminus \partial \tilde K$ appear in counterclockwise order, similarly $\tilde z_1,\tilde z_2, \tilde z_3\in \partial \tilde K\setminus \partial K$. Then there is an indexable homeomorphism $\phi:\partial K \to \partial \tilde K$ sending $z_i \mapsto \tilde z_i$ for $i=1,2,3$, so that $\eta(\phi) \ge 0$.
\end{threep}

\noindent Two Jordan domains are in \emph{transverse position} if their boundary Jordan curves cross wherever they meet, c.f.\ Definition \ref{transverse position} in Section \ref{sec background}. The example given in Figure \ref{fig:ex neg winding forced} shows that if we prescribe the images of four points, then a negative fixed-point index may be forced.

A version of the Three Point Prescription Lemma \ref{threep} is stated in \cite{MR2131318}*{Lemma 8.14}, but we have not been able to fill in the details of the argument. The idea of the approach is as follows: first, any Riemann mapping $\Phi:\Omega \to \tilde\Omega$ between open Jordan domains having non-self-intersecting boundaries extends to a homeomorphism $\partial \Phi :  \partial \Omega \to \partial \tilde \Omega$ of their boundaries, and we may prescribe the images of three points of $\partial \Omega$ in $\partial \tilde \Omega$ by post-composing with self-biholomorphisms of $\tilde\Omega$. Next, it is known that any isolated fixed point of a holomorphic map has non-negative multiplicity, see \cite{MR1207210}*{Section 2}. Thus the map $\partial \Phi$, if it is indexable, has non-negative fixed-point index, because the fixed-point index of $\partial \Phi$ counts the fixed points of $\Phi$ with multiplicity, completing the argument. However, it is not clear how to deal with possible fixed points in the induced boundary map $\partial \Phi$. Our proof of Lemma \ref{3p} uses only induction and plane topology arguments and is given in Section \ref{sec:3p}.

For a discussion on the strength of the hypotheses of
Lemma \ref{tppl}, refer to
Remark \ref{closing-remark}
at the end of the article.\medskip

We also state and prove a new fundamental lemma on fixed-point index, generalizing the well-known Circle Index Lemma \ref{cil} which states that the fixed-point index of an indexable homeomorphism between circles is always non-negative. The Circle Index Lemma was a crucial ingredient in all of the applications of fixed-point index described above. In our generalization, round disks are replaced by closed Jordan domains which do not disconnect each other. In particular, the closed Jordan domains $K$ and $\tilde K$ are said to \emph{cut each other} if $K\setminus \tilde K$ or $\tilde K \setminus K$ is disconnected. Then:

\begin{lemma}
\label{no cut index}
Let $K$ and $\tilde K$ be closed Jordan domains in transverse position, which do not cut each other, having boundaries oriented positively. Let $\phi : \partial K \to \partial \tilde K$ be an indexable homeomorphism. Then $\eta(\phi) \ge 0$.
\end{lemma}

\noindent The proof appears at the end of Section \ref{incompat section}.\medskip

As an example of the power of fixed-point index, we apply the Three Point Prescription Lemma \ref{threep} and Lemma \ref{no cut index} to prove a version of the Incompatibility Theorem of Schramm \cite{MR1076089}*{Theorem 3.1}, as described in Section \ref{incompat section}. The Incompatibility Theorem is then easily applied in Section \ref{cp proofs} to prove some well-known rigidity theorems for circle packings. The ideas for these proofs are borrowed from \citelist{\cite{MR1207210} \cite{MR2131318}*{Chapter 8}}.\medskip

\noindent {\bf Acknowledgments.}
Thanks to Jordan Watkins for many fruitful discussions,
especially for pointing us strongly in the direction of
what we call \emph{torus parametrization}.
Thanks also to Mario Bonk for many fruitful discussions,
especially for suggestions that greatly simplified
the proof of Lemma \ref{no cut index}.
Thanks finally to the anonymous referee,
who found several errors in an earlier version of the article,
and whose suggestions and comments improved the exposition
and led to a greatly simplified
proof of the Three Point Prescription Lemma \ref{tppl}.

\section{Background lemmas and definitions}
\label{sec background}

In the upcoming discussion it will be useful to have access to two well-known lemmas on fixed-point index. This section is devoted to introducing these two lemmas.\medskip

Our first background lemma says essentially that ``the fixed-point index between two circles is always non-negative'':

\begin{cil}
\label{cil}
\label{lem:circle index lemma}
Let $K$ and $\tilde K$ be closed Jordan domains in $\bbC$, with boundaries oriented positively, and let $\phi:\partial K \to \partial \tilde K$ be an indexable homeomorphism. Then the following hold.
\begin{enumerate}
\item We have $\eta(\phi)=\eta(\phi^{-1})$.
\item If $K\subseteq \tilde K$ or $\tilde K\subseteq K$, then $\eta(\phi)=1$.
\item If $K$ and $\tilde K$ have disjoint interiors, then $\eta(\phi) = 0$.
\item If $\partial K$ and $\partial \tilde K$ intersect in exactly two points, then $\eta(\phi)\ge 0$.
\end{enumerate}
As a consequence of the above, if $K$ and $\tilde K$ are closed disks in the plane, then $\eta(\phi) \ge 0$.
\end{cil}

\noindent Lemma \ref{cil} can be found in \cite{MR1207210}*{Lemma 2.2}, with a clear and complete proof. There it is indicated that a version of the lemma appeared earlier in \cite{MR0051934}.\medskip

The moral of our second background lemma is that fixed-point indices ``add nicely'':

\begin{additivity}
\label{lem:indices add}
\label{additivity}
\label{ial}
Suppose that $K$ and $L$ are interiorwise disjoint closed Jordan domains which meet along a single positive-length Jordan arc $\partial K \cap \partial L$, similarly for $\tilde K$ and $\tilde L$. Then $K\cup L$ and $\tilde K \cup \tilde L$ are closed Jordan domains.

Let $\phi:\partial K\to \partial \tilde K$ and $\psi:\partial L\to \partial\tilde L$ be indexable homeomorphisms. Suppose that $\phi$ and $\psi$ agree on $\partial K \cap \partial L$. Let $\theta:\partial (K\cup L) \to \partial (\tilde K \cup \tilde L)$ be induced via restriction to $\phi$ or $\psi$ as necessary. Then $\theta$ is an indexable homeomorphism and $\eta(\theta) = \eta(\phi) + \eta(\psi)$.
\end{additivity}

\begin{figure}
\centering

% Generated with LaTeXDraw 2.0.8
% Wed Nov 30 14:42:46 EST 2011
% \usepackage[usenames,dvipsnames]{pstricks}
% \usepackage{epsfig}
% \usepackage{pst-grad} % For gradients
% \usepackage{pst-plot} % For axes
\scalebox{.6} % Change this value to rescale the drawing.
{
\begin{pspicture}(0,-2.49)(7.2690625,2.51)

\definecolor{color111b}{rgb}{0.9019607843137255,0.9019607843137255,0.9019607843137255}

\pscustom[linestyle=none,fillstyle=solid,fillcolor=color111b]
{
\newpath
\moveto(3.30,1.69)
\lineto(3.39,1.79)
\curveto(3.42,1.84)(3.52,1.905)(3.59,1.92)
\curveto(3.66,1.935)(3.815,1.995)(3.9,2.04)
\curveto(3.985,2.085)(4.17,2.205)(4.27,2.28)
\curveto(4.37,2.355)(4.595,2.45)(4.72,2.47)
\curveto(4.845,2.49)(5.155,2.45)(5.34,2.39)
\curveto(5.525,2.33)(5.83,2.045)(5.95,1.82)
\curveto(6.07,1.595)(6.25,1.16)(6.31,0.95)
\curveto(6.37,0.74)(6.52,0.365)(6.61,0.2)
\curveto(6.7,0.035)(6.815,-0.34)(6.84,-0.55)
\curveto(6.865,-0.76)(6.765,-1.19)(6.64,-1.41)
\curveto(6.515,-1.63)(6.165,-1.925)(5.94,-2.0)
\curveto(5.715,-2.075)(5.345,-2.185)(5.2,-2.22)
\curveto(5.055,-2.255)(4.72,-2.28)(4.53,-2.27)
\curveto(4.34,-2.26)(4.065,-2.145)(3.98,-2.04)
\curveto(3.895,-1.935)(3.79,-1.74)(3.77,-1.65)
\curveto(3.75,-1.56)(3.705,-1.44)(3.68,-1.41)
\curveto(3.655,-1.38)(3.615,-1.33)(3.6,-1.31)
\curveto(3.585,-1.29)(3.57,-1.265)(3.57,-1.25)
\lineto(1,0)
\lineto(3.30,1.69)
}

\pscustom[linestyle=none,fillstyle=solid,fillcolor=color111b]
{
\newpath
\moveto(3.27,1.83)
\lineto(3.21,1.93)
\curveto(3.18,1.98)(3.055,2.09)(2.96,2.15)
\curveto(2.865,2.21)(2.635,2.305)(2.5,2.34)
\curveto(2.365,2.375)(2.07,2.385)(1.91,2.36)
\curveto(1.75,2.335)(1.455,2.235)(1.32,2.16)
\curveto(1.185,2.085)(0.93,1.9)(0.81,1.79)
\curveto(0.69,1.68)(0.445,1.415)(0.32,1.26)
\curveto(0.195,1.105)(0.035,0.725)(0.0,0.5)
\curveto(-0.035,0.275)(0.01,-0.235)(0.09,-0.52)
\curveto(0.17,-0.805)(0.38,-1.23)(0.51,-1.37)
\curveto(0.64,-1.51)(0.89,-1.785)(1.01,-1.92)
\curveto(1.13,-2.055)(1.4,-2.26)(1.55,-2.33)
\curveto(1.7,-2.4)(2.08,-2.47)(2.31,-2.47)
\curveto(2.54,-2.47)(2.89,-2.395)(3.01,-2.32)
\curveto(3.13,-2.245)(3.31,-2.045)(3.37,-1.92)
\curveto(3.43,-1.795)(3.515,-1.555)(3.54,-1.44)
\curveto(3.565,-1.325)(3.555,-1.085)(3.52,-0.96)
\curveto(3.485,-0.835)(3.43,-0.64)(3.41,-0.57)
\curveto(3.39,-0.5)(3.36,-0.315)(3.35,-0.2)
\curveto(3.34,-0.085)(3.335,0.135)(3.34,0.24)
\curveto(3.345,0.345)(3.35,0.555)(3.35,0.66)
\curveto(3.35,0.765)(3.34,0.975)(3.33,1.08)
\curveto(3.32,1.185)(3.305,1.335)(3.3,1.38)
\curveto(3.295,1.425)(3.295,1.505)(3.3,1.54)
\curveto(3.305,1.575)(3.31,1.635)(3.31,1.66)
\curveto(3.31,1.685)(3.3,1.73)(3.29,1.75)
\curveto(3.28,1.77)(3.265,1.795)(3.27,1.83)
}

\pscustom[linewidth=0.02]
{
\newpath
\moveto(3.27,1.83)
\lineto(3.21,1.93)
\curveto(3.18,1.98)(3.055,2.09)(2.96,2.15)
\curveto(2.865,2.21)(2.635,2.305)(2.5,2.34)
\curveto(2.365,2.375)(2.07,2.385)(1.91,2.36)
\curveto(1.75,2.335)(1.455,2.235)(1.32,2.16)
\curveto(1.185,2.085)(0.93,1.9)(0.81,1.79)
\curveto(0.69,1.68)(0.445,1.415)(0.32,1.26)
\curveto(0.195,1.105)(0.035,0.725)(0.0,0.5)
\curveto(-0.035,0.275)(0.01,-0.235)(0.09,-0.52)
\curveto(0.17,-0.805)(0.38,-1.23)(0.51,-1.37)
\curveto(0.64,-1.51)(0.89,-1.785)(1.01,-1.92)
\curveto(1.13,-2.055)(1.4,-2.26)(1.55,-2.33)
\curveto(1.7,-2.4)(2.08,-2.47)(2.31,-2.47)
\curveto(2.54,-2.47)(2.89,-2.395)(3.01,-2.32)
\curveto(3.13,-2.245)(3.31,-2.045)(3.37,-1.92)
\curveto(3.43,-1.795)(3.515,-1.555)(3.54,-1.44)
\curveto(3.565,-1.325)(3.555,-1.085)(3.52,-0.96)
\curveto(3.485,-0.835)(3.43,-0.64)(3.41,-0.57)
\curveto(3.39,-0.5)(3.36,-0.315)(3.35,-0.2)
\curveto(3.34,-0.085)(3.335,0.135)(3.34,0.24)
\curveto(3.345,0.345)(3.35,0.555)(3.35,0.66)
\curveto(3.35,0.765)(3.34,0.975)(3.33,1.08)
\curveto(3.32,1.185)(3.305,1.335)(3.3,1.38)
\curveto(3.295,1.425)(3.295,1.505)(3.3,1.54)
\curveto(3.305,1.575)(3.31,1.635)(3.31,1.66)
\curveto(3.31,1.685)(3.3,1.73)(3.29,1.75)
\curveto(3.28,1.77)(3.265,1.795)(3.27,1.83)
}

\psline[linewidth=0.05,arrows=<-](3.69,-0.5)(3.69,0.4)

\psline[linewidth=0.05,arrows=->](3.09,-0.5)(3.09,0.4)

\psdots[dotsize=0.12](3.33,1.67)
\psdots[dotsize=0.12](3.57,-1.29)
\usefont{T1}{ptm}{m}{n}
\rput(3.7545313,1.6){$v$}
\usefont{T1}{ptm}{m}{n}
\rput(3.954531,-1.18){$u$}

\psline[linewidth=0.1,linestyle=none,arrows=->](-0.035,0.275)(0.09,-0.52)

\pscustom[linewidth=0.02]
{
\newpath
\moveto(3.30,1.69)
\lineto(3.39,1.79)
\curveto(3.42,1.84)(3.52,1.905)(3.59,1.92)
\curveto(3.66,1.935)(3.815,1.995)(3.9,2.04)
\curveto(3.985,2.085)(4.17,2.205)(4.27,2.28)
\curveto(4.37,2.355)(4.595,2.45)(4.72,2.47)
\curveto(4.845,2.49)(5.155,2.45)(5.34,2.39)
\curveto(5.525,2.33)(5.83,2.045)(5.95,1.82)
\curveto(6.07,1.595)(6.25,1.16)(6.31,0.95)
\curveto(6.37,0.74)(6.52,0.365)(6.61,0.2)
\curveto(6.7,0.035)(6.815,-0.34)(6.84,-0.55)
\curveto(6.865,-0.76)(6.765,-1.19)(6.64,-1.41)
\curveto(6.515,-1.63)(6.165,-1.925)(5.94,-2.0)
\curveto(5.715,-2.075)(5.345,-2.185)(5.2,-2.22)
\curveto(5.055,-2.255)(4.72,-2.28)(4.53,-2.27)
\curveto(4.34,-2.26)(4.065,-2.145)(3.98,-2.04)
\curveto(3.895,-1.935)(3.79,-1.74)(3.77,-1.65)
\curveto(3.75,-1.56)(3.705,-1.44)(3.68,-1.41)
\curveto(3.655,-1.38)(3.615,-1.33)(3.6,-1.31)
\curveto(3.585,-1.29)(3.57,-1.265)(3.57,-1.25)
}

\psline[linewidth=0.1,linestyle=none,arrows=<-](6.7,0.035)(6.84,-0.55)

\usefont{T1}{ptm}{m}{n}
\rput(0.42453125,2.14){$K$}
\usefont{T1}{ptm}{m}{n}
\rput(6.9045315,1.0){$L$}
\end{pspicture} 
}

\caption[An illustration of fixed-point index additivity]
{
\label{fig:lem indices add 1}
{\bf An illustration of fixed-point index additivity.}
}

\end{figure}
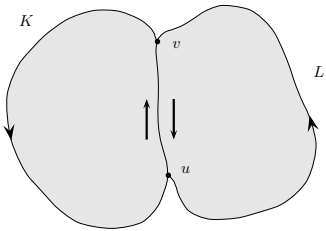

\begin{proof}
\label{ex:indices add}
The situation is as depicted in Figure \ref{fig:lem indices add 1}. We may consider $\eta(\phi)$ to be $1/2\pi$ times the change in argument of the vector $\phi(z) - z$, as $z$ traverses $\partial K$ once in the positive direction. Then as $z$ varies positively in $\partial K$ and in $\partial L$ the contributions to the sum $\eta(\phi) + \eta(\psi)$ along $\partial K \cap \partial L$ cancel.
\end{proof}\medskip

Next, we wish to make precise our notion of \emph{transverse position}:

\begin{definition}
\label{transverse position}
Two Jordan curves $\gamma$ and $\tilde\gamma$ are in \emph{transverse position} if for any point $z\in \gamma \cap \tilde\gamma$ where they meet, they \emph{cross transversely}, which here means that there is an open neighborhood $U\subset \bbC$ of $z$ and a homeomorphism $\phi : U \to \bbD$ between $U$ and the open unit disk $\bbD$, so that $\phi(\gamma \cap U) = \bbR \cap \bbD$ and $\phi(\tilde\gamma \cap U) = i \bbR \cap \bbD$. Two (open or closed) Jordan domains are said to be in \emph{transverse position} if their boundary Jordan curves are in transverse position.
\end{definition}

\noindent Note that if two Jordan curves are in transverse position, then they meet finitely many times, by a compactness argument.\medskip

Finally, it will be helpful to have available to us the terminology of the following definition:

\begin{definition}
Suppose that $X_1,\ldots,X_n$ and $ X'_1,\ldots, X'_n$ are all subsets of $\bbC$. Then we say that the collections $\{X_1,\ldots,X_n\}$ and $\{X'_1,\ldots, X'_n\}$ are in the same \emph{topological configuration} if there is an orientation-preserving homeomorphism $\varphi:\bbC \to \bbC$ so that $\varphi(X_i) =  X'_i$ for all $1\le i\le n$. In practice the collections of objects under consideration will not be labeled $X_i$ and $ X'_i$, but there will be some natural bijection between them. Then our requirement is that $\varphi$ respects this natural bijection. We say that certain conditions on some objects \emph{uniquely determine} their topological configuration if any two collections of objects satisfying the given conditions are in the same topological configuration.
\end{definition}

\noindent For example, considering a point $z\in \bbC$ and an open Jordan domain $\Omega$, the condition that $z\in \Omega$ uniquely determines the topological configuration of $\{z, \Omega\}$, but the condition that $z\not \in \Omega$ does not uniquely determine the topological configuration of $\{z, \Omega\}$. (We may have $z\in \partial \Omega$, or $z\in \bbC \setminus (\Omega \cup \partial \Omega)$, and these situations are topologically distinct.)

The following lemma says that when working with fixed-point index, we need to consider our Jordan domains only ``up to topological configuration.''

\begin{lemma}
\label{prop:index invariance}
Suppose $K$ and $\tilde K$ are closed Jordan domains. Let $f:\partial K \to \partial \tilde K$ be an indexable homeomorphism. Suppose that $K'$ and $\tilde K'$ are also closed Jordan domains, so that $\{K,\tilde K\}$ and $\{K',\tilde K'\}$ are in the same topological configuration, via the homeomorphism $\varphi: \bbC\to \bbC$. Let $f':\partial K'\to \partial \tilde K'$ be induced in the natural way, explicitly as $f' = \varphi|_{\partial \tilde K} \circ f \circ \varphi\inv|_{\partial K'}$. Then $f'$ is indexable with respect to the usual orientation on $\partial K'$ and $\partial \tilde K'$, and $\eta(f) = \eta(f')$.
\end{lemma}

\begin{proof}
The following is well-known. For a reference, see Chapters 1 and 2 of \cite{MR2850125}.
\begin{fact}
Every orientation-preserving homeomorphism $\bbC\to \bbC$ is homotopic to the identity map via homeomorphisms.
\end{fact}
\noindent Thus let $H_t:\bbC\times [0,1]\to \bbC$ be such a homotopy from the identity to $\varphi$. Explicitly, for fixed $t$ we have that $H_t$ is an orientation-preserving homeomorphism $\bbC\to \bbC$, with $H_0$ equal to the identity on $\bbC$ and $H_1 = \varphi$.

Let $K_t = H_t(K)$ and $\tilde K_t = H_t(\tilde K)$. Then $K_t$ and $\tilde K_t$ are closed Jordan domains, because $H_t$ is a homeomorphism. Let $f_t:\partial K_t\to \partial \tilde K_t$ be induced in the natural way, explicitly as $H_t|_{\partial \tilde K} \circ f \circ H_t\inv|_{\partial K_t}$. Let $\gamma_t = \{f_t(z) - z\}_{z\in \partial K_t}$. Then tautologically $\eta(f)$ is the winding number of $\gamma_0$ around the origin, and $\eta(f')$ is the winding number of $\gamma_1$ around the origin.

Every $\gamma_t$ is a closed curve because $\partial K_t$ is a closed curve and $f_t$ is continuous. Once we establish that no $\gamma_t$ passes through the origin, Lemma \ref{prop:index invariance} will be proved because we have an induced homotopy from $\gamma_0$ to $\gamma_1$, and two curves homotopic in $\bbC\setminus \{0\}$ have the same winding number around the origin. Suppose for contradiction that $0\in \gamma_t$. Then there is a $z\in \partial K_t$ so that $f_t(z) = z$. Thus $H_t\circ f\circ H_t\inv(z)= z$, and so $f(H_t\inv (z)) = H_t\inv (z)$, contradicting the fixed-point-free condition on $f$.
\end{proof}

\section{The Incompatibility Theorem}
\label{incompat section}

In this section, we state and prove the Incompatibility Theorem of Oded Schramm, appearing in \cite{MR1076089}*{Theorem 3.1}. Before doing so, we need some preliminary definitions:\medskip

First, a \emph{topological rectangle} is a closed Jordan domain $R$ with four marked points on its boundary $\partial R$, which we naturally call its \emph{corners}. A \emph{side} of $R$ is a closed sub-arc of $\partial R$ having two corners of $R$ as its endpoints, and containing no other corner of $R$. We abuse notation slightly and use the same symbol, in this case $R$, to refer both to a topological rectangle and to its constituent closed Jordan domain. We define \emph{topological triangles}, their \emph{corners}, and their \emph{sides} analogously, and employ the same abuse of notation.

A \emph{packing} of a topological rectangle $R$ consists of a finite collection $\{K_1,\allowbreak \ldots,\allowbreak K_n\}$ of closed Jordan domains so that the following hold:

\begin{itemize}
\item Every $K_i$ is contained in $R$.
\item The $K_i$ are pairwise interiorwise disjoint, and any two of them meet at at most one point.
\item Each of the $K_i$ meets $\partial R$ at at most one point, and no $K_i$ meets a corner of $R$.
\item For every connected component $U$ of $R\setminus \cup_{i=1}^n K_i$, we have that the closure of $U$ is a topological triangle $T$ each of whose sides is contained in one of the $\partial K_i$, or in a side of $R$. Then every corner of $T$ is either a corner of $R$, or an intersection point of some $\partial K_i$ either with some other $\partial K_j$ or with $\partial R$.
\end{itemize}

\noindent See Figure \ref{clown1} for two examples of packings of topological rectangles. Let $S_a,S_b,S_c,S_d$ denote the sides of $R$. The \emph{contact graph} of the packing of $R$ by $K_1,\ldots,K_n$ is the graph having vertices $v_1,\ldots,v_n$ corresponding to the $K_1,\ldots,K_n$ and $v_a,v_b,v_c,v_d$ corresponding to the sides of $R$, both in the natural way, so that two distinct vertices share an edge if and only if the corresponding sets meet. Note that for example, the contact graph of a packing of a rectangle is always a triangulation of a square, that is, a triangulation of a topological closed disk having four boundary edges.

%\noindent The \emph{contact graph} of such a collection $\{K_1,\ldots,K_n\}$ is the graph having a vertex $v_i$ for each $K_i$, so that distinct $v_i$ and $v_j$ share an edge if and only if $K_i$ and $K_j$ meet.

%Suppose that the collection $\{K_1,\ldots,K_n\}$ packs $R$. Then for every corner $z$ of $R$, there is exactly one of the $K_1,\ldots,K_n$, which we will denote $K_{(z)}$, which meets the closure of that connected component of $R\setminus \cup_{i=1}^n K_i$ which contains $z$. See Figure \ref{clown1} for two examples of packings of topological rectangles by closed Jordan domains. Note that in general, even if $z_1$ and $z_2$ are distinct corners of $R$, it may be that $K_{(z_1)} = K_{(z_2)}$.

\begin{figure}
\centering
% Generated with LaTeXDraw 2.0.8
% Wed Aug 15 16:47:15 EDT 2012
% \usepackage[usenames,dvipsnames]{pstricks}
% \usepackage{epsfig}
% \usepackage{pst-grad} % For gradients
% \usepackage{pst-plot} % For axes
\scalebox{1} % Change this value to rescale the drawing.
{
\begin{pspicture}(0,-2.52)(11.74,2.54)
\definecolor{color1589b}{rgb}{0.8901960784313725,0.8862745098039215,0.8862745098039215}
\psframe[linewidth=0.04,dimen=outer](5.54,1.22)(0.46,-1.36)
\psbezier[linewidth=0.02](1.7,-0.12)(1.4,-0.12)(0.45218492,0.28114378)(0.5,0.58)(0.5478151,0.8788562)(1.4001769,1.1611928)(1.8,1.18)(2.1998231,1.1988072)(3.0059357,0.9609888)(3.0,0.68)(2.9940643,0.3990112)(2.6,0.28)(2.5,0.18)(2.4,0.08)(2.0,-0.12)(1.7,-0.12)
\psbezier[linewidth=0.02](2.5,0.18)(2.7,0.28)(3.123653,0.2064788)(3.3,0.08)(3.476347,-0.0464788)(3.2,-0.42)(2.9,-0.42)(2.6,-0.42)(2.3,0.08)(2.5,0.18)
\psbezier[linewidth=0.02](3.3,0.08)(3.2,0.18)(3.0,0.38)(3.0,0.68)(3.0,0.98)(3.4026165,1.1522919)(4.1,1.18)(4.7973833,1.2077081)(5.5714183,0.5651983)(5.5,0.28)(5.4285817,-0.0051983013)(4.698314,0.03803818)(4.2,-0.02)(3.7016857,-0.07803818)(3.4,-0.02)(3.3,0.08)
\psbezier[linewidth=0.02](4.2,-0.02)(4.8,-0.02)(5.5,-0.02)(5.5,-0.52)(5.5,-1.02)(4.0,-1.32)(3.0,-1.32)(2.0,-1.32)(0.5,-1.22)(0.5,-0.52)(0.5,0.18)(0.9,-0.12)(1.7,-0.12)(2.5,-0.12)(1.9003581,-0.44675693)(2.9,-0.42)(3.899642,-0.39324304)(3.6,-0.02)(4.2,-0.02)
\psframe[linewidth=0.04,dimen=outer](10.62,2.4)(6.54,-2.48)
\psbezier[linewidth=0.02,fillcolor=color1589b](6.58,1.16)(6.6309495,1.3284905)(6.68,1.76)(6.88,1.96)(7.08,2.16)(7.38,2.36)(7.78,2.36)(8.18,2.36)(8.58,2.06)(8.58,1.66)(8.58,1.26)(8.375797,0.76556337)(8.18,0.66)(7.984203,0.5544366)(7.48,0.56)(7.28,0.56)(7.08,0.56)(6.5290504,0.9915095)(6.58,1.16)
\psbezier[linewidth=0.02,fillcolor=color1589b](8.88,0.86)(8.76,0.94)(8.380138,0.82594526)(8.18,0.66)(7.9798613,0.49405473)(8.18,0.36)(8.28,0.26)(8.38,0.16)(8.48,-0.14)(8.68,-0.14)(8.88,-0.14)(8.86,0.12)(8.98,0.26)(9.1,0.4)(9.1,0.36)(9.16,0.5)(9.22,0.64)(9.0,0.78)(8.88,0.86)
\psbezier[linewidth=0.02,fillcolor=color1589b](9.18,2.36)(8.98,2.36)(8.602183,1.9525076)(8.58,1.66)(8.5578165,1.3674924)(8.68,1.46)(8.78,1.26)(8.88,1.06)(8.754272,0.94004184)(8.88,0.86)(9.005728,0.7799582)(9.15612,0.70268345)(9.28,0.66)(9.403879,0.61731654)(9.34,0.54)(9.58,0.56)(9.82,0.58)(9.98,0.66)(10.18,0.76)(10.38,0.86)(10.48,0.96)(10.58,1.16)(10.68,1.36)(10.25945,1.8262779)(10.18,1.96)(10.100551,2.093722)(9.812722,2.1994033)(9.68,2.26)(9.547278,2.3205967)(9.38,2.36)(9.18,2.36)
\psbezier[linewidth=0.02,fillcolor=color1589b](6.58,-1.04)(6.574384,-0.74032813)(7.0397286,0.6195747)(7.28,0.56)(7.520272,0.5004253)(8.18,-0.14)(8.68,-0.14)(9.18,-0.14)(9.352913,0.5601822)(9.58,0.56)(9.807088,0.55981773)(10.600395,-0.640208)(10.58,-0.94)(10.559605,-1.239792)(9.999562,-1.8670098)(9.78,-2.04)(9.560438,-2.2129903)(8.98,-2.44)(8.38,-2.44)(7.78,-2.44)(7.24302,-2.2863748)(7.08,-2.04)(6.91698,-1.7936251)(6.5856156,-1.3396719)(6.58,-1.04)
\usefont{T1}{ptm}{m}{n}
\rput(5.06,-1.615){$R$}
\usefont{T1}{ptm}{m}{n}
\rput(10.98,-2.095){$\tilde R$}
\usefont{T1}{ptm}{m}{n}
%\rput(0.26,1.425){$z$}
\usefont{T1}{ptm}{m}{n}
%\rput(6.3,2.345){$\tilde z$}
\usefont{T1}{ptm}{m}{n}
%\rput(1.76,0.545){$K_{(z)}$}
\usefont{T1}{ptm}{m}{n}
%\rput(7.58,1.465){$\tilde K_{(\tilde z)}$}
%\psdots[dotsize=0.12](6.58,2.36)
%\psdots[dotsize=0.12](0.48,1.2)
\end{pspicture} 
}
\caption[Two topological rectangles packed with shapes]
{
\label{clown1}
{\bf Two topological rectangles packed with shapes.}
}
\end{figure}
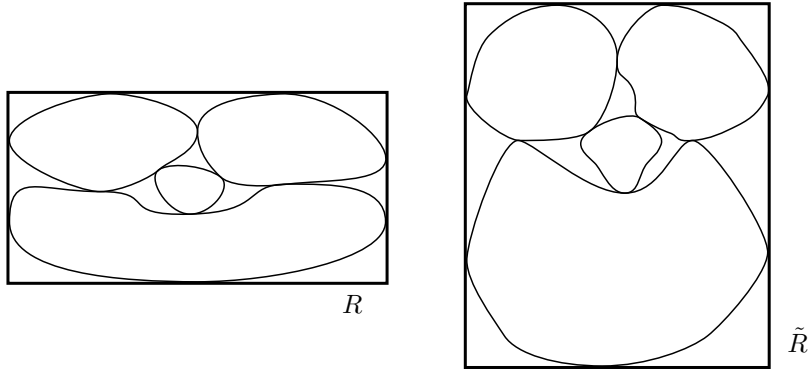

Next, suppose that we have packings of topological rectangles $R$ and $\tilde R$ by collections of closed Jordan domains $K_1,\ldots,K_n$ and $\tilde K_1,\ldots,\tilde K_n$, respectively. The two packings are said to be in \emph{transverse position} if:

\begin{itemize}
\item For every pair of integers $i$ and $\tilde i$, both between $1$ and $n$, with possibly $i=\tilde i$, we have that $K_i$ and $\tilde K_{\tilde i}$ are in transverse position as closed Jordan domains.
\item For every $1\le i\le n$, we have that $K_i$ and $\tilde R$ are in transverse position as closed Jordan domains, as are $\tilde K_i$ and $R$.
\item No intersection point of a pair of the sets $\partial K_1,\ldots,\partial K_n,\partial R$ lies on any $\partial \tilde K_i$ nor on $\partial \tilde R$. Similarly, no intersection point of a pair of the sets $\partial \tilde K_1,\ldots,\partial \tilde K_n,\partial \tilde R$ lies on any $\partial K_i$ nor on $\partial R$.
\end{itemize}

\noindent For two examples of packings in transverse position, see Figure \ref{clown2}.\medskip

Finally, two closed Jordan domains $K$ and $\tilde K$ which are in transverse position are said to \emph{cut} one another if at least one of the sets $K\setminus \tilde K$ and $\tilde K\setminus K$ is disconnected. Though it is not important for us, it holds for such $K$ and $\tilde K$ that $K\setminus \tilde K$ is connected if and only if $\tilde K \setminus K$ is.\medskip

We are now ready to state our version of the Incompatibility Theorem originally due to Schramm. The original appears in \cite{MR1076089}*{Theorem 3.1}, where it is used to give the first proof of the rigidity of circle packings filling the complex plane.

When processing our statement of the Incompatibility Theorem \ref{incompat}, it may be helpful to keep the following in mind: in Schramm's original formulation, our notion of shapes \emph{cutting one another} is referred to as \emph{incompatibility}. Then the Incompatibility Theorem \ref{incompat} may be remembered as, ``if two incompatible rectangles are packed in the same combinatorial way, then at least one pair of corresponding shapes of the two packings will be incompatible.''  The precise statement given in \cite{MR1076089} is somewhat different from the one we give here. Our statement communicates the main idea of the theorem, and suffices for our applications. Our proof uses fixed-point index, with the main new tool being Lemma \ref{no cut index}. The original proof by Schramm is via different methods, but is also elementary.

\begin{figure}
\centering
% Generated with LaTeXDraw 2.0.8
% Tue Jan 15 15:28:55 CST 2013
% \usepackage[usenames,dvipsnames]{pstricks}
% \usepackage{epsfig}
% \usepackage{pst-grad} % For gradients
% \usepackage{pst-plot} % For axes
\scalebox{1} % Change this value to rescale the drawing.
{
\begin{pspicture}(0,-1.7292187)(6.6690626,1.7292187)
\psframe[linewidth=0.04,dimen=middle](5.81,0.5757812)(0.81,-0.62421876)
\psframe[linewidth=0.04,linestyle=dashed,dash=0.16cm 0.16cm,dimen=middle](5.21,1.1757812)(1.41,-1.2242187)
\psdots[dotsize=0.12](1.41,1.1757812)
\psdots[dotsize=0.12](5.21,1.1757812)
\psdots[dotsize=0.12](5.81,0.5757812)
\psdots[dotsize=0.12](5.81,-0.62421876)
\psdots[dotsize=0.12](0.81,-0.62421876)
\psdots[dotsize=0.12](0.81,0.5757812)
\psdots[dotsize=0.12](1.41,-1.2242187)
\psdots[dotsize=0.12](5.21,-1.2242187)
\usefont{T1}{ppl}{m}{n}
\rput(6.1045313,-0.03421875){$S_a$}
\usefont{T1}{ppl}{m}{n}
\rput(2.8545313,0.30578125){$S_b$}
\usefont{T1}{ppl}{m}{n}
\rput(0.45453125,0.02578125){$S_c$}
\usefont{T1}{ppl}{m}{n}
\rput(3.7145312,-0.37421876){$S_d$}
\usefont{T1}{ppl}{m}{n}
\rput(4.8545313,0.00578125){$\tilde S_a$}
\usefont{T1}{ppl}{m}{n}
\rput(3.2845314,1.4657812){$\tilde S_b$}
\usefont{T1}{ppl}{m}{n}
\rput(1.7645313,-0.03421875){$\tilde S_c$}
\usefont{T1}{ppl}{m}{n}
\rput(3.2045312,-1.4942187){$\tilde S_d$}
\usefont{T1}{ppl}{m}{n}
\rput(6.054531,0.74578124){$R$}
\usefont{T1}{ppl}{m}{n}
\rput(1.2645313,1.5257813){$\tilde R$}
\end{pspicture} 
}
\caption
{
\label{topo conf incompat rect}
}
\end{figure}

\begin{incompat}
\label{incompat}
Let $R$ and $\tilde R$ be topological rectangles having sides $S_a,\allowbreak S_b,\allowbreak S_c,\allowbreak S_d$ and $\tilde S_a,\allowbreak \tilde S_b,\allowbreak \tilde S_c,\allowbreak \tilde S_d$ respectively, in the topological configuration depicted in Figure \ref{topo conf incompat rect}. Suppose that we are given packings of $R$ and $\tilde R$ by collections of closed Jordan domains $K_1,\ldots,K_n$ and $\tilde K_1,\ldots,\tilde K_n$ respectively, in transverse position, so that the packings are combinatorially equivalent in the following precise sense: denoting the contact graphs of the packings by $G$ and $\tilde G$ respectively, we insist that the following holds: letting $v_a,v_b,v_c,v_d,v_1,\ldots,v_n$ and $\tilde v_a,\tilde v_b, \tilde v_c, \tilde v_d, \tilde v_1, \ldots, \tilde v_n$ denote the vertex sets of $G$ and $\tilde G$ in the natural way, we have that $G$ and $\tilde G$ are isomorphic via the identification $v_i \mapsto \tilde v_i$ for $i=a,b,c,d,1,\ldots,n$. Then, there is an $1\le i\le n$ so that $K_i$ and $\tilde K_i$ cut each other.
\end{incompat}

\noindent For example, the packings shown in Figure \ref{clown1} share a contact graph. In Figure \ref{clown2} we have overlaid them in two different ways, in both cases ensuring that the hypotheses of Theorem \ref{incompat} are satisfied. We see that in each case, a pair of corresponding closed Jordan domains cut one another.\medskip

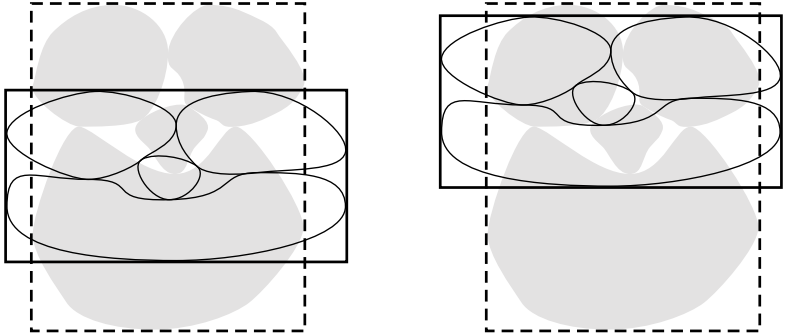
\begin{figure}
\centering
\subfloat
{
% Generated with LaTeXDraw 2.0.8
% Wed Aug 15 16:30:56 EDT 2012
% \usepackage[usenames,dvipsnames]{pstricks}
% \usepackage{epsfig}
% \usepackage{pst-grad} % For gradients
% \usepackage{pst-plot} % For axes
\scalebox{0.9} % Change this value to rescale the drawing.
{
\begin{pspicture}(0,-2.47)(5.1414185,2.47)
\definecolor{color721b}{rgb}{0.8901960784313725,0.8862745098039215,0.8862745098039215}
\psbezier[linewidth=0.02,fillstyle=solid,fillcolor=color721b,linestyle=none](0.44,1.2)(0.49094954,1.3684905)(0.54,1.8)(0.74,2.0)(0.94,2.2)(1.24,2.4)(1.64,2.4)(2.04,2.4)(2.44,2.1)(2.44,1.7)(2.44,1.3)(2.2357972,0.80556333)(2.04,0.7)(1.8442029,0.59443665)(1.34,0.6)(1.14,0.6)(0.94,0.6)(0.38905045,1.0315095)(0.44,1.2)
\psbezier[linewidth=0.02,fillstyle=solid,fillcolor=color721b,linestyle=none](2.74,0.9)(2.62,0.98)(2.2401388,0.8659453)(2.04,0.7)(1.8398613,0.5340547)(2.04,0.4)(2.14,0.3)(2.24,0.2)(2.34,-0.1)(2.54,-0.1)(2.74,-0.1)(2.72,0.16)(2.84,0.3)(2.96,0.44)(2.96,0.4)(3.02,0.54)(3.08,0.68)(2.86,0.82)(2.74,0.9)
\psbezier[linewidth=0.02,fillstyle=solid,fillcolor=color721b,linestyle=none](3.04,2.4)(2.84,2.4)(2.4621832,1.9925076)(2.44,1.7)(2.4178166,1.4074924)(2.54,1.5)(2.64,1.3)(2.74,1.1)(2.6142726,0.9800418)(2.74,0.9)(2.8657274,0.8199582)(3.0161204,0.7426834)(3.14,0.7)(3.2638795,0.65731657)(3.2,0.58)(3.44,0.6)(3.68,0.62)(3.84,0.7)(4.04,0.8)(4.24,0.9)(4.34,1.0)(4.44,1.2)(4.54,1.4)(4.1194496,1.8662778)(4.04,2.0)(3.9605503,2.1337223)(3.6727219,2.2394032)(3.54,2.3)(3.407278,2.3605967)(3.24,2.4)(3.04,2.4)
\psbezier[linewidth=0.02,fillstyle=solid,fillcolor=color721b,linestyle=none](0.44,-1.0)(0.43438423,-0.7003281)(0.8997284,0.65957475)(1.14,0.6)(1.3802716,0.54042524)(2.04,-0.1)(2.54,-0.1)(3.04,-0.1)(3.2129126,0.60018224)(3.44,0.6)(3.6670876,0.59981775)(4.4603953,-0.600208)(4.44,-0.9)(4.419605,-1.199792)(3.8595622,-1.8270097)(3.64,-2.0)(3.4204378,-2.1729903)(2.84,-2.4)(2.24,-2.4)(1.64,-2.4)(1.10302,-2.2463748)(0.94,-2.0)(0.77698004,-1.7536252)(0.44561577,-1.2996719)(0.44,-1.0)
\psframe[linewidth=0.04,dimen=inner](5.06,1.12)(0.06,-1.38)
\psbezier[linewidth=0.02](1.26,-0.18)(0.96,-0.18)(0.012184902,0.2211438)(0.06,0.52)(0.107815094,0.8188562)(0.9601769,1.1011928)(1.36,1.12)(1.7598231,1.1388072)(2.5659356,0.9009888)(2.56,0.62)(2.5540643,0.3390112)(2.16,0.22)(2.06,0.12)(1.96,0.02)(1.56,-0.18)(1.26,-0.18)
\psbezier[linewidth=0.02](2.06,0.12)(2.26,0.22)(2.683653,0.1464788)(2.86,0.02)(3.036347,-0.1064788)(2.76,-0.48)(2.46,-0.48)(2.16,-0.48)(1.86,0.02)(2.06,0.12)
\psbezier[linewidth=0.02](2.86,0.02)(2.76,0.12)(2.56,0.32)(2.56,0.62)(2.56,0.92)(2.9626164,1.0922918)(3.66,1.12)(4.3573837,1.1477082)(5.1314187,0.5051983)(5.06,0.22)(4.9885817,-0.0651983)(4.258314,-0.02196182)(3.76,-0.08)(3.2616856,-0.13803819)(2.96,-0.08)(2.86,0.02)
\psbezier[linewidth=0.02](3.76,-0.08)(4.36,-0.08)(5.06,-0.08)(5.06,-0.58)(5.06,-1.08)(3.56,-1.38)(2.56,-1.38)(1.56,-1.38)(0.06,-1.28)(0.06,-0.58)(0.06,0.12)(0.46,-0.18)(1.26,-0.18)(2.06,-0.18)(1.460358,-0.50675696)(2.46,-0.48)(3.459642,-0.45324305)(3.16,-0.08)(3.76,-0.08)
\psframe[linewidth=0.04,dimen=inner,linestyle=dashed](4.44,2.4)(0.44,-2.4)
\end{pspicture} 
}
}
\qquad
\subfloat
{
% Generated with LaTeXDraw 2.0.8
% Wed Aug 15 16:31:35 EDT 2012
% \usepackage[usenames,dvipsnames]{pstricks}
% \usepackage{epsfig}
% \usepackage{pst-grad} % For gradients
% \usepackage{pst-plot} % For axes
\scalebox{0.9} % Change this value to rescale the drawing.
{
\begin{pspicture}(0,-2.47)(5.1414185,2.47)
\definecolor{color721b}{rgb}{0.8901960784313725,0.8862745098039215,0.8862745098039215}
\psbezier[linewidth=0.02,fillstyle=solid,fillcolor=color721b,linestyle=none](0.74,1.2)(0.7909495,1.3684905)(0.84,1.8)(1.04,2.0)(1.24,2.2)(1.54,2.4)(1.94,2.4)(2.34,2.4)(2.74,2.1)(2.74,1.7)(2.74,1.3)(2.535797,0.80556333)(2.34,0.7)(2.144203,0.59443665)(1.64,0.6)(1.44,0.6)(1.24,0.6)(0.68905044,1.0315095)(0.74,1.2)
\psbezier[linewidth=0.02,fillstyle=solid,fillcolor=color721b,linestyle=none](3.04,0.9)(2.92,0.98)(2.5401387,0.8659453)(2.34,0.7)(2.1398613,0.5340547)(2.34,0.4)(2.44,0.3)(2.54,0.2)(2.64,-0.1)(2.84,-0.1)(3.04,-0.1)(3.02,0.16)(3.14,0.3)(3.26,0.44)(3.26,0.4)(3.32,0.54)(3.38,0.68)(3.16,0.82)(3.04,0.9)
\psbezier[linewidth=0.02,fillstyle=solid,fillcolor=color721b,linestyle=none](3.34,2.4)(3.14,2.4)(2.7621832,1.9925076)(2.74,1.7)(2.7178168,1.4074924)(2.84,1.5)(2.94,1.3)(3.04,1.1)(2.9142725,0.9800418)(3.04,0.9)(3.1657274,0.8199582)(3.3161204,0.7426834)(3.44,0.7)(3.5638795,0.65731657)(3.5,0.58)(3.74,0.6)(3.98,0.62)(4.14,0.7)(4.34,0.8)(4.54,0.9)(4.64,1.0)(4.74,1.2)(4.84,1.4)(4.41945,1.8662778)(4.34,2.0)(4.2605505,2.1337223)(3.9727218,2.2394032)(3.84,2.3)(3.7072783,2.3605967)(3.54,2.4)(3.34,2.4)
\psbezier[linewidth=0.02,fillstyle=solid,fillcolor=color721b,linestyle=none](0.74,-1.0)(0.73438424,-0.7003281)(1.1997284,0.65957475)(1.44,0.6)(1.6802716,0.54042524)(2.34,-0.1)(2.84,-0.1)(3.34,-0.1)(3.5129125,0.60018224)(3.74,0.6)(3.9670875,0.59981775)(4.7603955,-0.600208)(4.74,-0.9)(4.7196045,-1.199792)(4.159562,-1.8270097)(3.94,-2.0)(3.7204378,-2.1729903)(3.14,-2.4)(2.54,-2.4)(1.94,-2.4)(1.40302,-2.2463748)(1.24,-2.0)(1.07698,-1.7536252)(0.7456158,-1.2996719)(0.74,-1.0)
\psframe[linewidth=0.04,dimen=inner](5.06,2.22)(0.06,-0.28)
\psbezier[linewidth=0.02](1.26,0.92)(0.96,0.92)(0.012184902,1.3211437)(0.06,1.62)(0.107815094,1.9188563)(0.9601769,2.2011929)(1.36,2.22)(1.7598231,2.2388072)(2.5659356,2.0009887)(2.56,1.72)(2.5540643,1.4390112)(2.16,1.32)(2.06,1.22)(1.96,1.12)(1.56,0.92)(1.26,0.92)
\psbezier[linewidth=0.02](2.06,1.22)(2.26,1.32)(2.683653,1.2464788)(2.86,1.12)(3.036347,0.9935212)(2.76,0.62)(2.46,0.62)(2.16,0.62)(1.86,1.12)(2.06,1.22)
\psbezier[linewidth=0.02](2.86,1.12)(2.76,1.22)(2.56,1.42)(2.56,1.72)(2.56,2.02)(2.9626164,2.192292)(3.66,2.22)(4.3573837,2.247708)(5.1314187,1.6051983)(5.06,1.32)(4.9885817,1.0348017)(4.258314,1.0780382)(3.76,1.02)(3.2616856,0.9619618)(2.96,1.02)(2.86,1.12)
\psbezier[linewidth=0.02](3.76,1.02)(4.36,1.02)(5.06,1.02)(5.06,0.52)(5.06,0.02)(3.56,-0.28)(2.56,-0.28)(1.56,-0.28)(0.06,-0.18)(0.06,0.52)(0.06,1.22)(0.46,0.92)(1.26,0.92)(2.06,0.92)(1.460358,0.59324306)(2.46,0.62)(3.459642,0.64675695)(3.16,1.02)(3.76,1.02)
\psframe[linewidth=0.04,dimen=inner,linestyle=dashed](4.74,2.4)(0.74,-2.4)
\end{pspicture}
}
}
\caption[Shapes cutting each other]
{
\label{clown2}
{\bf Shapes cutting each other.}  We have drawn $R$ and $\tilde R$ on top of each other in two different ``incompatible'' ways, and in both cases some pair of corresponding shapes are ``incompatible.''
}
\end{figure}

The rest of this section consists of a proof of the Incompatibility Theorem \ref{incompat}. Our proof relies on the following simple lemma, which appeared in the introduction but is restated here for the convenience of the reader:

\renewcommand{\thelemmafree}{\ref{no cut index}}
\begin{lemmafree}
Let $K$ and $\tilde K$ be closed Jordan domains in transverse position, which do not cut each other, having boundaries oriented positively. Let $\phi : \partial K \to \partial \tilde K$ be an indexable homeomorphism. Then $\eta(\phi) \ge 0$.
\end{lemmafree}

\noindent We now give the proof of Theorem \ref{incompat} assuming Lemma \ref{no cut index}:

First, note that there is a natural bijection between the $U_f$ and the $\tilde U_f$, where we write $\{U_f\}_{f\in F}$ to denote the connected components of $R \setminus \partial R \cup \bigcup_{i=1}^n K_i$, similarly $\{\tilde U_f\}_{f\in F}$. Moreover, for fixed $f$, we have that $U_f$ and $\tilde U_f$ are topological triangles, and that the corners of $U_f$ correspond in a natural way to those of its partner $\tilde U_f$. (To see this, one may consider the graphs $G$ and $\tilde G$. Each is the 1-skeleton of a triangulation of a topological closed disk, and because the graphs are isomorphic, we get that the triangulations are combinatorially equivalent. Then the $U_f$ are in natural bijection with the faces $F$ of this combinatorial triangulation, as are the $\tilde U_f$.)

Given such a pair $U_f$ and $\tilde U_f$, we have that $U_f$ and $\tilde U_f$ are in transverse position as closed Jordan domains, because the packings we began with are in transverse position. For every such pair $U_f$ and $\tilde U_f$, orient $\partial U_f$ and $\tilde\partial U_f$ positively, and let $\phi_f : \partial U_f \to \partial \tilde U_f$ be an indexable homeomorphism identifying corresponding corners, satisfying $\eta(\phi_f) \ge 0$. We may do so by the Three Point Prescription Lemma \ref{tppl}.

Next, for every $1\le i\le n$, orienting $\partial K_i$ and $\partial \tilde K_i$ positively, we obtain an indexable homeomorphism $\phi_i : \partial K_i \to \partial \tilde K_i$ by restriction to the $\phi_f$ as necessary. Also, via the same procedure, orienting $\partial R$ and $\partial \tilde R$ positively, we obtain an indexable homeomorphism $\phi_R : \partial R \to \partial \tilde R$. Then, by the Index Additivity Lemma \ref{ial}, we have:

\[
\eta(\phi_R) = \sum_{i=1}^n \eta(\phi_i) + \sum_{f\in F} \eta(\phi_f)
\]

\noindent Now, as we saw in Figure \ref{fig:ex neg winding forced}, we have that $\eta(\phi_R) = -1$. On the other hand, every $\eta(\phi_f)$ is non-negative by construction, and the $\eta(\phi_i)$ are non-negative by Lemma \ref{no cut index}, which gives us a contradiction.\medskip

To complete the proof of the Incompatibility Theorem \ref{incompat}, we require the proofs of several lemmas. The proof of the Three Point Prescription Lemma \ref{tppl} is given in Section \ref{sec:3p} using torus parametrization. Our proof of Lemma \ref{no cut index} is inspired\footnote{Thanks to Mario Bonk for suggesting this line of proof, greatly simplifying the required arguments.} by an argument given by Schramm in \cite{MR1207210}*{Proof of Lemma 2.2} in support of the Circle Index Lemma \ref{cil}. We first need to prove a topological lemma on Jordan domains which do not cut each other:

\begin{lemma}
\label{uniqa}
Suppose $K$ and $\tilde K$ are closed Jordan domains in transverse position which do not cut each other, whose boundaries meet at least twice. Then the topological configuration of $\{K,\tilde K\}$ is determined by how many times $\partial K$ and $\partial \tilde K$ meet.
\end{lemma}

\begin{proof}
First, suppose without loss of generality that $K$ is the closed unit disk $\bar\bbD$. Let $2m\ge 2$ be the number of intersection points of $\partial K$ with $\partial \tilde K$. Note also that we may without loss of generality pick these arbitrarily along $\partial K = \partial \bbD$. Label these $z_1,\ldots,z_{2m}$ in clockwise order around $\partial K = \partial \bbD$. This brings us to the situation of Figure \ref{jon1}.

Orient $\partial K$ and $\partial \tilde K$ as usual. We now follow what happens as we traverse $\partial \tilde K$. By relabeling we may suppose that $\partial \tilde K$ enters $K$ at $z_1$. Because the interior of $\tilde K$ stays to the left of $\partial \tilde K$, and because $\partial \tilde K$ crosses $\partial K$ at every point where the two curves meet, it follows that $\partial \tilde K$ exits $K$ at $z_2$. The same reasoning allows us to conclude that $\partial \tilde K$ enters $K$ at $z_3$, etc., so we get that $\partial \tilde K$ enters $K$ at $z_i$ for all odd $1\le i\le 2m$, and exits $K$ at $z_i$ for all even $1\le i\le 2m$. This brings us to the situation of Figure \ref{jon2}.

We now consider where $\partial \tilde K$ goes after it crosses $z_1$. Denote by $z_i$ the point of $\partial K \cap \partial \tilde K$ at which it arrives immediately after crossing $z_1$, noting that then $i$ is even. We wish to establish that then $i=2$, so suppose for contradiction that $i\ne 2$. This brings us to the situation of Figure \ref{jon3}. Then $[z_1\to z_i]_{\partial \tilde K}$ disconnects $K = \bar\bbD$ into two components, call them $A_1$ and $A_2$. We have that every connected component of $K \setminus \tilde K$ must then be completely contained in one of $A_1$ and $A_2$. But $K\setminus \tilde K$ is connected by hypothesis, so one of $A_1$ and $A_2$ must be disjoint from $K\setminus \tilde K$. We then get a contradiction, because, keeping careful track of the orientation of $\partial \tilde K$, we see that there are points of $K \setminus \tilde K$ immediately counterclockwise from $z_1$ along $\partial K$, and points of $K \setminus \tilde K$ immediately clockwise from $z_2$ along $\partial K$, and these lie in different components of $K \setminus [z_1\to z_i]_{\partial \tilde K}$ unless $i=2$. The same reasoning allows us to conclude that, for any odd $1\le j\le 2m$, after entering $K$ at $z_j$, the curve $\partial \tilde K$ exits $K$ at $z_{j+1}$, bringing us to the situation of Figure \ref{jon4}.

By the same reasoning as in the previous paragraph, we get that, for any even $1\le j\le 2m$, after exiting $K$ at $z_j$, the curve $\partial \tilde K$ enters $K$ at $z_{j+1}$, adopting the convention that $z_{2m + 1} = z_1$. However, in this case, there are two ways to connect $z_j$ to $z_{j+1}$: roughly speaking, we may either travel clockwise around the complement of $K=\bar\bbD$ from $z_j$ to $z_{j+1}$, or counterclockwise. If we connect every such pair $z_j$ and $z_{j+1}$ with ``clockwise'' arcs, then the resulting orientation on $\partial \tilde K$ is not positive with respect to the Jordan domain it bounds, see Figure \ref{jon5}. Thus suppose without loss of generality, by relabeling if necessary, that $z_{2m}$ and $z_1$ are connected with a ``counterclockwise'' arc. Now for the remaining pairs $z_j,z_{j+1}$, with $j$ even, there are no choices (up to topological equivalence) about how to draw the connecting sub-arcs of $\partial \tilde K$ between them, and we arrive at the situation of Figure \ref{jon6}.
\end{proof}

\begin{remark}
Which \emph{a priori} topological configurations can occur for two Jordan curves in transverse position is a poorly understood question, and is known as the study of \emph{meanders}. We are fortunate that our setting is nice enough that a statement like that of Lemma \ref{uniqa} is possible. Thanks to Thomas Lam for informing us of the topic of meander theory.
\end{remark}

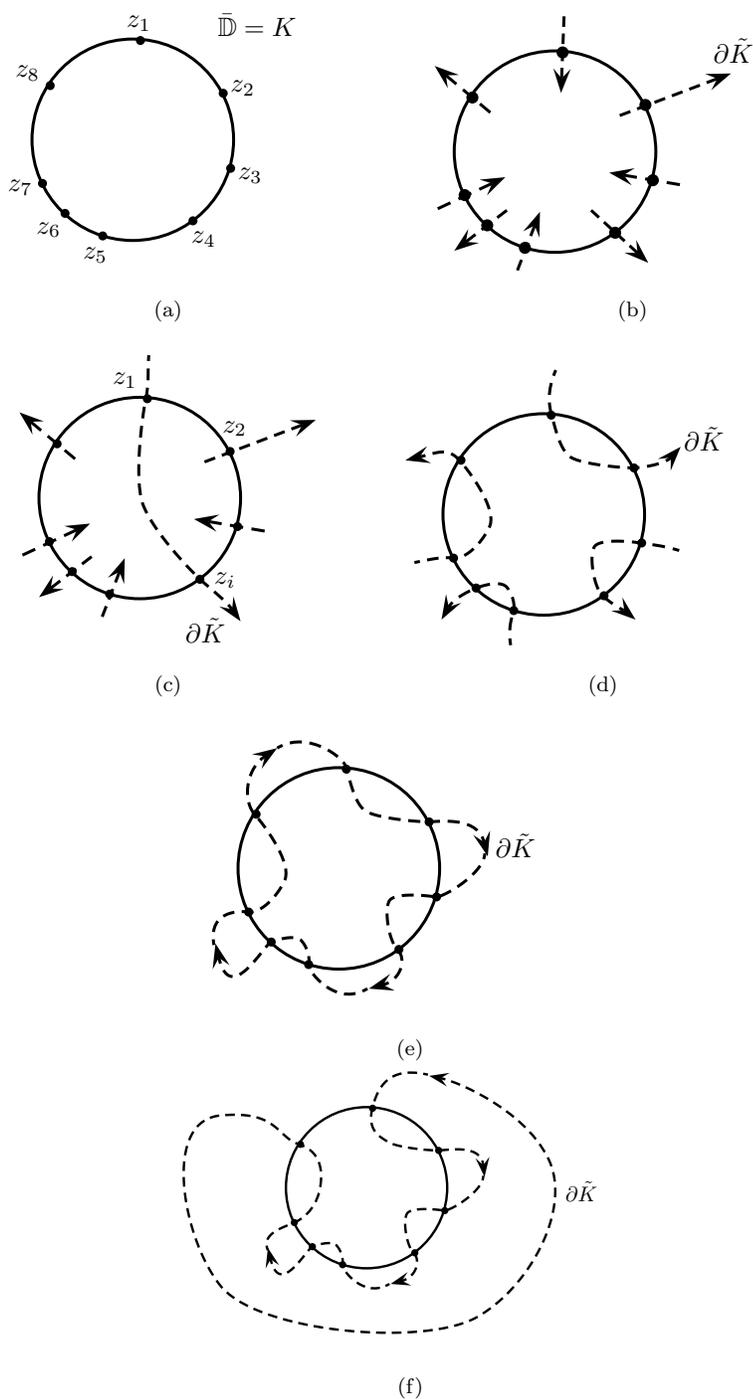
\begin{figure}
\centering
\subfloat[\label{jon1}]
{
% Generated with LaTeXDraw 2.0.8
% Sun Mar 24 14:08:20 PDT 2013
% \usepackage[usenames,dvipsnames]{pstricks}
% \usepackage{epsfig}
% \usepackage{pst-grad} % For gradients
% \usepackage{pst-plot} % For axes
\scalebox{1} % Change this value to rescale the drawing.
{
\begin{pspicture}(0,-1.748125)(4.7428126,1.748125)
\pscircle[linewidth=0.04,dimen=outer](1.9009376,0.0096875){1.36}
\usefont{T1}{ptm}{m}{n}
\rput(3.5523438,1.5346875){$\bar\bbD = K$}
\psdots[dotsize=0.12](2.0009375,1.3296875)
\psdots[dotsize=0.12](3.1009376,0.6296875)
\psdots[dotsize=0.12](3.2009375,-0.3703125)
\psdots[dotsize=0.12](2.7009375,-1.0703125)
\psdots[dotsize=0.12](1.5009375,-1.2703125)
\psdots[dotsize=0.12](1.0009375,-0.9703125)
\psdots[dotsize=0.12](0.7009375,-0.5703125)
\psdots[dotsize=0.12](0.8009375,0.7296875)
\usefont{T1}{ptm}{m}{n}
\rput(1.9823438,1.5546875){$z_1$}
\usefont{T1}{ptm}{m}{n}
\rput(3.3623438,0.6946875){$z_2$}
\usefont{T1}{ptm}{m}{n}
\rput(3.4623437,-0.4653125){$z_3$}
\usefont{T1}{ptm}{m}{n}
\rput(2.8423438,-1.3053125){$z_4$}
\usefont{T1}{ptm}{m}{n}
\rput(1.4023438,-1.5253125){$z_5$}
\usefont{T1}{ptm}{m}{n}
\rput(0.78234375,-1.1853125){$z_6$}
\usefont{T1}{ptm}{m}{n}
\rput(0.40234375,-0.6253125){$z_7$}
\usefont{T1}{ptm}{m}{n}
\rput(0.5023438,0.8946875){$z_8$}
\end{pspicture} 
}
}\qquad
\subfloat[\label{jon2}]
{
% Generated with LaTeXDraw 2.0.8
% Sun Mar 24 14:10:52 PDT 2013
% \usepackage[usenames,dvipsnames]{pstricks}
% \usepackage{epsfig}
% \usepackage{pst-grad} % For gradients
% \usepackage{pst-plot} % For axes
\scalebox{1} % Change this value to rescale the drawing.
{
\begin{pspicture}(0,-1.71)(5.261875,1.71)
\pscircle[linewidth=0.04,dimen=outer](1.6,-0.11){1.36}
\psdots[dotsize=0.16](1.7,1.21)
\psdots[dotsize=0.16](2.8,0.51)
\psdots[dotsize=0.16](2.9,-0.49)
\psdots[dotsize=0.16](2.4,-1.19)
\psdots[dotsize=0.16](1.2,-1.39)
\psdots[dotsize=0.16](0.7,-1.09)
\psdots[dotsize=0.16](0.4,-0.69)
\psdots[dotsize=0.16](0.5,0.61)
\psline[linewidth=0.04cm,linestyle=dashed,dash=0.16cm 0.1cm,arrowsize=0.1523123cm 2.0,arrowlength=1.4,arrowinset=0.4]{->}(1.72,1.69)(1.68,0.67)
\psline[linewidth=0.04cm,linestyle=dashed,dash=0.16cm 0.1cm,arrowsize=0.1523123cm 2.0,arrowlength=1.4,arrowinset=0.4]{->}(2.46,0.37)(3.94,0.93)
\psline[linewidth=0.04cm,linestyle=dashed,dash=0.16cm 0.1cm,arrowsize=0.1523123cm 2.0,arrowlength=1.4,arrowinset=0.4]{->}(3.26,-0.55)(2.32,-0.39)
\psline[linewidth=0.04cm,linestyle=dashed,dash=0.16cm 0.1cm,arrowsize=0.1523123cm 2.0,arrowlength=1.4,arrowinset=0.4]{->}(2.08,-0.89)(2.84,-1.59)
\psline[linewidth=0.04cm,linestyle=dashed,dash=0.16cm 0.1cm,arrowsize=0.1523123cm 2.0,arrowlength=1.4,arrowinset=0.4]{->}(1.1,-1.69)(1.4,-0.91)
\psline[linewidth=0.04cm,linestyle=dashed,dash=0.16cm 0.1cm,arrowsize=0.1523123cm 2.0,arrowlength=1.4,arrowinset=0.4]{->}(0.96,-0.89)(0.26,-1.43)
\psline[linewidth=0.04cm,linestyle=dashed,dash=0.16cm 0.1cm,arrowsize=0.1523123cm 2.0,arrowlength=1.4,arrowinset=0.4]{->}(0.04,-0.87)(0.94,-0.45)
\psline[linewidth=0.04cm,linestyle=dashed,dash=0.16cm 0.1cm,arrowsize=0.1523123cm 2.0,arrowlength=1.4,arrowinset=0.4]{->}(0.74,0.41)(0.0,1.03)
\usefont{T1}{ptm}{m}{n}
\rput(3.9714062,1.235){$\partial \tilde K$}
\end{pspicture} 
}
}

\subfloat[\label{jon3}]
{
% Generated with LaTeXDraw 2.0.8
% Sun Mar 24 14:14:18 PDT 2013
% \usepackage[usenames,dvipsnames]{pstricks}
% \usepackage{epsfig}
% \usepackage{pst-grad} % For gradients
% \usepackage{pst-plot} % For axes
\scalebox{1} % Change this value to rescale the drawing.
{
\begin{pspicture}(0,-1.9584374)(3.96,1.9384375)
\pscircle[linewidth=0.04,dimen=outer](1.6,0.0184375){1.36}
\psdots[dotsize=0.12](1.7,1.3384376)
\psdots[dotsize=0.12](2.8,0.6384375)
\psdots[dotsize=0.12](2.9,-0.3615625)
\psdots[dotsize=0.12](2.4,-1.0615625)
\psdots[dotsize=0.12](1.2,-1.2615625)
\psdots[dotsize=0.12](0.7,-0.9615625)
\psdots[dotsize=0.12](0.4,-0.5615625)
\psdots[dotsize=0.12](0.5,0.7384375)
\psline[linewidth=0.04cm,linestyle=dashed,dash=0.16cm 0.1cm,arrowsize=0.1523123cm 2.0,arrowlength=1.4,arrowinset=0.4]{->}(2.46,0.4984375)(3.94,1.0584375)
\psline[linewidth=0.04cm,linestyle=dashed,dash=0.16cm 0.1cm,arrowsize=0.1523123cm 2.0,arrowlength=1.4,arrowinset=0.4]{->}(3.26,-0.4215625)(2.32,-0.2615625)
\psline[linewidth=0.04cm,linestyle=dashed,dash=0.16cm 0.1cm,arrowsize=0.1523123cm 2.0,arrowlength=1.4,arrowinset=0.4]{->}(1.1,-1.5615625)(1.4,-0.7815625)
\psline[linewidth=0.04cm,linestyle=dashed,dash=0.16cm 0.1cm,arrowsize=0.1523123cm 2.0,arrowlength=1.4,arrowinset=0.4]{->}(0.96,-0.7615625)(0.26,-1.3015625)
\psline[linewidth=0.04cm,linestyle=dashed,dash=0.16cm 0.1cm,arrowsize=0.1523123cm 2.0,arrowlength=1.4,arrowinset=0.4]{->}(0.04,-0.7415625)(0.94,-0.3215625)
\psline[linewidth=0.04cm,linestyle=dashed,dash=0.16cm 0.1cm,arrowsize=0.1523123cm 2.0,arrowlength=1.4,arrowinset=0.4]{->}(0.74,0.5384375)(0.0,1.1584375)
\usefont{T1}{ptm}{m}{n}
\rput(2.4714062,-1.7365625){$\partial \tilde K$}
\psbezier[linewidth=0.04,linestyle=dashed,dash=0.16cm 0.1cm,arrowsize=0.1523123cm 2.0,arrowlength=1.4,arrowinset=0.4]{->}(1.72,1.9184375)(1.7,1.6984375)(1.7286909,1.5772514)(1.68,1.3384376)(1.6313092,1.0996236)(1.52,0.2784375)(1.64,-0.0615625)(1.76,-0.4015625)(2.2492707,-0.9222526)(2.4,-1.0615625)(2.5507293,-1.2008724)(2.74,-1.3415625)(2.94,-1.6215625)
\usefont{T1}{ptm}{m}{n}
\rput(1.4014063,1.5834374){$z_1$}
\usefont{T1}{ptm}{m}{n}
\rput(2.7314062,-1.0765625){$z_i$}
\usefont{T1}{ptm}{m}{n}
\rput(2.8614063,0.9434375){$z_2$}
\end{pspicture} 
}
}\qquad
\subfloat[\label{jon4}]
{
% Generated with LaTeXDraw 2.0.8
% Sun Mar 24 14:17:35 PDT 2013
% \usepackage[usenames,dvipsnames]{pstricks}
% \usepackage{epsfig}
% \usepackage{pst-grad} % For gradients
% \usepackage{pst-plot} % For axes
\scalebox{1} % Change this value to rescale the drawing.
{
\begin{pspicture}(0,-1.85)(5.261875,1.85)
\pscircle[linewidth=0.04,dimen=outer](1.84,-0.09){1.36}
\psdots[dotsize=0.12](1.94,1.23)
\psdots[dotsize=0.12](3.04,0.53)
\psdots[dotsize=0.12](3.14,-0.47)
\psdots[dotsize=0.12](2.64,-1.17)
\psdots[dotsize=0.12](1.44,-1.37)
\psdots[dotsize=0.12](0.94,-1.07)
\psdots[dotsize=0.12](0.64,-0.67)
\psdots[dotsize=0.12](0.74,0.63)
\psbezier[linewidth=0.04,linestyle=dashed,dash=0.16cm 0.1cm,arrowsize=0.1523123cm 2.0,arrowlength=1.4,arrowinset=0.4]{->}(2.0,1.83)(1.9,1.61)(1.8946548,1.4084758)(1.94,1.25)(1.9853452,1.0915242)(2.0495415,0.7625156)(2.2,0.67)(2.3504584,0.57748437)(2.86,0.53)(3.02,0.53)(3.18,0.53)(3.46,0.59)(3.66,0.81)
\psbezier[linewidth=0.04,linestyle=dashed,dash=0.16cm 0.1cm,arrowsize=0.1523123cm 2.0,arrowlength=1.4,arrowinset=0.4]{->}(3.64,-0.57)(3.38,-0.49)(3.22,-0.49)(3.12,-0.47)(3.02,-0.45)(2.6064017,-0.41049954)(2.52,-0.51)(2.4335983,-0.60950047)(2.5005684,-1.055488)(2.62,-1.17)(2.7394316,-1.284512)(2.94,-1.39)(3.06,-1.49)
\psbezier[linewidth=0.04,linestyle=dashed,dash=0.16cm 0.1cm,arrowsize=0.1523123cm 2.0,arrowlength=1.4,arrowinset=0.4]{->}(1.4,-1.83)(1.36,-1.59)(1.42,-1.45)(1.44,-1.37)(1.46,-1.29)(1.42,-1.13)(1.34,-1.05)(1.26,-0.97)(1.04,-1.03)(0.94,-1.07)(0.84,-1.11)(0.64,-1.21)(0.48,-1.49)
\psbezier[linewidth=0.04,linestyle=dashed,dash=0.16cm 0.1cm,arrowsize=0.1523123cm 2.0,arrowlength=1.4,arrowinset=0.4]{->}(0.12,-0.73)(0.32,-0.69)(0.48,-0.65)(0.64,-0.67)(0.8,-0.69)(1.02,-0.47)(1.12,-0.25)(1.22,-0.03)(0.88,0.45)(0.72,0.63)(0.56,0.81)(0.42,0.73)(0.0,0.63)
\usefont{T1}{ptm}{m}{n}
\rput(3.9714062,0.915){$\partial \tilde K$}
\end{pspicture} 
}
}

\subfloat[\label{jon5}]
{
% Generated with LaTeXDraw 2.0.8
% Sun Mar 24 14:20:42 PDT 2013
% \usepackage[usenames,dvipsnames]{pstricks}
% \usepackage{epsfig}
% \usepackage{pst-grad} % For gradients
% \usepackage{pst-plot} % For axes
\scalebox{1} % Change this value to rescale the drawing.
{
\begin{pspicture}(0,-1.87)(5.381875,1.87)
\pscircle[linewidth=0.04,dimen=outer](1.74,0.01){1.36}
\psdots[dotsize=0.12](1.84,1.33)
\psdots[dotsize=0.12](2.94,0.63)
\psdots[dotsize=0.12](3.04,-0.37)
\psdots[dotsize=0.12](2.54,-1.07)
\psdots[dotsize=0.12](1.34,-1.27)
\psdots[dotsize=0.12](0.84,-0.97)
\psdots[dotsize=0.12](0.54,-0.57)
\psdots[dotsize=0.12](0.64,0.73)
\psbezier[linewidth=0.04,linestyle=dashed,dash=0.16cm 0.1cm,arrowsize=0.1523123cm 2.0,arrowlength=1.4,arrowinset=0.4]{->}(0.92,1.61)(1.3,1.85)(1.7946547,1.5084757)(1.84,1.35)(1.8853452,1.1915243)(1.9495416,0.8625156)(2.1,0.77)(2.2504585,0.6774844)(2.76,0.63)(2.92,0.63)(3.08,0.63)(3.58,0.71)(3.72,0.17)
\psbezier[linewidth=0.04,linestyle=dashed,dash=0.16cm 0.1cm,arrowsize=0.1523123cm 2.0,arrowlength=1.4,arrowinset=0.4]{->}(3.68,0.21)(3.68,-0.05)(3.12,-0.39)(3.02,-0.37)(2.92,-0.35)(2.5064015,-0.31049952)(2.42,-0.41)(2.3335984,-0.50950044)(2.4005682,-0.955488)(2.52,-1.07)(2.6394317,-1.184512)(2.58,-1.47)(2.14,-1.59)
\psbezier[linewidth=0.04,linestyle=dashed,dash=0.16cm 0.1cm,arrowsize=0.1523123cm 2.0,arrowlength=1.4,arrowinset=0.4]{->}(2.12,-1.59)(1.7,-1.85)(1.32,-1.35)(1.34,-1.27)(1.36,-1.19)(1.32,-1.03)(1.24,-0.95)(1.16,-0.87)(0.94,-0.93)(0.84,-0.97)(0.74,-1.01)(0.3,-1.83)(0.1,-0.99)
\psbezier[linewidth=0.04,linestyle=dashed,dash=0.16cm 0.1cm,arrowsize=0.1523123cm 2.0,arrowlength=1.4,arrowinset=0.4]{->}(0.08,-1.01)(0.0,-0.57)(0.38,-0.55)(0.54,-0.57)(0.7,-0.59)(0.92,-0.37)(1.02,-0.15)(1.12,0.07)(0.78,0.55)(0.62,0.73)(0.46,0.91)(0.5,1.33)(0.92,1.61)
\usefont{T1}{ptm}{m}{n}
\rput(4.0914063,0.315){$\partial \tilde K$}
\end{pspicture} 
}
}\qquad
\subfloat[\label{jon6}]
{
% Generated with LaTeXDraw 2.0.8
% Sun Mar 24 14:24:50 PDT 2013
% \usepackage[usenames,dvipsnames]{pstricks}
% \usepackage{epsfig}
% \usepackage{pst-grad} % For gradients
% \usepackage{pst-plot} % For axes
\scalebox{0.8} % Change this value to rescale the drawing.
{
\begin{pspicture}(0,-2.351911)(8.272506,2.351911)
\pscircle[linewidth=0.04,dimen=outer](3.410631,0.31191108){1.36}
\psdots[dotsize=0.12](3.510631,1.631911)
\psdots[dotsize=0.12](4.610631,0.93191105)
\psdots[dotsize=0.12](4.710631,-0.068088934)
\psdots[dotsize=0.12](4.210631,-0.76808894)
\psdots[dotsize=0.12](3.010631,-0.9680889)
\psdots[dotsize=0.12](2.510631,-0.6680889)
\psdots[dotsize=0.12](2.210631,-0.26808894)
\psdots[dotsize=0.12](2.310631,1.031911)
\psbezier[linewidth=0.04,linestyle=dashed,dash=0.16cm 0.1cm,arrowsize=0.1523123cm 2.0,arrowlength=1.4,arrowinset=0.4]{->}(4.430631,2.191911)(3.810631,2.331911)(3.610631,2.0319111)(3.510631,1.651911)(3.410631,1.271911)(3.6201725,1.1644267)(3.770631,1.0719111)(3.9210894,0.97939545)(4.430631,0.93191105)(4.590631,0.93191105)(4.750631,0.93191105)(5.310631,1.0519111)(5.3706307,0.49191105)
\psbezier[linewidth=0.04,linestyle=dashed,dash=0.16cm 0.1cm,arrowsize=0.1523123cm 2.0,arrowlength=1.4,arrowinset=0.4]{->}(5.3506308,0.51191103)(5.3506308,0.25191107)(4.790631,-0.08808894)(4.690631,-0.068088934)(4.590631,-0.048088938)(4.1770325,-0.008588469)(4.090631,-0.10808894)(4.0042295,-0.2075894)(4.0711994,-0.65357697)(4.190631,-0.76808894)(4.310063,-0.8826009)(4.250631,-1.1680889)(3.810631,-1.2880889)
\psbezier[linewidth=0.04,linestyle=dashed,dash=0.16cm 0.1cm,arrowsize=0.1523123cm 2.0,arrowlength=1.4,arrowinset=0.4]{->}(3.790631,-1.2880889)(3.370631,-1.5480889)(2.9906309,-1.0480889)(3.010631,-0.9680889)(3.030631,-0.88808894)(2.9906309,-0.7280889)(2.910631,-0.64808893)(2.830631,-0.56808895)(2.610631,-0.62808895)(2.510631,-0.6680889)(2.410631,-0.70808893)(1.970631,-1.5280889)(1.770631,-0.68808895)
\usefont{T1}{ptm}{m}{n}
\rput(6.982037,0.25691107){$\partial \tilde K$}
\psbezier[linewidth=0.04,linestyle=dashed,dash=0.16cm 0.1cm,arrowsize=0.1523123cm 2.0,arrowlength=1.4,arrowinset=0.4]{<-}(4.450631,2.171911)(4.950631,2.071911)(6.090631,1.5719111)(6.450631,0.75191104)(6.810631,-0.068088934)(6.1057973,-1.3642668)(5.230631,-1.848089)(4.3554645,-2.331911)(1.6253798,-2.1008074)(0.990631,-1.3280889)(0.35588214,-0.5553704)(0.0,1.4150878)(0.97063094,1.511911)(1.9412619,1.6087344)(1.910631,1.151911)(2.290631,1.031911)(2.670631,0.9119111)(2.690631,0.55191106)(2.630631,0.27191105)(2.570631,-0.008088937)(2.390631,-0.18808894)(2.210631,-0.26808894)(2.030631,-0.34808895)(1.650631,-0.44808894)(1.770631,-0.7280889)
\end{pspicture} 
}
}
\caption
{
\label{fig:jon}{\bf The construction of the topologically unique pair $K$ and $\tilde K$ of transversely positioned closed Jordan domains, not cutting each other, having boundaries meeting at 8 points.}  The orientation on $\partial K = \partial \bbD$ is the usual, positive one, that is, the counterclockwise one. In all cases, dashed arcs are sub-arcs of $\partial \tilde K$.
}
\end{figure}

\begin{proof}[Proof of Lemma \ref{no cut index}]
\phantomsection \label{no cut index proof}
In light of Lemmas \ref{prop:index invariance} and \ref{uniqa}, we may suppose that $K$ and $\tilde K$ are as in Figure \ref{no cut index figure}, of course drawn with the correct number of meeting points between $\partial K$ and $\partial \tilde K$. Recalling that $\eta(f)$ counts the winding number of the vector $f(z) - z$ around the origin, we consider when it is possible for $f(z) - z$ to be a positive real number. If $z$ does not lie on the intersection of the left side of the rectangle $\partial K$ with the interior of $\tilde K$, then $f(z) - z$ certainly has either a negative real component, or a non-zero imaginary component. Thus pick a $z$ lying in this intersection. Then the only way for $f(z) - z$ to be real and positive is for $f(z)$ to lie on the semicircular sub-arc of $\partial \tilde K$ to the right of $z$, as in the figure. But then, considering the orientations on $\partial K$ and $\partial \tilde K$, we get that the vector $f(z) - z$ must always be turning counterclockwise at such a $z$, as we traverse $\partial K$ in the positive direction. We conclude that whenever the curve $\{f(z) - z\}_{\partial K}$ crosses the positive real axis, it does so in the positive direction, thus the winding number of this curve around the origin is non-negative.
\end{proof}

\begin{figure}
\centering
% Generated with LaTeXDraw 2.0.8
% Sun Mar 24 16:08:28 PDT 2013
% \usepackage[usenames,dvipsnames]{pstricks}
% \usepackage{epsfig}
% \usepackage{pst-grad} % For gradients
% \usepackage{pst-plot} % For axes
\scalebox{1} % Change this value to rescale the drawing.
{
\begin{pspicture}(0,-3.49)(7.3028126,3.497)
\psframe[linewidth=0.04,dimen=outer](6.6609373,2.91)(3.6609375,-3.49)
\psarc[linewidth=0.04,linestyle=dashed,dash=0.16cm 0.1cm](3.6609375,2.11){0.6}{-90.0}{90.0}
\psarc[linewidth=0.04,linestyle=dashed,dash=0.16cm 0.1cm](3.6609375,-0.29){0.6}{-90.0}{90.0}
\psarc[linewidth=0.04,linestyle=dashed,dash=0.16cm 0.1cm](3.6609375,-2.69){0.6}{-90.0}{90.0}
\psarc[linewidth=0.04,linestyle=dashed,dash=0.16cm 0.1cm](3.6609375,0.91){0.6}{90.0}{270.0}
\psarc[linewidth=0.04,linestyle=dashed,dash=0.16cm 0.1cm](3.6609375,-1.49){0.6}{90.0}{270.0}
\psarc[linewidth=0.04,linestyle=dashed,dash=0.16cm 0.1cm](3.6609375,-0.29){3.0}{90.0}{270.0}
\psline[linewidth=0.04cm,linestyle=dotted,dotsep=0.05cm,arrowsize=0.05291667cm 2.0,arrowlength=1.4,arrowinset=0.4]{->}(3.6609375,2.11)(4.2609377,2.11)
\psdots[dotsize=0.12](3.6609375,2.11)
\psdots[dotsize=0.12](4.2609377,2.11)
\usefont{T1}{ptm}{m}{n}
\rput(4.782344,2.135){$f(z)$}
\usefont{T1}{ptm}{m}{n}
\rput(3.3223438,2.115){$z$}
\usefont{T1}{ptm}{m}{n}
\rput(6.952344,1.235){$K$}
\usefont{T1}{ptm}{m}{n}
\rput(0.67234373,1.535){$\tilde K$}
\psbezier[linewidth=0.0139999995,arrowsize=0.05291667cm 2.0,arrowlength=1.4,arrowinset=0.4]{<-}(3.9209375,2.19)(4.0209374,3.11)(4.0209374,3.49)(3.4009376,3.37)(2.7809374,3.25)(2.4409375,3.11)(1.9409375,3.09)
\usefont{T1}{ptm}{m}{n}
\rput(1.1623437,3.075){$f(z) - z$}
\end{pspicture} 
}
\caption
{
\label{no cut index figure} We see that whenever $f(z) - z$ is positive and real, as we traverse $\partial K$ and $\partial \tilde K$ positively, the point $f(z)$ is moving upward, and the point $z$ is moving downward, so the vector $f(z) - z$ is winding counterclockwise, thus in the positive direction.
}
\end{figure}
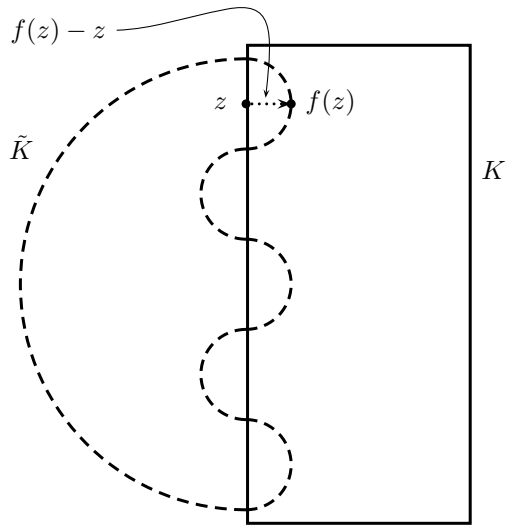

\begin{remark}
The same Figure \ref{no cut index figure} can be used to show that under the hypotheses of Lemma \ref{no cut index}, we get that $\eta(f) \le 2$. Thus the only fixed-point indices which can be achieved in this setting are $0, 1, 2$, and all three of these occur. We do not use these facts, so working out the details is left as an exercise for the interested reader.
\end{remark}

\section{Proof of rigidity and uniformization of circle packings}
\label{sec:rigid in plane}
\label{cp proofs}

As an example of the power of fixed-point index, in this section we prove three rigidity and uniformization theorems on circle packings. Theorem \ref{rigid in sphere} is usually credited to Koebe \cite{koebe-1936}, Andreev \cite{MR0273510}, and Thurston\footnote{Originally at his talk at the International Congress of Mathematicians, Helsinki, 1978, according to \cite{MR1303402}*{p.\ 135}. See also \cite{thurston-gt3m-notes}*{Chapter 13}.}. All of their proofs were via methods different from ours. Theorems \ref{rigid in plane} and \ref{thm:no luv} are originally due to Schramm \cite{MR1076089}*{Rigidity Theorems 1.1, 5.1}, who proved them essentially via the Incompatibility Theorem \ref{incompat}, although his proof of the Incompatibility Theorem is not via fixed-point index techniques. The first ones to study circle packings by fixed-point index techniques were He and Schramm in \cite{MR1207210}, although they did not proceed via the Incompatibility Theorem, instead applying other normalizations to the packings in question, to similarly argue by contradiction. For further references on circle packing, see for example the articles \citelist{\cite{MR1303402} \cite{MR2884870}} and their bibliographies.\medskip

We begin with some basic definitions. A \emph{circle packing} is a collection $\calP = \{D_i\}$ of pairwise interiorwise disjoint round closed disks in the Riemann sphere $\hat\bbC$. The \emph{contact graph} of a circle packing is the graph having a vertex for every disk of the packing, so that two vertices share an edge if and only if the corresponding disks meet. Then the following hold:

\begin{theorem}
\label{rigid in sphere}
Suppose that $\calP$ and $\tilde\calP$ are circle packings in $\hat\bbC$, sharing a contact graph that triangulates the 2-sphere $\bbS^2$. Then $\calP$ and $\tilde\calP$ differ by a M\"obius or anti-M\"obius transformation.
\end{theorem}

\begin{theorem}
\label{rigid in plane}
Suppose that $\calP$ and $\tilde\calP$ are circle packings which are locally finite in $\bbC$, sharing a contact graph that triangulates a topological open disk. Then $\calP$ and $\tilde\calP$ differ by a Euclidean similarity.
\end{theorem}

\begin{theorem}
\label{thm:no luv}
There cannot be two circle packings $\calP$ and $\tilde \calP$ sharing a contact graph triangulating a topological open disk, so that one is locally finite in $\bbC$ and the other is locally finite in the hyperbolic plane $\bbH^2$, equivalently the open unit disk $\bbD$.
\end{theorem}

The rest of this section consists of the proofs of these three theorems. Before moving on, we make one note:

\begin{remark}
The statement for locally finite packings in the hyperbolic plane $\bbH^2 \cong \bbD$ analogous to Theorem \ref{rigid in plane} also holds. However, to apply our techniques to prove it, one would need to show that two combinatorially equivalent packings, both locally finite in $\bbH^2$, having contact graphs triangulating $\bbH^2$, induce a homeomorphism, or at least some appropriately behaved identification, on the boundary $\partial_\infty \bbH^2 \cong \partial \bbD$. This turns out to be true (in fact, it follows from the rigidity theorem we discuss in this remark), but non-trivial.
\end{remark}

\begin{proof}[Proof of Theorem \ref{rigid in sphere}]
The argument is by contradiction. The main idea is to superimpose the packings $\calP$ and $\tilde\calP$ on the Riemann sphere in a convenient way. In particular, we isolate two topological quadrilaterals $Q$ and $\tilde Q$ so that they are packed in combinatorially equivalent ways by disks of $\calP$ and $\tilde \calP$, but so that $Q$ and $\tilde Q$ cut each other, as in Figure \ref{clown2}. Then the simple observation that two round disks cannot cut each other gives us our desired contradiction, via the Incompatibility Theorem \ref{incompat}.\medskip

Let $X = (V,E,F)$ be the common triangulation of $\bbS^2$ which $\calP$ and $\tilde\calP$ realize. Suppose for contradiction that $\calP$ and $\tilde \calP$ are not equivalent under any M\"obius or anti-M\"obius transformation. For the first part of the proof, we apply a sequence of normalizations to $\calP$ and to $\tilde\calP$. Let $f_0 = \left<v_1,v_2,v_3\right> \in F$ be a face of $X$. We first normalize via a M\"obius transformation so that $D_i = \tilde D_i$ for $i=1,2,3$. Here $D_i\in \calP$ is the disk corresponding to the vertex $v_i\in V$ of $X$, similarly $\tilde D_i\in \tilde\calP$. In particular, the correct M\"obius transformation is the one that sends the intersection points of the $\partial D_i$ to the corresponding ones of the $\partial \tilde D_i$.

Our next normalization is in our initial choice of $f_0$ and our labeling of the $v_i$, as per the following observation:

\begin{observation}
\label{obs:old prf}
Let $v_4$ denote the vertex of $X$ other than $v_1$ so that $\left<v_2,v_3,v_4\right>$ is a face of $X$. Then there is some choice of $f_0 = \left<v_1,v_2,v_3\right>$ so that the disks $D_4$ and $\tilde D_4$ are not equal after our normalization identifying $D_i = \tilde D_i$ for $i=1,2,3$.
\end{observation}

\noindent If there were no such choice of $f_0$ then in fact every pair of corresponding disks $D_i$ and $\tilde D_i$ would coincide after our first normalization, and so $\calP$ and $\tilde \calP$ coincide.

An \emph{interstice} of the packing $\calP$ is a connected component of $\hat\bbC \setminus \calP$. Every interstice is necessarily a curvilinear triangle, because the packing's contact graph triangulates $\hat\bbC$. For our next normalization, we insist that $\infty$ lies in the interstice formed by $D_1=\tilde D_1,D_2= \tilde D_2,D_3=\tilde D_3$ which contains no other disks of either packing. Finally, we insist that $D_1=\tilde D_1,D_2=\tilde D_2,D_3=\tilde D_3$ all have Euclidean radius 1, and that $D_2$ and $D_3$ are tangent at a point lying on the horizontal axis, so that $D_1$ lies to their left. The situation for $\calP$ is depicted in Figure \ref{fig:calp}.\medskip

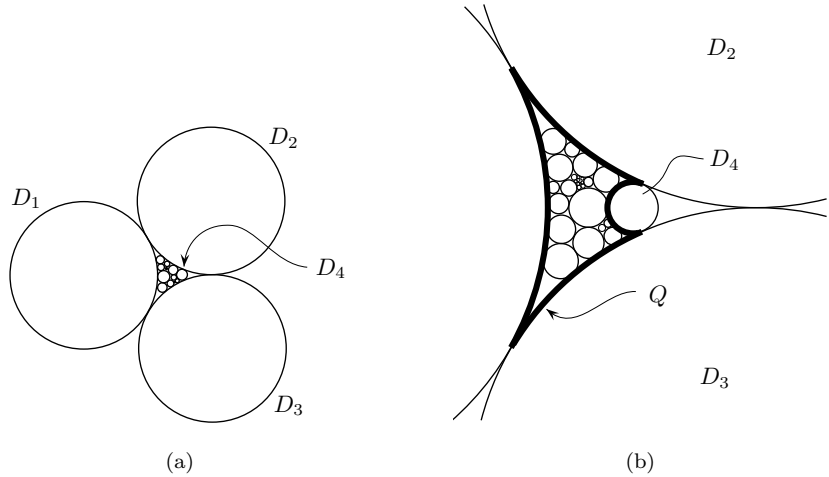
\begin{figure}
\centering
\subfloat[]
{
\label{fig:qq0}
% Generated with LaTeXDraw 2.0.8
% Tue Jul 10 01:08:42 EDT 2012
% \usepackage[usenames,dvipsnames]{pstricks}
% \usepackage{epsfig}
% \usepackage{pst-grad} % For gradients
% \usepackage{pst-plot} % For axes
\scalebox{0.9} % Change this value to rescale the drawing.
{
\begin{pspicture}(0,-2.19)(5.52,2.21)
\pscircle[linewidth=0.02,dimen=outer](1.34,-0.01){1.1}
\pscircle[linewidth=0.02,dimen=outer](3.22,1.09){1.1}
\pscircle[linewidth=0.02,dimen=outer](3.24,-1.09){1.1}
\pscircle[linewidth=0.02,dimen=outer](2.79,0.0){0.09}
\pscircle[linewidth=0.02,dimen=outer](2.47,0.22){0.07}
\pscircle[linewidth=0.02,dimen=outer](2.5,-0.19){0.08}
\pscircle[linewidth=0.02,dimen=outer](2.52,-0.03){0.1}
\pscircle[linewidth=0.02,dimen=outer](2.62,-0.13){0.06}
\pscircle[linewidth=0.02,dimen=outer](2.48,0.11){0.06}
\pscircle[linewidth=0.02,dimen=outer](2.66,0.07){0.08}
\pscircle[linewidth=0.02,dimen=outer](2.57,0.16){0.05}
\pscircle[linewidth=0.02,dimen=outer](2.67,-0.04){0.05}
\psdots[dotsize=0.051999997](2.56,0.09)
\pscircle[linewidth=0.02,dimen=outer](2.72,-0.09){0.04}
\psdots[dotsize=0.04](2.68,-0.11)
\usefont{T1}{ptm}{m}{n}
\rput(0.49,1.125){$D_1$}
\usefont{T1}{ptm}{m}{n}
\rput(4.29,2.018){$D_2$}
\usefont{T1}{ptm}{m}{n}
\rput(4.37,-1.935){$D_3$}
\usefont{T1}{ptm}{m}{n}
\rput(4.99,0.095){$D_4$}
\psbezier[linewidth=0.01](2.82,0.12)(2.84,0.49)(3.113913,0.75)(3.5,0.73)(3.886087,0.71)(4.16,0.15)(4.64,0.11)
\psline[linewidth=0.04,linestyle=none]{<-}(2.82,0.12)(2.85,0.49)
\end{pspicture} 
}
}\qquad
\subfloat[]
{\label{fig:qq}
% Generated with LaTeXDraw 2.0.8
% Tue Jul 10 04:50:52 EDT 2012
% \usepackage[usenames,dvipsnames]{pstricks}
% \usepackage{epsfig}
% \usepackage{pst-grad} % For gradients
% \usepackage{pst-plot} % For axes
\scalebox{0.9} % Change this value to rescale the drawing.
{
\begin{pspicture*}(7,-3.1801555)(12.527415,3.0141149)
\psarc[linewidth=0.02](11.48,4.2){4.2}{196.54298}{284.30026}
\psarc[linewidth=0.02](4.2,0.0){4.2}{310.9922}{45.0}
\psarc[linewidth=0.02](11.48,-4.2){4.2}{75.69972}{165.42578}
\pscircle[linewidth=0.02,dimen=outer](9.66,0.0){0.38}
\usefont{T1}{ptm}{m}{n}
\rput(11.01,0.705){$D_4$}
\usefont{T1}{ptm}{m}{n}
\rput(6.81,0.745){$D_1$}
\usefont{T1}{ptm}{m}{n}
\rput(10.97,2.345){$D_2$}
\usefont{T1}{ptm}{m}{n}
\rput(10.87,-2.515){$D_3$}
\pscircle[linewidth=0.02,dimen=outer](8.59,-0.79){0.27}
\pscircle[linewidth=0.02,dimen=outer](8.48,0.98){0.2}
\pscircle[linewidth=0.02,dimen=outer](8.57,0.59){0.21}
\pscircle[linewidth=0.02,dimen=outer](8.6,-0.32){0.22}
\pscircle[linewidth=0.02,dimen=outer](9.0,0.0){0.3}
\pscircle[linewidth=0.016,dimen=outer](9.2,-0.3){0.06}
\pscircle[linewidth=0.02,dimen=outer](8.71,0.29){0.13}
\pscircle[linewidth=0.02,dimen=outer](9.36,-0.42){0.14}
\pscircle[linewidth=0.02,dimen=outer](8.56,0.06){0.16}
\pscircle[linewidth=0.02,dimen=outer](8.48,0.3){0.1}
\pscircle[linewidth=0.02,dimen=outer](9.0,-0.52){0.24}
\pscircle[linewidth=0.02,dimen=outer](8.76,0.86){0.12}
\pscircle[linewidth=0.02,dimen=outer](9.26,0.42){0.2}
\pscircle[linewidth=0.02,dimen=outer](8.95,0.63){0.19}
\pscircle[linewidth=0.02,dimen=outer](9.29,-0.23){0.07}
\pscircle[linewidth=0.02,dimen=outer](9.0,0.38){0.08}
\pscircle[linewidth=0.016,dimen=outer](8.79,0.45){0.05}
\pscircle[linewidth=0.016,dimen=outer](8.92,0.32){0.04}
\pscircle[linewidth=0.016,dimen=outer](8.86,0.3){0.04}
\pscircle[linewidth=0.0139999995,dimen=outer](8.85,0.45){0.03}
\pscircle[linewidth=0.0139999995,dimen=outer](8.81,0.39){0.03}
\pscircle[linewidth=0.0139999995,dimen=outer](8.85,0.41){0.03}
\pscircle[linewidth=0.0139999995,dimen=outer](8.91,0.43){0.03}
\pscircle[linewidth=0.0139999995,dimen=outer](8.85,0.35){0.03}
\pscircle[linewidth=0.0139999995,dimen=outer](8.88,0.36){0.02}
\pscircle[linewidth=0.0139999995,dimen=outer](8.91,0.37){0.03}
%\pscircle[linewidth=0.0139999995,dimen=outer](8.86,0.4){0.07}
%\psbezier[linewidth=0.0139999995](8.76,0.36)(8.36,0.045)(8.08,0.0375)(7.98,-0.09)(7.88,-0.2175)(7.86,-0.46)(7.84,-0.765)
%\psline[linewidth=0.04,linestyle=none]{<-}(8.76,0.36)(8.36,0.045)
\usefont{T1}{ptm}{m}{n}
%\rput(7.8,-0.995){$z_\infty$}
\psarc[linewidth=0.09](11.48,4.2){4.2}{210.49336}{246.30931}
\psarc[linewidth=0.09](4.2,0.0){4.2}{330.50504}{29.389011}
\psarc[linewidth=0.09](11.48,-4.2){4.2}{113.7495}{148.75307}
\psarc[linewidth=0.09](9.66,0.0){0.38}{66.80141}{281.30994}
\psbezier[linewidth=0.0139999995](8.38,-1.5)(8.82,-1.92)(9.24,-1.22)(9.7,-1.24)
\psline[linewidth=0.04,linestyle=none]{<-}(8.38,-1.5)(8.92,-1.92)
\usefont{T1}{ptm}{m}{n}
\rput(10.03,-1.315){$Q$}
\psbezier[linewidth=0.0139999995](9.8,0.2)(10.18,0.42)(10.12,0.58)(10.24,0.7)(10.36,0.82)(10.5,0.74)(10.7,0.72)
%\psline[linewidth=0.04,linestyle=none]{<-}(9.8,0.2)(10.18,0.42)
\end{pspicture*} 
}
}
\caption[The packing $\calP$ after some normalizations]
{
\label{fig:calp}
{\bf The packing $\calP$ after some normalizations.}  The disks of $\calP$ all lie between $D_1,D_2,D_3$. Note that the interstice formed by $D_1, D_2, D_3$ on $\hat\bbC$ is the outside region in these figures. The disk $D_4$ is ``the first disk of $\calP\setminus \{D_1,D_2,D_3\}$ we get to if we start scanning from the right.''  In \subref{fig:qq} the topological quadrilateral $Q$ is outlined in bold.
}
\end{figure}

%We summarize the picture as it stands after these normalizations. First, we return to the plane $\bbC$. We insist that the Euclidean center of $D_1$ and the tangency point between $D_2$ and $D_3$ both lie on the horizontal axis, with the center of $D_1$ to the left of the tangency point between $D_2$ and $D_3$. The collections $\calP$ and $\tilde\calP$ are still made up of closed disks, since the point we sent to infinity does not lie in any disk of either packing. Corresponding pairs of disks $D_i$ and $\tilde D_i$ are equal for $i = 1,2,3$, and these disks are the ``boundary disks'' of both packings. The disks $\calP$ accumulate around some point $z_\infty$ in the plane, as do the disks of $\tilde\calP$ around $\tilde z_\infty$. The points $z_\infty$ and $\tilde z_\infty$ may be the same or may be different. Finally, every face of $X$ corresponds to an interstice in each packing, except for $f= \left<v_1,v_2,v_3\right>$: the disks $D_i$, equivalently $\tilde D_i$, for $i=1,2,3$, form an interstice in $\hat\bbC$, but not in $\bbC$. Recall that for us, an interstice is a topological open disk that meets no disks of the packing.

From now on we work in the plane $\bbC$, in the sense that $\infty \in \hat\bbC$ will not move again for the remainder of the proof. Note that every face $f$ of $F$ corresponds to some interstices $U_f\subset \bbC$ and $\tilde U_f\subset \bbC$ of $\calP$ and $\tilde \calP$ respectively, except for $f_0$, for which the interstices $U_{f_0} = \tilde U_{f_0}$ contain $\infty$. Let $Q$ be the topological quadrilateral shown in Figure \ref{fig:qq}. More precisely, let $V_Q = V\setminus \{v_1,v_2,v_3\}$, and let $F_Q = F\setminus \{f_0 = \left<v_1,v_2,v_3\right>, \left<v_2,v_3,v_4\right>\}$. Then we define $Q = \bigcup_{v\in V_Q} D_v \cup \bigcup_{f\in F_Q} U_f$. Define the analogous objects for $\tilde\calP$ in the obvious way.\medskip

We now apply one final transformation to $\calP$. First, suppose without loss of generality that the Euclidean radius of $D_4$ is larger than that of $\tilde D_4$. Then translate every disk of $\calP$ to the right by a small amount $\varepsilon>0$, leaving the disks of $\tilde\calP$ unchanged. Denote this transformation by $T_\varepsilon$. We will discuss more precise requirements on $\varepsilon$ later. The situation is depicted in Figure \ref{fig:qqmain}. The essential point is that there is an open interval of values that $\varepsilon>0$ may take so that after all of our transformations, the topological quadrilaterals $Q$ and $\tilde Q$ are arranged qualitatively as in Figure \ref{fig:qq3}. In particular, we may choose $\varepsilon$ so that the packings of $Q$ and $\tilde Q$ by the remaining disks of $\calP$ and $\tilde\calP$ are in transverse position, because there are only finitely many values of $\varepsilon>0$ for which this fails. Then our desired contradiction follows immediately from the Incompatibility Theorem \ref{incompat}, because round closed disks cannot cut one another.
\end{proof}

\begin{figure}
\centering
\subfloat[]
{\label{fig:qq2}
% Generated with LaTeXDraw 2.0.8
% Tue Jul 10 01:40:46 EDT 2012
% \usepackage[usenames,dvipsnames]{pstricks}
% \usepackage{epsfig}
% \usepackage{pst-grad} % For gradients
% \usepackage{pst-plot} % For axes
\scalebox{1} % Change this value to rescale the drawing.
{
\begin{pspicture*}(6,-3.1801555)(12.62,3.0141149)
\psarc[linewidth=0.02](11.48,4.2){4.2}{196.54298}{284.30026}
\psarc[linewidth=0.02](4.2,0.0){4.2}{310.9922}{45.0}
\psarc[linewidth=0.02](11.48,-4.2){4.2}{75.69972}{165.42578}
\psarc[linewidth=0.03, linestyle=dashed,dash=0.1cm 0.08cm](11.48,4.2){4.2}{196.54298}{284.30026}
\psarc[linewidth=0.03, linestyle=dashed,dash=0.1cm 0.08cm](4.2,0.0){4.2}{310.9922}{45.0}
\psarc[linewidth=0.03, linestyle=dashed,dash=0.1cm 0.08cm](11.48,-4.2){4.2}{75.69972}{165.42578}
\pscircle[linewidth=0.02,dimen=outer](9.66,0.0){0.38}
\pscircle[linewidth=0.02,linestyle=dashed,dash=0.1cm 0.08cm,dimen=outer](10.0,0.0){0.26}
\psbezier[linewidth=0.01](9.46,0.24)(9.24,1.0)(9.0,1.24)(9.04,1.44)(9.08,1.64)(9.42,1.6)(9.78,1.56)
%\psline[linewidth=0.04,linestyle=none]{<-}(9.46,0.24)(9.24,1.0)
\psbezier[linewidth=0.01](10.14,-0.12)(10.14,-0.68)(10.54,-0.68)(10.56,-1.04)(10.58,-1.4)(10.5,-1.6)(9.88,-1.6)
%\psline[linewidth=0.04,linestyle=none]{<-}(10.14,-0.12)(10.16,-.68)
\usefont{T1}{ptm}{m}{n}
\rput(9.55,-1.615){$\tilde D_4$}
\usefont{T1}{ptm}{m}{n}
\rput(10.21,1.525){$D_4$}
\usefont{T1}{ptm}{m}{n}
\rput(7.08,0.025){$D_1 = \tilde D_1$}
\usefont{T1}{ptm}{m}{n}
\rput(10.44,2.745){$D_2 = \tilde D_2$}
\usefont{T1}{ptm}{m}{n}
\rput(10.4,-2.575){$D_3 = \tilde D_3$}
\end{pspicture*} 
}
}\qquad\qquad
\subfloat[]
{
\label{fig:qq3}
% Generated with LaTeXDraw 2.0.8
% Tue Jul 10 05:49:53 EDT 2012
% \usepackage[usenames,dvipsnames]{pstricks}
% \usepackage{epsfig}
% \usepackage{pst-grad} % For gradients
% \usepackage{pst-plot} % For axes
\scalebox{1} % Change this value to rescale the drawing.
{
\begin{pspicture*}(7,-2.0199099)(10.157331,2.1304433)
\psarc[linewidth=0.02,linestyle=dashed,dash=0.1cm 0.08cm](11.48,4.2){4.2}{209.87599}{251.39503}
\psarc[linewidth=0.02,linestyle=dashed,dash=0.1cm 0.08cm](4.2,0.0){4.2}{331.55707}{30.018368}
\psarc[linewidth=0.02,linestyle=dashed,dash=0.1cm 0.08cm](11.48,-4.2){4.2}{109.112976}{148.57043}
\psarc[linewidth=0.02](11.8,4.2){4.2}{209.6785}{246.19406}
\psarc[linewidth=0.02](4.52,0.0){4.2}{331.51782}{30.217274}
\psarc[linewidth=0.02](11.8,-4.2){4.2}{114.35301}{148.57043}
\psarc[linewidth=0.02](9.92,0.0){0.4}{61.38954}{293.9625}
\psarc[linewidth=0.02,linestyle=dashed,dash=0.1cm 0.08cm](10.04,0.0){0.24}{63.434948}{296.56506}
\usefont{T1}{ptm}{m}{n}
\rput(9.43,1.225){$Q$}
\usefont{T1}{ptm}{m}{n}
\rput(7.91,0.185){$\tilde Q$}
\end{pspicture*} 
}
}
\caption[The interaction between $\calP$ and $\tilde{\calP}$ before and after applying $T_\varepsilon$]
{
\label{fig:qqmain}
{\bf The interaction between $\calP$ and $\tilde\calP$ before and after applying $T_\varepsilon$.}  In \subref{fig:qq2} we see the superimposition of the $D_i$ with the $\tilde D_i$ before applying $T_\varepsilon$ to $\calP$. The disks $D_i$ are drawn solid, and the disks $\tilde D_i$ are drawn dashed. In \subref{fig:qq3} we see the relative positions of $Q$ and $\tilde Q$ after applying $T_\varepsilon$ to $\calP$.
}
\end{figure}
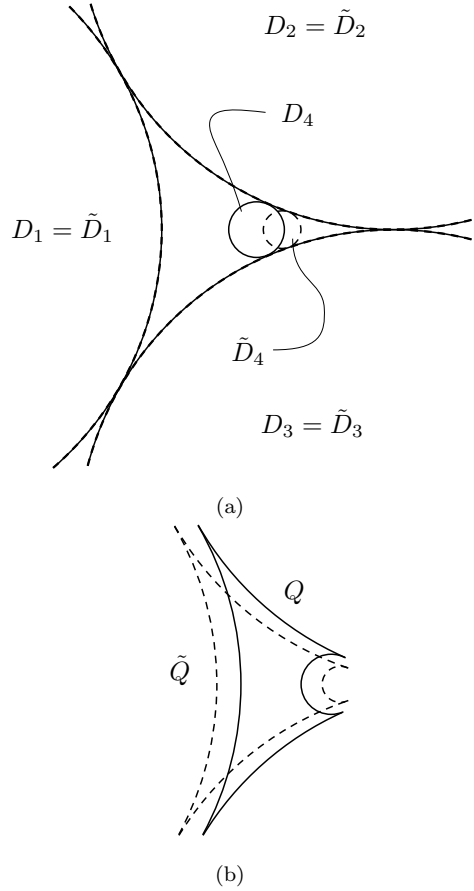

\begin{proof}[Proof of Theorem \ref{rigid in plane}]
The proof of Theorem \ref{rigid in plane} proceeds along the same lines, except that after our first round of normalizations identifying $D_i$ and $\tilde D_i$ for $i=1,2,3$, and sending a point of their common interstice to $\infty$, the remaining disks of $\calP$ accumulate around a point $z_\infty\in \bbC$, as do those of $\tilde \calP$ around a point $\tilde z_\infty\in \bbC$. The points $z_\infty$ and $\tilde z_\infty$ may coincide or may be different. We define and apply $T_\varepsilon$ as before, this time making sure that $z_\infty$ and $\tilde z_\infty$ differ after applying $T_\varepsilon$.

Next, pick small disjoint neighborhoods $W$ and $\tilde W$ of $z_\infty$ and $\tilde z_\infty$ respectively, and contained in $Q$ and $\tilde Q$ respectively. Then, let $V_L$ be the set of vertices $v\in V$ so that both $D_v\subset W$ and $\tilde D_v\subset \tilde W$. Remove vertices from $V_L$ until the sub-triangulation of $X$ having vertices $V\setminus V_L$ is a triangulation of a topological closed disk. Let $F_L$ be the set of faces of $X$ corresponding to interstices formed by disks whose vertices are in $V_L$. Let $L$ (which stands for leftovers) be the union $\bigcup_{v\in V_L} D_v \cup \bigcup_{f\in F_L} U_f\cup z_\infty$, and define $\tilde L$ similarly. Then $\{D_v\}_{v\in V\setminus V_L}$ together with $L$ form a packing of the topological quadrilateral $Q$ by closed Jordan domains, as do $\{\tilde D_v\}_{v\in V\setminus V_L}, \tilde L$ in $\tilde Q$. Furthermore, because $L$ and $\tilde L$ are disjoint by construction, these two domains do not cut each other. Then we get our desired contradiction by the Incompatibility Theorem \ref{incompat} as before.
\end{proof}

\begin{proof}[Proof of Theorem \ref{thm:no luv}]
The adaptation here is along similar lines as the adaptation to prove Theorem \ref{rigid in plane}. Suppose for contradiction that $\calP$ is locally finite in $\bbC$, and $\tilde\calP$ is locally finite in $\bbH^2 \cong \bbD$. This time, after our normalizations, the disks of $\calP$ accumulate around a single point $z_\infty$, and the disks of $\tilde\calP$ accumulate around some round circle $C$ contained in the bounded region in the plane formed between $D_1=\tilde D_1,D_2=\tilde D_1,D_3=\tilde D_3$. This time, ensure that we chose $\varepsilon$ so that $z_\infty$ does not lie on the circle $C$. We define $L$ and $\tilde L$ by throwing away disks of our circle packings, as before, but this time, either $L$ and $\tilde L$ are disjoint, or one contains the other. In either case, the two do not cut each other, and the conclusion of the proof proceeds as before.
\end{proof}

\section{Torus parametrization}
\label{chap:torus}

Before defining torus parametrization, it will be helpful to have access to the following simple lemma:

\begin{lemma}
\label{lem:gen pos lemma}
Suppose $K$ and $\tilde K$ are closed Jordan domains in transverse position. Suppose that $z\in \partial K \cap \partial \tilde K$. Orient $\partial K$ and $\partial \tilde K$ positively as usual. Then one of the following two mutually exclusive possibilities holds at the point $z$.
\begin{enumerate}
\item \label{case:lem:orientations:1} The curve $\partial \tilde K$ is entering $K$, and the curve $\partial K$ is exiting $\tilde K$.
\item \label{case:lem:orientations:2} The curve $\partial K$ is entering $\tilde K$, and the curve $\partial \tilde K$ is exiting $K$.
\end{enumerate}
Thus as we traverse $\partial K$, we alternate arriving at points of $\partial K \cap \partial \tilde K$ where (\ref{case:lem:orientations:1}) occurs and those where (\ref{case:lem:orientations:2}) occurs, and the same holds as we traverse $\partial \tilde K$.
\end{lemma}

\begin{proof}
Let $z\in \partial K \cap \partial \tilde K$. We may assume, by applying a homeomorphism, that locally near $z$ the picture looks like Figure \ref{fig:in out in out}, with $\partial K$ oriented down-to-up as shown. Then $K$ lies to the left of $\partial K$. Now, certainly $\partial \tilde K$ is either entering or exiting $K$ at $z$. Suppose $\partial \tilde K$ is entering $K$ at $z$. Then $\partial \tilde K$ is oriented right-to-left, and so $\tilde K$ is below $\partial \tilde K$. Thus $\partial K$ is exiting $\tilde K$, and case (1) occurs. Similarly, if $\partial \tilde K$ is exiting $K$ at $z$ then $\partial K$ is entering $\tilde K$ at $z$, so case (2) occurs.
\end{proof}

\begin{figure}
\centering
% Generated with LaTeXDraw 2.0.8
% Wed Mar 07 20:49:32 EST 2012
% \usepackage[usenames,dvipsnames]{pstricks}
% \usepackage{epsfig}
% \usepackage{pst-grad} % For gradients
% \usepackage{pst-plot} % For axes
\scalebox{1} % Change this value to rescale the drawing.
{
\begin{pspicture}(0,-1.7067188)(5.8790627,1.7267188)
\pscustom[linewidth=0.02]
{
\newpath
\moveto(2.0,1.5132812)
\lineto(1.95,1.2132813)
\curveto(1.925,1.0632813)(2.0,0.7132813)(2.1,0.5132812)
\curveto(2.2,0.31328124)(2.275,-0.03671875)(2.25,-0.18671875)
\curveto(2.225,-0.33671874)(2.175,-0.6117188)(2.15,-0.7367188)
\curveto(2.125,-0.8617188)(2.15,-1.1117188)(2.2,-1.2367188)
\curveto(2.25,-1.3617188)(2.275,-1.5367187)(2.2,-1.6867187)
}
\psline[linestyle=none,linewidth=0.06]{<-}(1.925,1.0632813)(2.1,0.5132812)
\pscustom[linewidth=0.02,linestyle=dashed,dash=0.16cm 0.16cm]
{
\newpath
\moveto(0.0,-0.18671875)
\lineto(0.35,-0.13671875)
\curveto(0.525,-0.11171875)(0.95,-0.16171876)(1.2,-0.23671874)
\curveto(1.45,-0.31171876)(1.875,-0.33671874)(2.05,-0.28671876)
\curveto(2.225,-0.23671874)(2.575,-0.18671875)(2.75,-0.18671875)
\curveto(2.925,-0.18671875)(3.3,-0.16171876)(3.5,-0.13671875)
\curveto(3.7,-0.11171875)(4.025,-0.08671875)(4.4,-0.08671875)
}
\usefont{T1}{ptm}{m}{n}
\rput(2.3945312,1.5232812){$\partial K$}
\usefont{T1}{ptm}{m}{n}
\rput(4.304531,-0.37671876){$\partial \tilde K$}
\psdots[dotsize=0.17](2.24,-0.24671875)
\usefont{T1}{ptm}{m}{n}
\rput(2.4845312,0.02328125){$z$}
\end{pspicture} 
}

\caption[A meeting point between two Jordan curves in transverse position]
{
\label{fig:in out in out}
{\bf A meeting point between two Jordan curves in transverse position.}
The orientation shown on $\partial K$ implies that $K$ lies to the left.
Depending on the orientation chosen for $\partial \tilde K$
we will get that $\tilde K$ lies above $\partial \tilde K$ or below it.}
\end{figure}
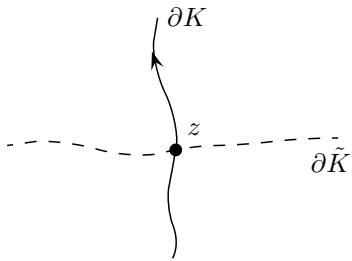

We are now ready to define \emph{torus parametrization}.
Throughout the definition, refer to Figure \ref{tpex1} for an example.

\begin{definition}
Let $K$ and $\tilde K$ be closed Jordan domains in transverse position,
so that $\partial K$
and $\partial \tilde K$
meet at $2M\ge 0$ points,
with boundaries oriented as usual.
Let
$\partial K \cap \partial \tilde K
= \{P_1,\allowbreak \ldots,\allowbreak P_M,\allowbreak \tilde P_1,\allowbreak \ldots,\allowbreak \tilde P_M\}$,
where $P_i$ and $\tilde P_i$
are labeled so that at every $P_i$
we have that $\partial K$
is entering $\tilde K$,
and at every $\tilde P_i$ we have that $\partial \tilde K$
is entering $K$.
Imbue $\bbS^1$
with an orientation and let $\kappa:\partial K\to \bbS^1$
and $\tilde \kappa : \partial\tilde K \to \bbS^1$
be orientation-preserving homeomorphisms.
We refer to this as fixing a \emph{torus parametrization} for $K$ and $\tilde K$.

We consider a point $(x,\tilde x)$ on the 2-torus
$\bbT = \bbS^1\times \bbS^1$ to be parametrizing
simultaneously a point $\kappa\inv(x) \in \partial K$
and a point $\tilde\kappa\inv(\tilde x) \in \partial \tilde K$.
We denote by
$p_i\in \bbT$ the unique point $(x,\tilde x)\in \bbT$
satisfying $\kappa\inv(x) = \tilde\kappa\inv(\tilde x) = P_i$,
similarly $\tilde p_i\in \bbT$.
Note that by the transverse position hypothesis
no pair of points in $\{p_1,\allowbreak \ldots,\allowbreak p_M,\allowbreak \tilde p_1,\allowbreak \ldots,\allowbreak \tilde p_M\}$
share a first coordinate, nor a second coordinate.

Suppose we pick $(x_0,\tilde x_0)\in \bbS^1\times \bbS^1$.
Then we may draw an image of $\bbT = \bbS^1\times \bbS^1$
by letting $\{x_0\}\times \bbS^1$
be the vertical axis and letting $\bbS^1 \times \{\tilde x_0\}$
be the horizontal axis. Then we call $(x_0,\tilde x_0)$
a \emph{base point} for the drawing.
\end{definition}

\begin{figure}
\centering
\subfloat
{
% Generated with LaTeXDraw 2.0.8
% Wed Apr 18 11:36:21 EDT 2012
% \usepackage[usenames,dvipsnames]{pstricks}
% \usepackage{epsfig}
% \usepackage{pst-grad} % For gradients
% \usepackage{pst-plot} % For axes
\scalebox{1} % Change this value to rescale the drawing.
{
\begin{pspicture}(0,-2.4092188)(4.8290625,2.4092188)
\pscustom[linewidth=0.02]
{
\newpath
\moveto(0.11,0.03578125)
\lineto(0.26,-0.31421876)
\curveto(0.335,-0.48921874)(0.46,-0.7892187)(0.51,-0.9142187)
\curveto(0.56,-1.0392188)(0.735,-1.2892188)(0.86,-1.4142188)
\curveto(0.985,-1.5392188)(1.235,-1.7142187)(1.36,-1.7642188)
\curveto(1.485,-1.8142188)(1.71,-1.8892188)(1.81,-1.9142188)
\curveto(1.91,-1.9392188)(2.135,-1.9642187)(2.26,-1.9642187)
\curveto(2.385,-1.9642187)(2.585,-1.9142188)(2.66,-1.8642187)
\curveto(2.735,-1.8142188)(2.835,-1.6642188)(2.86,-1.5642188)
\curveto(2.885,-1.4642187)(2.91,-1.2392187)(2.91,-1.1142187)
\curveto(2.91,-0.9892188)(2.785,-0.7392188)(2.66,-0.6142188)
\curveto(2.535,-0.48921874)(2.26,-0.21421875)(2.11,-0.06421875)
\curveto(1.96,0.08578125)(1.86,0.31078124)(1.91,0.38578126)
\curveto(1.96,0.46078125)(2.135,0.56078124)(2.26,0.5857813)
\curveto(2.385,0.61078125)(2.61,0.7107813)(2.71,0.78578126)
\curveto(2.81,0.86078125)(2.985,1.0857812)(3.06,1.2357812)
\curveto(3.135,1.3857813)(3.16,1.6107812)(3.11,1.6857812)
\curveto(3.06,1.7607813)(2.91,1.8607812)(2.81,1.8857813)
\curveto(2.71,1.9107813)(2.485,1.9607812)(2.36,1.9857812)
\curveto(2.235,2.0107813)(1.985,2.0357811)(1.86,2.0357811)
\curveto(1.735,2.0357811)(1.51,2.0107813)(1.41,1.9857812)
\curveto(1.31,1.9607812)(1.11,1.9107813)(1.01,1.8857813)
\curveto(0.91,1.8607812)(0.71,1.7357812)(0.61,1.6357813)
\curveto(0.51,1.5357813)(0.36,1.3607812)(0.31,1.2857813)
\curveto(0.26,1.2107812)(0.16,1.0357813)(0.11,0.93578124)
\curveto(0.06,0.8357813)(0.01,0.61078125)(0.01,0.48578125)
\curveto(0.01,0.36078125)(0.035,0.18578126)(0.11,0.03578125)
}
\psline[linestyle=none,linewidth=0.06]{->}(0.335,-0.48921874)(0.51,-0.9142187)
\pscustom[linewidth=0.02,linestyle=dashed,dash=0.16cm 0.16cm]
{
\newpath
\moveto(4.11,0.23578125)
\lineto(4.11,0.48578125)
\curveto(4.11,0.61078125)(4.06,0.86078125)(4.01,0.98578125)
\curveto(3.96,1.1107812)(3.86,1.3107812)(3.81,1.3857813)
\curveto(3.76,1.4607812)(3.66,1.6107812)(3.61,1.6857812)
\curveto(3.56,1.7607813)(3.385,1.8607812)(3.26,1.8857813)
\curveto(3.135,1.9107813)(2.91,1.9357812)(2.81,1.9357812)
\curveto(2.71,1.9357812)(2.51,1.7857813)(2.41,1.6357813)
\curveto(2.31,1.4857812)(2.235,1.2107812)(2.26,1.0857812)
\curveto(2.285,0.9607813)(2.335,0.73578125)(2.36,0.6357812)
\curveto(2.385,0.53578126)(2.435,0.33578125)(2.46,0.23578125)
\curveto(2.485,0.13578124)(2.51,-0.08921875)(2.51,-0.21421875)
\curveto(2.51,-0.33921874)(2.435,-0.63921875)(2.36,-0.81421876)
\curveto(2.285,-0.9892188)(2.235,-1.2642188)(2.26,-1.3642187)
\curveto(2.285,-1.4642187)(2.385,-1.6142187)(2.46,-1.6642188)
\curveto(2.535,-1.7142187)(2.685,-1.8142188)(2.76,-1.8642187)
\curveto(2.835,-1.9142188)(3.035,-1.9642187)(3.16,-1.9642187)
\curveto(3.285,-1.9642187)(3.51,-1.9142188)(3.61,-1.8642187)
\curveto(3.71,-1.8142188)(3.86,-1.6892188)(3.91,-1.6142187)
\curveto(3.96,-1.5392188)(4.01,-1.3392187)(4.01,-1.2142187)
\curveto(4.01,-1.0892187)(4.035,-0.83921874)(4.06,-0.71421874)
\curveto(4.085,-0.58921874)(4.11,-0.33921874)(4.11,-0.21421875)
\curveto(4.11,-0.08921875)(4.11,0.08578125)(4.11,0.23578125)
}
\psline[linestyle=none,linewidth=0.06]{->}(4.11,0.61078125)(4.01,0.98578125)
\psline[linestyle=none,linewidth=0.06]{->}(4.11,0.48578125)(4.11,0.61078125)
\usefont{T1}{ptm}{m}{n}
\rput(0.32453126,-1.2942188){$K$}
\usefont{T1}{ptm}{m}{n}
\rput(0.27453125,1.7057812){$u$}
\usefont{T1}{ptm}{m}{n}
\rput(3.9345312,1.8057812){$\tilde K$}
\usefont{T1}{ptm}{m}{n}
\rput(2.0045311,-1.5142188){$\tilde u$}
\usefont{T1}{ptm}{m}{n}
\rput(2.7345312,-2.1742187){$P_1$}
\usefont{T1}{ptm}{m}{n}
\rput(1.9745313,0.80578125){$P_2$}
\usefont{T1}{ptm}{m}{n}
\rput(2.1245313,-0.47421876){$\tilde P_1$}
\usefont{T1}{ptm}{m}{n}
\rput(2.6645312,2.2057812){$\tilde P_2$}
\psdots[dotsize=0.12](0.53,1.5557812)
\psdots[dotsize=0.12](2.29,-1.4642187)
\end{pspicture} 
}
}\qquad
\subfloat
{
% Generated with LaTeXDraw 2.0.8
% Wed Apr 18 11:45:46 EDT 2012
% \usepackage[usenames,dvipsnames]{pstricks}
% \usepackage{epsfig}
% \usepackage{pst-grad} % For gradients
% \usepackage{pst-plot} % For axes
\scalebox{1} % Change this value to rescale the drawing.
{
\begin{pspicture}(0,-2.2775)(7.3690624,2.2375)
\psframe[linewidth=0.02,dimen=middle](7.07,2.2375)(3.07,-1.7625)
\psline[linestyle=none,linewidth=0.06]{->}(3.07,-1.7625)(6,-1.7625)
\psline[linestyle=none,linewidth=0.06]{->}(3.07,2.2375)(6,2.2375)
\psline[linestyle=none,linewidth=0.06]{->}(3.07,-1.7625)(3.07,0)
\psline[linestyle=none,linewidth=0.06]{->}(3.07,-1.7625)(3.07,0.2)
\psline[linestyle=none,linewidth=0.06]{->}(7.07,-1.7625)(7.07,0)
\psline[linestyle=none,linewidth=0.06]{->}(7.07,-1.7625)(7.07,0.2)
\usefont{T1}{ptm}{m}{n}
\rput(3.7745314,-0.9525){$p_1$}
\usefont{T1}{ptm}{m}{n}
\rput(5.2745314,0.5475){$p_2$}
\usefont{T1}{ptm}{m}{n}
\rput(4.284531,1.5475){$\tilde p_1$}
\usefont{T1}{ptm}{m}{n}
\rput(6.284531,-0.4525){$\tilde p_2$}
\psdots[dotsize=0.12](3.57,-1.2625)
\psdots[dotsize=0.12](4.57,1.7375)
\psdots[dotsize=0.12](5.57,0.7375)
\psdots[dotsize=0.12](6.57,-0.2625)
\psdots[dotsize=0.12](3.07,-1.7625)
\usefont{T1}{ptm}{m}{n}
\rput(2.9845312,-2.0525){$(\kappa(u),\tilde\kappa(\tilde u))$}
\end{pspicture} 
}
}
\caption[A pair of closed Jordan domains $K$ and $\tilde K$ and a torus parametrization for them, drawn with base point $(\kappa(u),\tilde\kappa(\tilde u))$]
{
\label{tpex1}
{\bf A pair of closed Jordan domains $K$ and $\tilde K$ and a torus parametrization for them, drawn with base point $(\kappa(u),\tilde\kappa(\tilde u))$.}  The key points to check are that as we vary the first coordinate of $\bbT$ positively starting at $u$, we arrive at $\kappa(P_1)$, $\kappa(\tilde P_1)$, $\kappa(P_2)$, and $\kappa(\tilde P_2)$ in that order, and as we vary the second coordinate of $\bbT$ positively starting at $\tilde\kappa(\tilde u)$, we arrive at $\tilde\kappa(P_1)$, $\tilde \kappa(\tilde P_2)$, $\tilde\kappa(P_2)$, and $\tilde\kappa(\tilde P_1)$ in that order.
}
\end{figure}
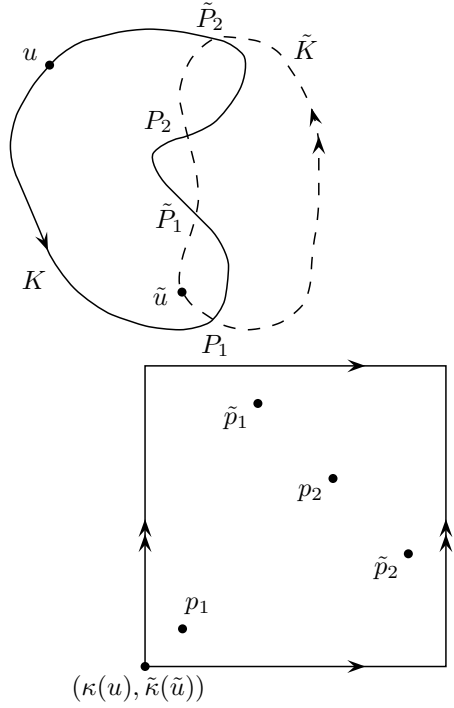

Suppose that $\phi:\partial K \to \partial \tilde K$ is an orientation-preserving homeomorphism. Then $\phi$ determines an oriented curve $\gamma$ in $\bbT$ for us, namely its graph $\gamma = \{(\kappa(z),\tilde\kappa(\phi(z))\}_{z\in \partial K}$, with orientation obtained by traversing $\partial K$ positively. Note that $\phi$ is fixed-point-free if and only if its associated curve $\gamma$ misses all of the $p_i$ and $\tilde p_i$. Pick $u\in \partial K$ and denote $\tilde u = \phi(u)$. Then if we draw the torus parametrization for $K$ and $\tilde K$ using the base point $(\kappa(u),\tilde\kappa(\tilde u))$, the curve $\gamma$ associated to $\phi$ essentially looks like the graph of a strictly increasing function from a closed interval to itself. The converse is also true: given any such $\gamma$, it determines for us an orientation-preserving homeomorphism $\partial K\to \partial \tilde K$ sending $u$ to $\tilde u$, which is fixed-point-free if and only if $\gamma$ misses all of the $p_i$ and $\tilde p_i$.

\newcommand{\Num}{\operatorname{Num}}

\newcommand{\nump}{\#p}
\newcommand{\numq}{\#\tilde p}

Suppose that $\phi(u) = \tilde u$, equivalently that $(\kappa(u),\tilde\kappa(\tilde u)) \in \gamma$. The curve $\gamma$ and the horizontal and vertical axes $\{\tilde\kappa(\tilde u)\}\times \bbS^1$ and $\bbS^1\times \{\kappa(u)\}$ divide $\bbT$ into two simply connected open sets $\Delta_\uparrow(u,\gamma)$ and $\Delta_\downarrow(u,\gamma)$ as shown in Figure \ref{fig:proving amazing index}. We suppress the dependence on $\tilde u$ in the notation because $\tilde u = \phi(u)$. If neither $u\in \partial \tilde K$ nor $\tilde u\in \partial K$ then every $p_i$ and every $\tilde p_i$ lies in either $\Delta_\downarrow(u,\gamma)$ or $\Delta_\uparrow(u,\gamma)$. In this case we write $\nump_\downarrow(u,\gamma)$ to denote $|\{p_1,\ldots,p_M\} \cap \Delta_\downarrow(u,\gamma)|$ the number of points $p_i$ which lie in $\Delta_\downarrow(u,\gamma)$, and we define $\nump_\uparrow(u,\gamma)$, $\numq_\downarrow(u,\gamma)$, and $\numq_\uparrow(u,\gamma)$ in the analogous way. Denote by $\omega(\alpha,z)$ the winding number of the closed curve $\alpha\subset \bbC$ around the point $z\not\in\alpha$.\medskip

Suppose $\gamma_0$ is any oriented closed curve in $\bbT \setminus \{p_1,\allowbreak \ldots,\allowbreak p_M,\allowbreak \tilde p_1,\allowbreak \ldots,\allowbreak \tilde p_M\}$. Then the closed curve $\{\tilde\kappa\inv(\tilde x) - \kappa\inv(x)\}_{(x,\tilde x)\in \gamma_0}$ misses the origin, and has a natural orientation obtained from that of $\gamma_0$. We denote by $w(\gamma_0)$ the winding number around the origin of $\{\tilde\kappa\inv(\tilde x) - \kappa\inv(x)\}_{(x,\tilde x)\in \gamma_0}$.

The following central lemma says that given an indexable homeomorphism $\phi: \partial K \to \partial \tilde K$, we may read off its fixed-point index $\eta(\phi)$ simply by examining the curve $\gamma$ associated to $\phi$ in the way we just described:

\begin{lemma}
\label{prop:computing index from torus}
Let $K$ and $\tilde K$ be closed Jordan domains. Fix a torus parametrization of $K$ and $\tilde K$ via $\kappa$ and $\tilde\kappa$. Let $\phi:\partial K \to \partial \tilde K$ be an indexable homeomorphism, with graph $\gamma$ in $\bbT$. Suppose that $\phi(u) = \tilde u$, where $u\not\in \partial \tilde K$ and $\tilde u\not\in \partial K$. Then:
\begin{align}
\label{eq 1:prop:computing index from torus eq}\eta(\phi)  = w(\gamma)& = \omega(\partial K, \tilde u) + \omega(\partial \tilde K, u ) - \nump_\downarrow(u,\gamma) + \numq_\downarrow(u,\gamma)\\
\label{eq 2:prop:computing index from torus eq}& = \omega(\partial K, \tilde u) + \omega(\partial \tilde K, u) + \nump_\uparrow(u,\gamma) - \numq_\uparrow(u,\gamma)
\end{align}
\end{lemma}

\noindent The remainder of the section is spent proving Lemma \ref{prop:computing index from torus}.\medskip

We begin with an observation:

\begin{observation}
\label{lem:homotopy classes}
If $\gamma_1$ and $\gamma_2$ are homotopic in $\bbT \setminus \{p_1,\allowbreak \ldots,\allowbreak p_M,\allowbreak \tilde p_1,\allowbreak \ldots,\allowbreak \tilde p_M\}$ then $w(\gamma_1) = w(\gamma_2)$.
\end{observation}

\noindent This is because the homotopy between $\gamma_1$ and $\gamma_2$ in $\bbT\setminus \{p_1,\allowbreak \ldots,\allowbreak p_M,\allowbreak \tilde p_1,\allowbreak \ldots,\allowbreak \tilde p_M\}$ induces a homotopy between the closed curves $\{\tilde\kappa\inv(\tilde x) - \kappa\inv(x)\}_{(x,\tilde x)\in \gamma_1}$ and $\{\tilde\kappa\inv(\tilde x) - \kappa\inv(x)\}_{(x,\tilde x)\in \gamma_2}$ in the punctured plane $\bbC\setminus \{0\}$.

Suppose that $\phi:\partial K\to \partial \tilde K$ is a fixed-point-free orientation-preserving homeomorphism. Let $\gamma$ be its graph in $\bbT$. If $\gamma$ has orientation induced by traversing $\partial K$ and $\partial \tilde K$ positively, then the following is a tautology.

\begin{observation}
\label{ex:never gonna reference this lol}
$\eta(\phi) = w(\gamma)$
\end{observation}

\begin{figure}
\centering

% Generated with LaTeXDraw 2.0.8
% Generated with LaTeXDraw 2.0.8
% Sun Dec 04 17:41:05 EST 2011
% \usepackage[usenames,dvipsnames]{pstricks}
% \usepackage{epsfig}
% \usepackage{pst-grad} % For gradients
% \usepackage{pst-plot} % For axes
\scalebox{1} % Change this value to rescale the drawing.
{
\begin{pspicture}(0,-3.2275)(9.279062,3.1875)
\psframe[linewidth=0.02,dimen=middle](5.82,3.1875)(0.0,-2.6325)

\psdots[dotsize=0.12](0.0,-2.6325)
\rput(0.0,-3.0025){$(\kappa(u),\tilde\kappa(\tilde u))$}

%\psline[linewidth=0.06cm,linestyle=dashed,dash=0.16cm 0.16cm](0.0,-2.6325)(5.82,3.1875)

\pscustom[linewidth=0.02,linestyle=dashed,dash=0.16cm 0.16cm]
{
\newpath
\moveto(0.0,-2.6325)
\lineto(0.08,-2.4725)
\curveto(0.11,-2.4225)(0.165,-2.3375)(0.19,-2.3025)
\curveto(0.215,-2.2675)(0.28,-2.1925)(0.32,-2.1525)
\curveto(0.36,-2.1125)(0.445,-2.0425)(0.49,-2.0125)
\curveto(0.535,-1.9825)(0.65,-1.8825)(0.72,-1.8125)
\curveto(0.79,-1.7425)(0.895,-1.6275)(0.93,-1.5825)
\curveto(0.965,-1.5375)(1.035,-1.4325)(1.07,-1.3725)
\curveto(1.105,-1.3125)(1.19,-1.1925)(1.24,-1.1325)
\curveto(1.29,-1.0725)(1.38,-0.9675)(1.42,-0.9225)
\curveto(1.46,-0.8775)(1.565,-0.7775)(1.63,-0.7225)
\curveto(1.695,-0.6675)(1.825,-0.5675)(1.89,-0.5225)
\curveto(1.955,-0.4775)(2.075,-0.4025)(2.13,-0.3725)
\curveto(2.185,-0.3425)(2.305,-0.2875)(2.37,-0.2625)
\curveto(2.435,-0.2375)(2.56,-0.1925)(2.62,-0.1725)
\curveto(2.68,-0.1525)(2.8,-0.0975)(2.86,-0.0625)
\curveto(2.92,-0.0275)(3.035,0.0525)(3.09,0.0975)
\curveto(3.145,0.1425)(3.235,0.2325)(3.27,0.2775)
\curveto(3.305,0.3225)(3.37,0.4125)(3.4,0.4575)
\curveto(3.43,0.5025)(3.48,0.5875)(3.5,0.6275)
\curveto(3.52,0.6675)(3.555,0.7625)(3.57,0.8175)
\curveto(3.585,0.8725)(3.62,0.9775)(3.64,1.0275)
\curveto(3.66,1.0775)(3.7,1.1625)(3.72,1.1975)
\curveto(3.74,1.2325)(3.795,1.3075)(3.83,1.3475)
\curveto(3.865,1.3875)(3.945,1.4675)(3.99,1.5075)
\curveto(4.035,1.5475)(4.125,1.6075)(4.17,1.6275)
\curveto(4.215,1.6475)(4.325,1.6925)(4.39,1.7175)
\curveto(4.455,1.7425)(4.59,1.7825)(4.66,1.7975)
\curveto(4.73,1.8125)(4.84,1.8425)(4.88,1.8575)
\curveto(4.92,1.8725)(5.015,1.9175)(5.07,1.9475)
\curveto(5.125,1.9775)(5.215,2.0425)(5.25,2.0775)
\curveto(5.285,2.1125)(5.35,2.1925)(5.38,2.2375)
\curveto(5.41,2.2825)(5.47,2.3875)(5.5,2.4475)
\curveto(5.53,2.5075)(5.57,2.6025)(5.58,2.6375)
\curveto(5.59,2.6725)(5.61,2.7375)(5.62,2.7675)
\curveto(5.63,2.7975)(5.655,2.8475)(5.67,2.8675)
\curveto(5.685,2.8875)(5.705,2.9275)(5.71,2.9475)
\curveto(5.715,2.9675)(5.73,3.0025)(5.74,3.0175)
\curveto(5.75,3.0325)(5.765,3.0625)(5.82,3.1875)
}
\usefont{T1}{ptm}{m}{n}
\rput(4.0445313,2.1175){$\Delta_\uparrow(u,\gamma)$}
\usefont{T1}{ptm}{m}{n}
\rput(4.7945313,1.2975){$\Delta_\downarrow(u,\gamma)$}
\usefont{T1}{ptm}{m}{n}
\rput(1.3845313,-0.3825){$\gamma$}

\psline[linewidth=0.02cm,linestyle=dotted,dotsep=0.16cm,arrowsize=0.05291667cm 2.0,arrowlength=1.4,arrowinset=0.4]{->}(5.7,0.8475)(5.0,0.0675)

\psline[linewidth=0.1cm,linestyle=none,arrows=<-](2.68,-0.1525)(2.86,-0.0625)

\psline[linewidth=0.06cm,linestyle=dashed,dash=0.16cm 0.16cm](0.0,-2.6325)(5.82,-2.6325)
\psline[linewidth=0.1cm,linestyle=none,arrows=->](0.0,-2.6325)(1.82,-2.6325)

\psline[linewidth=0.06cm,linestyle=dashed,dash=0.16cm 0.16cm](5.82,-2.6325)(5.82,3.1875)
\psline[linewidth=0.1cm,linestyle=none,arrows=->](5.82,-2.6325)(5.82,0.1875)

\psdots[dotsize=0.12](0.84,0.9475)
\psdots[dotsize=0.12](2.18,2.0075)
\psdots[dotsize=0.12](2.64,-1.8325)
\psdots[dotsize=0.12](4.84,-0.4925)
\usefont{T1}{ptm}{m}{n}
\rput(3.2245312,-3.0025){$\bbS^1\times \{\tilde\kappa(\tilde u)\}$}
\usefont{T1}{ptm}{m}{n}
\rput(7.044531,0.5775){$\{\kappa(u)\}\times \bbS^1$}
\usefont{T1}{ptm}{m}{n}
%\rput(2.7645313,0.5975){$\Delta$}
\usefont{T1}{ptm}{m}{n}
\rput(0.8045313,1.3175){$p_1$}
\usefont{T1}{ptm}{m}{n}
\rput(2.1645312,2.3975){$\tilde p_1$}
\usefont{T1}{ptm}{m}{n}
\rput(2.9245312,-1.8025){$p_2$}
\usefont{T1}{ptm}{m}{n}
\rput(4.5645313,-0.4825){$\tilde p_2$}

\psframe[linewidth=0.06,linestyle=dashed,dash=0.16cm 0.16cm,dimen=middle](5.02,-0.0525)(4.26,-0.8125)

\psline[linewidth=0.1cm,linestyle=none,arrows=->](5.02,-0.0525)(4.46,-0.0525)

\psframe[linewidth=0.06,linestyle=dashed,dash=0.16cm 0.16cm,dimen=middle](3.18,-1.3925)(2.38,-2.1925)

\psline[linewidth=0.1cm,linestyle=none,arrows=<-](2.98,-2.1925)(2.38,-2.1925)

\psline[linewidth=0.06cm,linestyle=dashed,dash=0.16cm 0.16cm](3.18,-1.3925)(4.26,-0.8125)

\psline[linewidth=0.03cm,linestyle=solid,arrows=->](3.38,-1.3925)(4.16,-0.9625)

\psline[linewidth=0.03cm,linestyle=solid,arrows=<-](3.28,-1.2425)(4.06,-0.8125)

\psline[linewidth=0.02cm,linestyle=dotted,dotsep=0.16cm,arrowsize=0.01291667cm 2.0,arrowlength=1.4,arrowinset=0.4]{->}(0.8,-2.3325)(2.08,-2.1125)
\psline[linewidth=0.02cm,linestyle=dotted,dotsep=0.16cm,arrowsize=0.01291667cm 2.0,arrowlength=1.4,arrowinset=0.4]{->}(2.66,-0.5925)(2.98,-1.1925)
\psline[linewidth=0.02cm,linestyle=dotted,dotsep=0.16cm,arrowsize=0.01291667cm 2.0,arrowlength=1.4,arrowinset=0.4]{->}(3.74,0.6675)(4.22,0.0475)
%\psline[linewidth=0.02cm,linestyle=dotted,dotsep=0.16cm,arrowsize=0.01291667cm 2.0,arrowlength=1.4,arrowinset=0.4]{->}(5.28,1.8675)(5.0,0.1875)
\psline[linewidth=0.02cm,linestyle=dotted,dotsep=0.16cm,arrowsize=0.01291667cm 2.0,arrowlength=1.4,arrowinset=0.4]{->}(5.58,-1.1525)(5.16,-0.8525)
\psline[linewidth=0.02cm,linestyle=dotted,dotsep=0.16cm,arrowsize=0.01291667cm 2.0,arrowlength=1.4,arrowinset=0.4]{->}(4.26,-2.4125)(3.72,-1.4125)
\psline[linewidth=0.02cm,linestyle=dotted,dotsep=0.16cm,arrowsize=0.01291667cm 2.0,arrowlength=1.4,arrowinset=0.4]{->}(3.34,-2.5125)(3.24,-2.1925)
\psline[linewidth=0.02cm,linestyle=dotted,dotsep=0.16cm,arrowsize=0.01291667cm 2.0,arrowlength=1.4,arrowinset=0.4]{->}(1.54,-1.3325)(2.3,-1.4325)
\psline[linewidth=0.02cm,linestyle=dotted,dotsep=0.16cm,arrowsize=0.01291667cm 2.0,arrowlength=1.4,arrowinset=0.4]{->}(5.58,-2.4525)(4.38,-1.1325)
\psline[linewidth=0.02cm,linestyle=dotted,dotsep=0.16cm,arrowsize=0.01291667cm 2.0,arrowlength=1.4,arrowinset=0.4]{->}(3.16,0.1275)(3.84,-0.8725)
\end{pspicture} 
}

\caption[A homotopy from $\partial\Delta_\downarrow(u,\gamma)$ to $\Gamma$]
{
\label{fig:proving amazing index}
{\bf A homotopy from $\partial\Delta_\downarrow(u,\gamma)$ to $\Gamma$.}  Here the orientation shown on $\gamma$ is the opposite of the orientation induced by traversing $\partial K$ positively.
}
\end{figure}
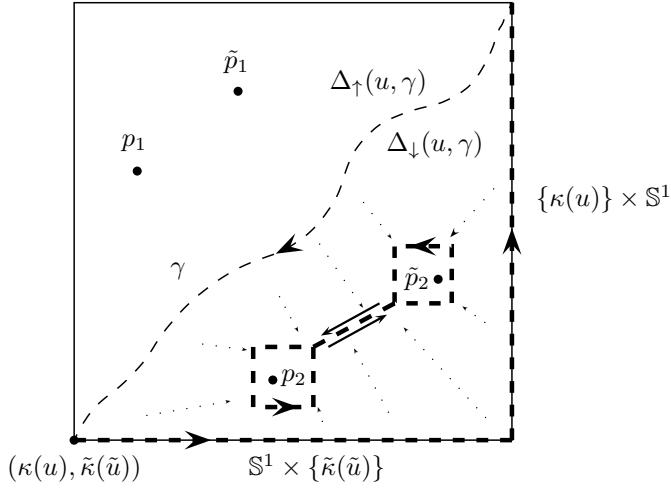

Orient $\partial \Delta_\downarrow(u,\gamma)$ as shown in Figure \ref{fig:proving amazing index}. Then $\partial \Delta_\downarrow(u,\gamma)$ is the concatenation of the curve $\gamma$ traversed backwards with $\bbS^1\times \{\tilde\kappa(\tilde u)\}$ and $\{\kappa(u)\}\times \bbS^1$, where the two latter curves are oriented according to the positive orientation on $\bbS^1$.

\begin{observation}
If $\bbS^1\times \{\tilde\kappa(\tilde u)\}$ and $\{\kappa(u)\}\times \bbS^1$ are oriented according to the positive orientation on $\bbS^1$, then $w(\bbS^1\times \{\tilde\kappa(\tilde u)\}) = \omega(\partial K, \tilde u)$ and $w(\{\kappa(u)\}\times \bbS^1) = \omega(\partial \tilde K, u)$.
\end{observation}
\noindent It is also easy to see that if we concatenate two closed curves $\gamma_1$ and $\gamma_2$ that meet at a point, we get $w(\gamma_1\circ \gamma_2) = w(\gamma_1) + w(\gamma_2)$. Thus in light of the orientations on $\partial \Delta_\downarrow(u,\gamma)$ and all other curves concerned we get:
\begin{align*}
\label{prop:computing index from torus eq}
w(\partial \Delta_\downarrow(u,\gamma))& = w(\bbS^1\times \{\tilde\kappa(\tilde u)\}) + w(\{\kappa(u)\}\times \bbS^1) - w(\gamma) \\
& = \omega(\partial K,\tilde u) +  \omega(\partial \tilde K, u) - \eta(\phi)
\end{align*}

For every $i$ let $\zeta(p_i)$ and $\zeta(\tilde p_i)$ be small squares around $p_i$ and $\tilde p_i$ respectively in $\bbT$, oriented as shown in Figure \ref{fig:proving amazing index}. By \emph{square} we mean a simple closed curve which decomposes into four ``sides,'' so that on a given side one of the two coordinates of $\bbS^1\times \bbS^1 = \bbT$ is constant. Pick the squares small enough so that the closed boxes they bound are pairwise disjoint and do not meet $\partial \Delta_\downarrow(u,\gamma)$.

Let $\Gamma$ be the closed curve in $\Delta_\downarrow(u,\gamma)$ obtained in the following way. First, start with every loop $\zeta(p_i)$ and $\zeta(\tilde p_j)$ for those $p_i$ and $\tilde p_j$ lying in $\Delta_\downarrow(u,\gamma)$. Let $\delta_0$ be an arc contained in the interior of $\Delta_\downarrow(u,\gamma)$ which meets each $\zeta(p_i)$ and $\zeta(\tilde p_j)$ contained in $\Delta_\downarrow(u,\gamma)$ at exactly one point. It is easy to prove inductively that such an arc exists. Let $\delta$ be the closed curve obtained by traversing $\delta_0$ first in one direction, then in the other. Then let $\Gamma$ be obtained by concatenating $\delta$ with every $\zeta(p_i)$ and $\zeta(\tilde p_j)$ contained in $\Delta_\downarrow(u,\gamma)$.
\begin{observation}
The curves $\Gamma$ and $\partial \Delta_\downarrow(u,\gamma)$ are homotopic in $\bbT\setminus \{p_1,\ldots,\allowbreak p_M,\allowbreak\tilde p_1,\ldots,\allowbreak\tilde p_M\}$. Also $w(\delta) = 0$. It follows that:
\[
w(\partial \Delta_\downarrow(u,\gamma)) = w(\Gamma) = \sum_{p_i\in \Delta_\downarrow(u,\gamma)} w(\zeta(p_i)) + \sum_{\tilde p_j\in \Delta_\downarrow(u,\gamma)} w(\zeta(\tilde p_j))
\]
\end{observation}

\noindent See Figure \ref{fig:proving amazing index} for an example. On the other hand, the following holds.

\begin{figure}
\centering

% Generated with LaTeXDraw 2.0.8
% Fri Dec 02 14:45:40 EST 2011
% \usepackage[usenames,dvipsnames]{pstricks}
% \usepackage{epsfig}
% \usepackage{pst-grad} % For gradients
% \usepackage{pst-plot} % For axes
\scalebox{.9} % Change this value to rescale the drawing.
{
\begin{pspicture}(0,-3.953125)(9.065,3.953125)
\definecolor{color308b}{rgb}{0.9215686274509803,0.9215686274509803,0.9215686274509803}
\pscustom[linestyle=none,fillstyle=solid,fillcolor = color308b]
{
\newpath
\moveto(1.69,2.5796876)
\lineto(1.89,2.4196875)
\curveto(1.99,2.3396876)(2.28,2.0546875)(2.47,1.8496875)
\curveto(2.66,1.6446875)(2.89,1.3196875)(2.93,1.1996875)
\curveto(2.97,1.0796875)(3.09,0.8246875)(3.17,0.6896875)
\curveto(3.25,0.5546875)(3.44,0.3296875)(3.55,0.2396875)
\curveto(3.66,0.1496875)(3.845,-0.0403125)(3.92,-0.1403125)
\curveto(3.995,-0.2403125)(4.145,-0.4453125)(4.22,-0.5503125)
\curveto(4.295,-0.6553125)(4.465,-0.8553125)(4.56,-0.9503125)
\curveto(4.655,-1.0453125)(4.82,-1.1903125)(4.89,-1.2403125)
\curveto(4.96,-1.2903125)(5.13,-1.4053125)(5.23,-1.4703125)
\curveto(5.33,-1.5353125)(5.48,-1.6753125)(5.53,-1.7503124)
\curveto(5.58,-1.8253125)(5.665,-1.9453125)(5.7,-1.9903125)
\curveto(5.735,-2.0353124)(5.81,-2.1303124)(5.85,-2.1803124)
\curveto(5.89,-2.2303126)(5.97,-2.3103125)(6.01,-2.3403125)
\curveto(6.05,-2.3703125)(6.13,-2.4153125)(6.17,-2.4303124)
\curveto(6.21,-2.4453125)(6.295,-2.4603126)(6.34,-2.4603126)
\curveto(6.385,-2.4603126)(6.5,-2.4353125)(6.57,-2.4103124)
\curveto(6.64,-2.3853126)(6.82,-2.2553124)(6.93,-2.1503124)
\curveto(7.04,-2.0453124)(7.235,-1.8803124)(7.32,-1.8203125)
\curveto(7.405,-1.7603126)(7.635,-1.5903125)(7.78,-1.4803125)
\curveto(7.925,-1.3703125)(8.155,-1.1853125)(8.24,-1.1103125)
\curveto(8.325,-1.0353125)(8.49,-0.8753125)(8.57,-0.7903125)
\curveto(8.65,-0.7053125)(8.765,-0.5353125)(8.8,-0.4503125)
\curveto(8.835,-0.3653125)(8.88,-0.2203125)(8.89,-0.1603125)
\curveto(8.9,-0.1003125)(8.93,0.0446875)(8.95,0.1296875)
\curveto(8.97,0.2146875)(9.005,0.3746875)(9.02,0.4496875)
\curveto(9.035,0.5246875)(9.045,0.7096875)(9.04,0.8196875)
\curveto(9.035,0.9296875)(8.98,1.1296875)(8.93,1.2196875)
\curveto(8.88,1.3096875)(8.725,1.5196875)(8.62,1.6396875)
\curveto(8.515,1.7596875)(8.305,2.0246875)(8.2,2.1696875)
\curveto(8.095,2.3146875)(7.9,2.5396874)(7.81,2.6196876)
\curveto(7.72,2.6996875)(7.445,2.8696876)(7.26,2.9596875)
\curveto(7.075,3.0496874)(6.76,3.1896875)(6.63,3.2396874)
\curveto(6.5,3.2896874)(6.09,3.3846874)(5.81,3.4296875)
\curveto(5.53,3.4746876)(5.115,3.5346875)(4.98,3.5496874)
\curveto(4.845,3.5646875)(4.53,3.5846875)(4.35,3.5896876)
\curveto(4.17,3.5946875)(3.845,3.5596876)(3.7,3.5196874)
\curveto(3.555,3.4796875)(3.275,3.3746874)(3.14,3.3096876)
\curveto(3.005,3.2446876)(2.79,3.1446874)(2.71,3.1096876)
\curveto(2.63,3.0746875)(2.485,3.0146875)(2.42,2.9896874)
\curveto(2.355,2.9646876)(2.195,2.8796875)(2.1,2.8196876)
\curveto(2.005,2.7596874)(1.89,2.6896875)(1.87,2.6796875)
\curveto(1.85,2.6696875)(1.825,2.6396875)(1.82,2.6196876)
\curveto(1.815,2.5996876)(1.8,2.5646875)(1.79,2.5496874)
\curveto(1.78,2.5346875)(1.77,2.5146875)(1.77,2.4996874)
}

\pscustom[linestyle=none,fillstyle=hlines]
{
\newpath
\moveto(1.47,-1.4403125)
\lineto(1.61,-1.4003125)
\curveto(1.68,-1.3803124)(1.92,-1.2903125)(2.09,-1.2203125)
\curveto(2.26,-1.1503125)(2.49,-1.0453125)(2.55,-1.0103126)
\curveto(2.61,-0.9753125)(2.74,-0.8903125)(2.81,-0.8403125)
\curveto(2.88,-0.7903125)(3.02,-0.6653125)(3.09,-0.5903125)
\curveto(3.16,-0.5153125)(3.295,-0.3603125)(3.36,-0.2803125)
\curveto(3.425,-0.2003125)(3.58,-0.0753125)(3.67,-0.0303125)
\curveto(3.76,0.0146875)(3.965,0.0946875)(4.08,0.1296875)
\curveto(4.195,0.1646875)(4.455,0.2946875)(4.6,0.3896875)
\curveto(4.745,0.4846875)(4.96,0.6396875)(5.03,0.6996875)
\curveto(5.1,0.7596875)(5.24,0.8746875)(5.31,0.9296875)
\curveto(5.38,0.9846875)(5.52,1.1046875)(5.59,1.1696875)
\curveto(5.66,1.2346874)(5.79,1.3496875)(5.85,1.3996875)
\curveto(5.91,1.4496875)(6.055,1.5446875)(6.14,1.5896875)
\curveto(6.225,1.6346875)(6.37,1.7096875)(6.43,1.7396874)
\curveto(6.49,1.7696875)(6.615,1.8246875)(6.68,1.8496875)
\curveto(6.745,1.8746876)(6.86,1.9096875)(6.91,1.9196875)
\curveto(6.96,1.9296875)(7.07,1.9396875)(7.13,1.9396875)
\curveto(7.19,1.9396875)(7.3,1.9396875)(7.35,1.9396875)
\curveto(7.4,1.9396875)(7.5,1.9246875)(7.55,1.9096875)
\curveto(7.6,1.8946875)(7.68,1.8596874)(7.71,1.8396875)
\curveto(7.74,1.8196875)(7.835,1.6996875)(7.9,1.5996875)
\curveto(7.965,1.4996876)(8.115,1.1346875)(8.2,0.8696875)
\curveto(8.285,0.6046875)(8.45,0.2146875)(8.53,0.0896875)
\curveto(8.61,-0.0353125)(8.705,-0.2703125)(8.72,-0.3803125)
\curveto(8.735,-0.4903125)(8.695,-0.7453125)(8.64,-0.8903125)
\curveto(8.585,-1.0353125)(8.41,-1.4003125)(8.29,-1.6203125)
\curveto(8.17,-1.8403125)(7.975,-2.1903124)(7.9,-2.3203125)
\curveto(7.825,-2.4503126)(7.7,-2.6303124)(7.65,-2.6803124)
\curveto(7.6,-2.7303126)(7.43,-2.8553126)(7.31,-2.9303124)
\curveto(7.19,-3.0053124)(6.87,-3.1503124)(6.67,-3.2203126)
\curveto(6.47,-3.2903125)(6.115,-3.3803124)(5.96,-3.4003124)
\curveto(5.805,-3.4203124)(5.51,-3.4453125)(5.37,-3.4503126)
\curveto(5.23,-3.4553125)(4.995,-3.4903126)(4.9,-3.5203125)
\curveto(4.805,-3.5503125)(4.575,-3.5853126)(4.44,-3.5903125)
\curveto(4.305,-3.5953126)(4.07,-3.6003125)(3.97,-3.6003125)
\curveto(3.87,-3.6003125)(3.57,-3.5553124)(3.37,-3.5103126)
\curveto(3.17,-3.4653125)(2.845,-3.3803124)(2.72,-3.3403125)
\curveto(2.595,-3.3003125)(2.34,-3.1453125)(2.21,-3.0303125)
\curveto(2.08,-2.9153125)(1.92,-2.7403126)(1.89,-2.6803124)
\curveto(1.86,-2.6203125)(1.815,-2.5003126)(1.8,-2.4403124)
\curveto(1.785,-2.3803124)(1.745,-2.2553124)(1.72,-2.1903124)
\curveto(1.695,-2.1253126)(1.65,-2.0153124)(1.63,-1.9703125)
\curveto(1.61,-1.9253125)(1.57,-1.8353125)(1.55,-1.7903125)
\curveto(1.53,-1.7453125)(1.505,-1.6703125)(1.5,-1.6403126)
\curveto(1.495,-1.6103125)(1.485,-1.5553125)(1.48,-1.5303125)
\curveto(1.475,-1.5053124)(1.48,-1.4603125)(1.49,-1.4403125)
}

\pscustom[linewidth=0.02]
{
\newpath
\moveto(1.47,-1.4403125)
\lineto(1.61,-1.4003125)
\curveto(1.68,-1.3803124)(1.92,-1.2903125)(2.09,-1.2203125)
\curveto(2.26,-1.1503125)(2.49,-1.0453125)(2.55,-1.0103126)
\curveto(2.61,-0.9753125)(2.74,-0.8903125)(2.81,-0.8403125)
\curveto(2.88,-0.7903125)(3.02,-0.6653125)(3.09,-0.5903125)
\curveto(3.16,-0.5153125)(3.295,-0.3603125)(3.36,-0.2803125)
\curveto(3.425,-0.2003125)(3.58,-0.0753125)(3.67,-0.0303125)
\curveto(3.76,0.0146875)(3.965,0.0946875)(4.08,0.1296875)
\curveto(4.195,0.1646875)(4.455,0.2946875)(4.6,0.3896875)
\curveto(4.745,0.4846875)(4.96,0.6396875)(5.03,0.6996875)
\curveto(5.1,0.7596875)(5.24,0.8746875)(5.31,0.9296875)
\curveto(5.38,0.9846875)(5.52,1.1046875)(5.59,1.1696875)
\curveto(5.66,1.2346874)(5.79,1.3496875)(5.85,1.3996875)
\curveto(5.91,1.4496875)(6.055,1.5446875)(6.14,1.5896875)
\curveto(6.225,1.6346875)(6.37,1.7096875)(6.43,1.7396874)
\curveto(6.49,1.7696875)(6.615,1.8246875)(6.68,1.8496875)
\curveto(6.745,1.8746876)(6.86,1.9096875)(6.91,1.9196875)
\curveto(6.96,1.9296875)(7.07,1.9396875)(7.13,1.9396875)
\curveto(7.19,1.9396875)(7.3,1.9396875)(7.35,1.9396875)
\curveto(7.4,1.9396875)(7.5,1.9246875)(7.55,1.9096875)
}

\psline[linestyle=none,linewidth=0.1,arrows=<-](5.1,0.7596875)(5.31,0.9296875)

\pscustom[linewidth=0.02]
{
\newpath
\moveto(1.69,2.5796876)
\lineto(1.89,2.4196875)
\curveto(1.99,2.3396876)(2.28,2.0546875)(2.47,1.8496875)
\curveto(2.66,1.6446875)(2.89,1.3196875)(2.93,1.1996875)
\curveto(2.97,1.0796875)(3.09,0.8246875)(3.17,0.6896875)
\curveto(3.25,0.5546875)(3.44,0.3296875)(3.55,0.2396875)
\curveto(3.66,0.1496875)(3.845,-0.0403125)(3.92,-0.1403125)
\curveto(3.995,-0.2403125)(4.145,-0.4453125)(4.22,-0.5503125)
\curveto(4.295,-0.6553125)(4.465,-0.8553125)(4.56,-0.9503125)
\curveto(4.655,-1.0453125)(4.82,-1.1903125)(4.89,-1.2403125)
\curveto(4.96,-1.2903125)(5.13,-1.4053125)(5.23,-1.4703125)
\curveto(5.33,-1.5353125)(5.48,-1.6753125)(5.53,-1.7503124)
\curveto(5.58,-1.8253125)(5.665,-1.9453125)(5.7,-1.9903125)
\curveto(5.735,-2.0353124)(5.81,-2.1303124)(5.85,-2.1803124)
\curveto(5.89,-2.2303126)(5.97,-2.3103125)(6.01,-2.3403125)
\curveto(6.05,-2.3703125)(6.13,-2.4153125)(6.17,-2.4303124)
\curveto(6.21,-2.4453125)(6.295,-2.4603126)(6.34,-2.4603126)
}

\psline[linewidth=0.1,linestyle=none,arrows=->](2.97,1.0796875)(3.17,0.6896875)

\usefont{T1}{ptm}{m}{n}
\rput(5.5845313,3.7496874){$K$}
\usefont{T1}{ptm}{m}{n}
\rput(5.7145314,-3.7503126){$\tilde K$}
\psdots[dotsize=0.12](3.77,0.0196875)
\usefont{T1}{ptm}{m}{n}
\rput(3.1945312,0.0496875){$P_i$}
\psdots[dotsize=0.12](2.55,1.7796875)
\psdots[dotsize=0.12](5.25,-1.4603125)
\psdots[dotsize=0.12](2.31,-1.1403126)
\psdots[dotsize=0.12](6.17,1.5996875)
\usefont{T1}{ptm}{m}{n}
\rput(1.8545312,1.6096874){$\kappa\inv(x_0)$}
\usefont{T1}{ptm}{m}{n}
\rput(11.174531,-1.2303125){$\kappa\inv(x_1)$}
\usefont{T1}{ptm}{m}{n}
\rput(5.9045315,2.0296875){$\tilde\kappa\inv(\tilde x_0)$}
\usefont{T1}{ptm}{m}{n}
\rput(1.7045312,-0.7903125){$\tilde \kappa\inv(\tilde x_1)$}

\pscustom[linewidth=0.01]
{
\newpath
\moveto(5.41,-1.3603125)
\lineto(5.58,-1.2303125)
\curveto(5.665,-1.1653125)(5.835,-1.0853125)(5.92,-1.0703125)
\curveto(6.005,-1.0553125)(6.185,-1.0403125)(6.28,-1.0403125)
\curveto(6.375,-1.0403125)(6.61,-1.0603125)(6.75,-1.0803125)
\curveto(6.89,-1.1003125)(7.21,-1.1853125)(7.39,-1.2503124)
\curveto(7.57,-1.3153125)(7.945,-1.4953125)(8.14,-1.6103125)
\curveto(8.335,-1.7253125)(8.67,-1.8903126)(8.81,-1.9403125)
\curveto(8.95,-1.9903125)(9.33,-2.0553124)(9.57,-2.0703125)
\curveto(9.81,-2.0853126)(10.225,-2.0503125)(10.4,-2.0003126)
\curveto(10.575,-1.9503125)(10.785,-1.8703125)(10.82,-1.8403125)
\curveto(10.855,-1.8103125)(10.915,-1.7553124)(10.94,-1.7303125)
\curveto(10.965,-1.7053125)(11.005,-1.6553125)(11.02,-1.6303124)
\curveto(11.035,-1.6053125)(11.055,-1.5553125)(11.07,-1.4803125)
}
\psline[linewidth=0.07,linestyle=none,arrows=->](5.58,-1.2303125)(5.41,-1.3603125)

\end{pspicture} 
}

\caption[The local picture near $P_i$]
{
\label{fig:proving local winding pi}
{\bf The local picture near $P_i$.}  This allows us to compute the ``local fixed-point index'' $w(\zeta(p_i))$ of $f$ near $P_i$.
}
\end{figure}
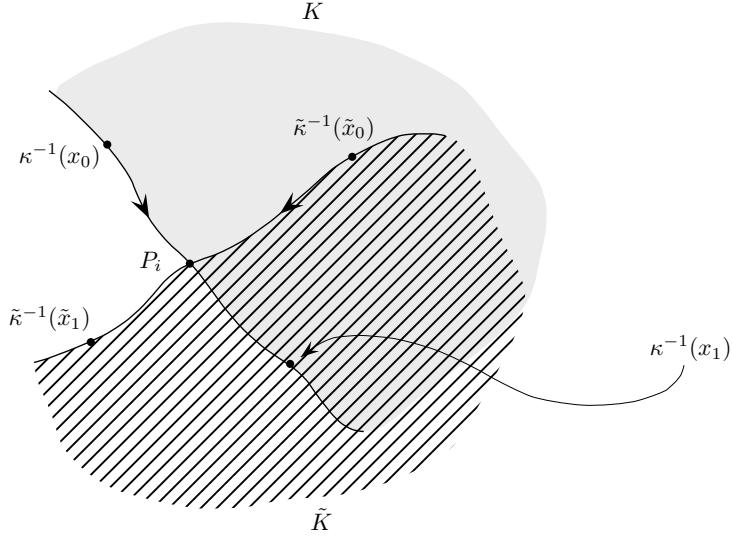

\begin{observation}
\label{lem:local windings for pi qi}
$w(\zeta(p_i)) = 1$, $w(\zeta(\tilde p_i)) = -1$
\end{observation}

\noindent To see why, suppose that $\zeta(p_i) = \partial ( [x_0 \to x_1]_{\bbS^1} \times [\tilde x_0\to \tilde x_1]_{\bbS^1})$.
(If $\alpha$ is an oriented topological circle with $a_0, a_1 \in \alpha$, then we denote by $[a_0 \to a_1]_\alpha$ the closed oriented sub-arc of $\alpha$ beginning at $a_0$ and ending at $a_1$.)
Then up to orientation-preserving homeomorphism the picture near $P_i$ is as in Figure \ref{fig:proving local winding pi}. We let $(x,\tilde x)$ traverse $\zeta(p_i)$ positively starting at $(x_0,\tilde x_0)$, keeping track of the vector $\tilde\kappa\inv(\tilde x) - \tilde\kappa\inv(x)$ as we do so. The vector $\tilde\kappa\inv(\tilde x_0) - \tilde\kappa\inv(x_0)$ points to the right. As $x$ varies from $x_0$ to $x_1$, the vector $\tilde\kappa\inv(\tilde x) - \tilde\kappa\inv(x)$ rotates in the positive direction, that is, counter-clockwise, until it arrives at $\tilde \kappa\inv (\tilde x_0) - \kappa\inv (x_1)$, which points upward. Continuing in this fashion, we see that $\tilde\kappa\inv(\tilde x) - \tilde\kappa\inv(x)$ makes one full counter-clockwise rotation as we traverse $\zeta(p_i)$. The proof that $w(\zeta(\tilde p_i)) = -1$ is similar. Combining all of our observations establishes equation \ref{eq 1:prop:computing index from torus eq}. The proof that equation \ref{eq 2:prop:computing index from torus eq} holds is similar.\qed

\section{Preparatory plane-topological lemmas} \label{geometric-lemmas-section}

As usual, let $K$ and $\tilde K$ be
compact Jordan domains in transverse position,
whose boundaries meet at least twice.
Note that the connected components of
$\jdIntersection$ are points,
that those of
$(\partial K \setminus \partial \tilde K)
\cup (\partial \tilde K \setminus \partial K)$
are topological open intervals,
and that those of
$\hat \bbC \setminus (\jdUnion)$
are topological open disks.
Denote by $X = (V, E, F)$
the cellular decomposition of the Riemann sphere
obtained by taking the connected components of
$\jdIntersection$ to be the vertices,
those of $(\partial K \setminus \partial \tilde K)
\cup (\partial \tilde K \setminus \partial K)$ to be the edges, and
those of $\hat \bbC \setminus (\jdUnion)$ to be the faces.

For a face $f \in F$, define $\deg(f)$ to be the number of sides of $f$.
Note that $\deg(f)$ is even for all $f \in F$.
Let $F_n = \{f \in F : \deg(f) = n\}$,
and $F_{> n} = \{f \in F : \deg(f) > n\}$.

\begin{lemma} \label{at-least-four-bigons-lemma}
    $|F_2| \ge 4$.
    Furthermore, if there is an $f\in F$ having $\deg(f) > 4$,
    then $|F_2| > 4$.
\end{lemma}

\begin{proof}
    Recall that for a cellular decomposition of a topological sphere,
    we have $|V| - |E| + |F| = 2$.
    In our setup, it is easy to see that $|E| = 2 |V|$.
    It follows that $|F| = |E| / 2 + 2$.

    Next, note that
    \[
        2 |E| = \sum_{f \in F} \deg(f) = 2 |F_2| + 4 |F_4| + 6 |F_6| + \cdots
    \]
    so we get that $|E| = |F_2| + 2 |F_4| + 3 |F_6| + \cdots$.
    Finally, note that $|F| = |F_2| + |F_{> 2}|$.

    Putting these together gives
    \begin{align*}
        |E| & = |F_2| + 2 |F_4| + 3 |F_6| + 4 |F_8| + \cdots \\
        & \ge |F_2| + 2 |F_4| + 2 |F_6| + 2 |F_8| + \cdots \\
        & = |F_2| + 2 |F_{>2}| \\
        & = 2 |F| - |F_2| = |E| + 4 - |F_2|
    \end{align*}
    noting that the inequality on the second line is strict
    if there is an $f \in F$ having $\deg(f) > 4$.
    The lemma follows.
\end{proof}

Let $F_K = \{f \in F : f \subset K \setminus \tilde K\}$,
and define $F_{\tilde K}$ analogously.
Let
$\fNone = \{f \in F : f \subset \hat \bbC \setminus (K \cup \tilde K)\}$,
and let
$\fBoth = \{f \in F : f \subset K \cap \tilde K\}$.
Then, for example, we have that
$F = F_K \sqcup F_{\tilde K} \sqcup \fNone \sqcup \fBoth$.
We call $f \in F$ a \emph{bigon of $X$} if $\deg(f) = 2$,
and we say that $f$ is \emph{finite} if
the point $\infty \in \hat \bbC$ is not in $f$.
For example, by Lemma \ref{at-least-four-bigons-lemma},
we always have at least 3 finite bigons.
Then $F_2$ denotes the set of all of the bigons of $X$.

\begin{lemma} \label{f-k-equals-f-tilde-k-lemma}
    $|F_K| = |F_{\tilde K}|$
\end{lemma}

\begin{proof}
    This is apparent if $\partial K$ and $\partial \tilde K$ meet twice,
    because in that case $K$ and $\tilde K$ can be in only one
    topological configuration (for example, by Lemma \ref{uniqa}).
    We proceed by induction on the size of $|V| = |\jdIntersection|$,
    using $|V| = 2$ for our base case.

    Suppose that $|V| \ge 4$ and
    let $f$ be a finite bigon of $F$.
    We have depicted the picture near $f$
    in Figure \ref{f-k-equals-f-tilde-k-figure}.
    In this figure, we have that $a, b, c, d, f$ are faces of $X = (V, E, F)$,
    the cellular decomposition induced by $K$ and $\tilde K$.
    Note that $c \ne d$ because $|V| \ge 4$.

    Our plan is to modify $K$ and $\tilde K$ as shown in the figure,
    to create two new compact Jordan domains $K'$ and $\tilde K'$.
    Using the prime symbol in the natural way, we have that
    $a', b', c'$ in the figure are faces of $X' = (V', E', F')$,
    the cellular decomposition induced by $K'$ and $\tilde K'$.
    The modification results in $c$ and $d$ coalescing to form $c'$,
    and causes $f$ to effectively disappear.
    It also reduces the number of intersection points of our
    Jordan domains, allowing us to apply induction.

    There are four possibilities to check.
    We work out the first in detail.
    Suppose that $f \in F_K$.
    In this case it follows that
    $a \in \fBoth$,
    $b \in \fNone$, and
    $c, d \in F_{\tilde K}$.
    Our modification to $K, \tilde K$
    in this case has the effect of decreasing each of
    $|F_K|$ and $|F_{\tilde K}|$ by 1.
    We apply the induction hypothesis and we are done.

    The remaining three cases to check are
    $f \in F_{\tilde K}, \fBoth, \fNone$.
    These are left as straightforward exercises for the reader.
\end{proof}

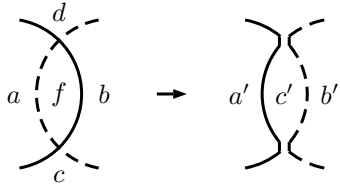
\begin{figure}
    \centering
    \psscalebox{1.0 1.0}
    {
        \begin{pspicture}(0,-1.4)(2.4,1.4)
            \psarc[linecolor=black, linewidth=0.04, dimen=outer](-0.7, 0){1}{-80.0}{80.0}
            \psarc[linecolor=black, linewidth=0.04, dimen=outer, linestyle=dashed, dash=0.18cm 0.1cm](0.7, 0){1}{100.0}{260.0}

            \usefont{T1}{ptm}{m}{n}
            \rput(-0.6, -0.05){$a$}
            \usefont{T1}{ptm}{m}{n}
            \rput(0, 0){$f$}
            \usefont{T1}{ptm}{m}{n}
            \rput(0.6, 0){$b$}
            \usefont{T1}{ptm}{m}{n}
            \rput(0, 1.1){$d$}
            \usefont{T1}{ptm}{m}{n}
            \rput(0, -1.1){$c$}

            \psline[linecolor=black, linewidth=0.04, arrowsize=0.05cm 2.0, arrowlength=1.4, arrowinset=0.0]{->}(1.3, 0)(1.7, 0)

            \psarc[linecolor=black, linewidth=0.04, dimen=outer](2.3, 0){1}{-80.0}{-50.0}
            \psarc[linecolor=black, linewidth=0.04, dimen=outer, linestyle=dashed, dash=0.18cm 0.1cm](2.3, 0){1}{-40.0}{40.0}
            \psarc[linecolor=black, linewidth=0.04, dimen=outer](2.3, 0){1}{50.0}{80.0}
            \psarc[linecolor=black, linewidth=0.04, dimen=outer, linestyle=dashed, dash=0.18cm 0.1cm](3.7, 0){1}{100.0}{130.0}
            \psarc[linecolor=black, linewidth=0.04, dimen=outer](3.7, 0){1}{140.0}{220.0}
            \psarc[linecolor=black, linewidth=0.04, dimen=outer, linestyle=dashed, dash=0.18cm 0.1cm](3.7, 0){1}{230.0}{260.0}

            \psline[linecolor=black, linewidth=0.04](2.93, 0.76)(2.93, 0.64)
            \psline[linecolor=black, linewidth=0.04](2.93, -0.64)(2.93, -0.76)
            \psline[linecolor=black, linewidth=0.04](3.06, -0.64)(3.06, -0.76)
            \psline[linecolor=black, linewidth=0.04](3.06, 0.64)(3.06, 0.76)

            \usefont{T1}{ptm}{m}{n}
            \rput(2.4, 0){$a'$}
            \usefont{T1}{ptm}{m}{n}
            \rput(3, 0){$c'$}
            \usefont{T1}{ptm}{m}{n}
            \rput(3.6, 0){$b'$}
        \end{pspicture}
    }

    \caption{\label{f-k-equals-f-tilde-k-figure}
        {\bf Applying the induction step
        in the proof of Lemma \ref{f-k-equals-f-tilde-k-lemma}.}
        The solid arcs represent segments of
        $\partial K$ and of $\partial K'$ on the
        left and right sides respectively, similarly
        the dashed arcs for $\partial \tilde K$ and $\partial \tilde K'$.
    }
\end{figure}

\begin{lemma} \label{f-both-equals-f-none-lemma}
    $|\fBoth| = |\fNone|$
\end{lemma}

\begin{proof}
    This follows from Lemma \ref{f-k-equals-f-tilde-k-lemma}.
    Let $z$ be a point in the interior of $\tilde K \setminus K$,
    let $\phi : \hat \bbC \to \hat \bbC$ be a homeomorphism
    interchanging $z$ and $\infty$, let
    $K' = \phi(K)$,
    let $\tilde K'$ be the closure of $\hat \bbC \setminus \phi(\tilde K)$,
    and apply Lemma \ref{f-k-equals-f-tilde-k-lemma}
    to $K'$ and $\tilde K'$.
    The details are left to the reader.
\end{proof}

\begin{lemma} \label{face-degree-sums-all-equal-lemma}
    $|V| =
    \sum_{f \in F_K} \deg(f) =
    \sum_{f \in F_{\tilde K}} \deg(f) =
    \sum_{f \in \fNone} \deg(f) =
    \sum_{f \in \fBoth} \deg(f)$
\end{lemma}

\begin{proof}
    Every vertex $v \in V$ lies on the boundaries of exactly four faces of $X$,
    one from each of $F_K, F_{\tilde K}, \fNone, \fBoth$,
    and a face's edge degree and its vertex degree are the same.
    Thus every vertex can be thought of as being counted
    exactly once in each of the sums in the statement of the lemma.
\end{proof}

\begin{lemma} \label{jordan-domain-combinatorics-lemma}
At least one of the following holds:
\begin{enumerate}
    \item \label{JDCL-number-of-bigons}
        There are strictly more than three finite bigons in $X$.
    \item \label{JDCL-bigons-in-intersection}
        There is a (necessarily finite) bigon of $X$ in $\fBoth$.
\end{enumerate}
\end{lemma}

\begin{proof}
    Suppose that Condition \ref{JDCL-number-of-bigons} fails.
    Then, by Lemma \ref{at-least-four-bigons-lemma},
    there are exactly three finite bigons in $X$,
    and one infinite bigon,
    necessarily belonging to $\fNone$.
    Consider the sums $S_\varnothing = \sum_{f \in \fNone} \deg(f)$
    and $S_\cap = \sum_{f \in \fBoth} \deg(f)$.
    The two sums are equal by Lemma \ref{face-degree-sums-all-equal-lemma},
    and they have the same number of terms,
    all of which are positive even integers,
    by Lemma \ref{f-both-equals-f-none-lemma}.
    The sum $S_\varnothing$ has at least one
    term which is equal to $2$ (coming from the infinite bigon).
    By the strictness portion of Lemma \ref{at-least-four-bigons-lemma},
    no term of the sum $S_\cap$ is strictly greater than 4, so
    this sum must have at least one term which is equal to 2,
    which is equivalent to
    Condition \ref{JDCL-bigons-in-intersection}.
\end{proof}

Suppose that $z_1, z_2, z_3 \in \partial K \setminus \partial \tilde K$
and $\tilde z_1, \tilde z_2, \tilde z_3 \in
\partial \tilde K \setminus \partial K$ appear in counterclockwise order
as in the statement of the Three Point Prescription Lemma \ref{tppl}.
Let $f$ be a bigon of $X$.
We define the notion of
\emph{constraint point count} for $f$, denoted $\cpc$, as follows:
\begin{align*}
    \cpc_K(f) & = |\partial f \cap \{z_1, z_2, z_3\}| \\
    \cpc_{\tilde K}(f) & = |\partial f \cap \{\tilde z_1, \tilde z_2, \tilde z_3\}| \\
    \cpc(f) & = (\cpc_K(f), \cpc_{\tilde K}(f))
\end{align*}
Then, for example, we get that:
\begin{align*}
  \sum_{f \in F_2} \cpc_K(f) & \le 3 \\
  \sum_{f \in F_2} \cpc_{\tilde K}(f) & \le 3
\end{align*}
Of course the constraint point count of $f$ depends on
the choices of $z_i, \tilde z_i$.
This dependence is suppressed in the notation
as it is normally clear which $z_i, \tilde z_i$ are meant.

\begin{lemma} \label{cpc-lemma}
    At least one of the following holds:
    \begin{enumerate}
        \item \label{cpc-lemma-condition-1}
            There is a finite bigon $f$ in $X$
            for which $\cpc(f)$ is equal to one of
            $(0, 0), (1, 0), (0, 1)$.
        \item \label{cpc-lemma-condition-2}
            There is a finite bigon $f$ in $X$
            satisfying $f \subset K \cap \tilde K$,
            for which $\cpc(f)$ is equal to one of
            $(2, 0), (1, 1), (0, 2)$.
    \end{enumerate}
\end{lemma}

\begin{proof}
    Recall that by Lemma \ref{at-least-four-bigons-lemma},
    there are always at least three finite bigons in $X$.

    First, suppose that we have strictly more than three
    finite bigons.
    In this case, we will see that
    Condition \ref{cpc-lemma-condition-1} holds.
    Let $f_1, f_2, f_3, f_4$ be four of our finite bigons,
    and suppose for contradiction that
    Condition \ref{cpc-lemma-condition-1} fails.
    Then, for each $f_i$ we have that
    $\cpc_K(f_i) + \cpc_{\tilde K}(f_i) \ge 2$.
    It follows that
    $\sum_{i = 1}^4 \cpc_K(f_i) + \cpc_{\tilde K}(f_i) \ge 8$.
    On the other hand, we know that
    $\sum_{i = 1}^4 \cpc_K(f_i) + \cpc_{\tilde K}(f_i) \allowbreak \le 6$,
    which is a contradiction.

    Next, suppose that we have exactly three finite bigons,
    implying by Lemma \ref{jordan-domain-combinatorics-lemma}
    that at least one of them, call it $f_0$, lies in $K \cap \tilde K$.
    Let $f_1, f_2$ denote the other two.
    Suppose that Condition \ref{cpc-lemma-condition-1} fails.
    Then, we have that
    $\cpc_K(f_i) + \cpc_{\tilde K}(f_i) \ge 2$
    for $i = 0, 1, 2$.
    On the other hand,
    we know as before that
    $\sum_{i = 0}^2 \cpc_K(f_i) + \cpc_{\tilde K}(f_i) \le 6$.
    Putting these inequalities together gives that
    $\cpc_K(f_i) + \cpc_{\tilde K}(f_i) = 2$,
    in particular for $i = 0$,
    implying Condition \ref{cpc-lemma-condition-2}.
\end{proof}

\section{Proof of the Three Point Prescription Lemma \ref{lem:3p}}
\label{sec:3p}

Let $K$ and $\tilde K$ be compact Jordan domains in transverse position.
Let $z_1,z_2,z_3\in \partial K\setminus \partial \tilde K$
appear in counterclockwise order,
similarly
$\tilde z_1,\tilde z_2, \tilde z_3\in \partial \tilde K\setminus \partial K$.
We wish to find
an indexable homeomorphism $\phi:\partial K\to \partial \tilde K$
sending $z_i\mapsto \tilde z_i$
for $i = 1, 2, 3$.

(We remind the interested reader to refer to
Remark \ref{closing-remark} at the end of the section
for a discussion on the
strength of the hypotheses on $K, \tilde K, z_i, \tilde z_i$ in
the statement of the lemma.)

We proceed by induction on the number of intersection points
$\partial K \cap \partial \tilde K$,
recalling that this number is always even.
The Circle Index Lemma \ref{cil} takes care of the cases where
$\partial K$ and $\partial \tilde K$ meet $0$ or $2$ times.
Thus, suppose for the remainder of the argument
that $\partial K$ and $\partial \tilde K$ meet at least $4$ times.

Our plan is to apply basically the same modification
to $K$ and $\tilde K$ as we did
in Lemma \ref{f-k-equals-f-tilde-k-lemma}
and Figure \ref{f-k-equals-f-tilde-k-figure},
except this time, we will be particular
in our choice of which finite bigon to focus on.
For the remainder of the proof, we use
$K'$ and $\tilde K'$ to denote the
domains obtained via this modification.
Figure \ref{fig:pull me 2} gives an alternative,
topologically equivalent, example visualization.

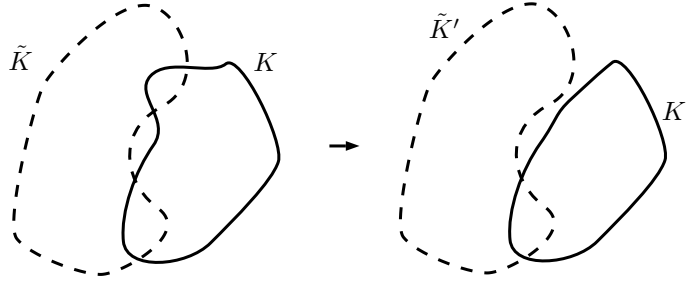
\begin{figure}
\centering
% Generated with LaTeXDraw 2.0.8
% Thu Jul 26 13:38:33 EDT 2012
% \usepackage[usenames,dvipsnames]{pstricks}
% \usepackage{epsfig}
% \usepackage{pst-grad} % For gradients
% \usepackage{pst-plot} % For axes
\scalebox{1} % Change this value to rescale the drawing.
{
\begin{pspicture}(0,-1.8308277)(9.78291,1.8308275)
\definecolor{color2327b}{rgb}{0.8,0.8,0.8}
%\psline[linewidth=0.02cm,fillcolor=color2327b,tbarsize=0.07055555cm 5.0,arrowsize=0.05291667cm 2.0,arrowlength=1.4,arrowinset=0.4]{|*->}(4.301016,-0.10917247)(5.7010155,-0.10917247)
\psline[
    linecolor=black, linewidth=0.04, arrowsize=0.05cm 2.0,
    arrowlength=1.4, arrowinset=0.0]{->}(4.8, -0.1)(5.2, -0.1)

\usefont{T1}{ptm}{m}{n}
\rput(3.9524708,1.0158275){$K$}
\usefont{T1}{ptm}{m}{n}
\rput(0.6924707,1.0758275){$\tilde K$}
\usefont{T1}{ptm}{m}{n}
\rput(9.43247,0.37582752){$K'$}
\usefont{T1}{ptm}{m}{n}
\rput(6.332471,1.4958276){$\tilde K'$}
\psbezier[linewidth=0.04,fillcolor=color2327b](2.4610157,-0.089172475)(2.7010157,0.25082752)(2.189661,0.5789255)(2.4210157,0.83082753)(2.6523702,1.0827296)(3.2210157,0.83082753)(3.4210157,0.9908275)(3.6210155,1.1508275)(4.2010155,-0.049172472)(4.1210155,-0.30917248)(4.0410156,-0.5691725)(3.5410156,-1.0691725)(3.2210157,-1.3891724)(2.9010155,-1.7091725)(2.1010156,-1.7691724)(2.0610156,-1.3691725)(2.0210156,-0.9691725)(2.2210157,-0.4291725)(2.4610157,-0.089172475)
\psbezier[linewidth=0.04,linestyle=dashed,dash=0.16cm 0.16cm,fillcolor=color2327b](0.64101565,-1.3291725)(0.8010156,-1.5491725)(1.484355,-1.8108275)(1.7610157,-1.8091725)(2.0376763,-1.8075174)(2.985388,-1.2667643)(2.5210156,-0.92917246)(2.0566432,-0.5915806)(1.9610156,0.11082753)(2.5810156,0.43082753)(3.2010157,0.75082755)(2.8410156,1.7508276)(2.3610156,1.7708275)(1.8810157,1.7908275)(1.1210157,0.9308275)(1.0010157,0.7108275)(0.8810156,0.49082753)(0.48101562,-1.1091725)(0.64101565,-1.3291725)
\psbezier[linewidth=0.04,fillcolor=color2327b](7.6010156,-0.06917247)(7.841016,0.27082753)(7.789661,0.3189255)(8.021015,0.5308275)(8.25237,0.74272954)(8.361015,0.8508275)(8.561016,1.0108275)(8.761016,1.1708275)(9.341016,-0.029172473)(9.261016,-0.28917247)(9.181016,-0.54917246)(8.681016,-1.0491725)(8.361015,-1.3691725)(8.041016,-1.6891725)(7.2410154,-1.7491724)(7.2010155,-1.3491725)(7.1610155,-0.9491725)(7.361016,-0.40917248)(7.6010156,-0.06917247)
\psbezier[linewidth=0.04,linestyle=dashed,dash=0.16cm 0.16cm,fillcolor=color2327b](5.7810154,-1.3091725)(5.9410157,-1.5291724)(6.624355,-1.7908275)(6.9010158,-1.7891725)(7.177676,-1.7875174)(8.125388,-1.2467643)(7.6610155,-0.9091725)(7.1966434,-0.57158065)(7.1010156,0.13082753)(7.7210155,0.45082754)(8.341016,0.77082753)(7.9810157,1.7708275)(7.5010157,1.7908275)(7.0210156,1.8108275)(6.2610154,0.95082754)(6.1410155,0.7308275)(6.0210156,0.51082754)(5.6210155,-1.0891725)(5.7810154,-1.3091725)
\end{pspicture} 
}
\caption{\label{fig:pull me 2}
{\bf A before-and-after view of the modification we apply to $K$ and $\tilde K$
in the induction step of the proof of the Three Point Prescription Lemma \ref{tppl}.}
We essentially get rid of one of our finite bigons.
In this case, we had three choices of which finite bigon to focus on.
}
\end{figure}

We continue with the notation of Section \ref{geometric-lemmas-section}.
Suppose that we have chosen a finite bigon $f$ of $X$.
Beginning with a torus parametrization for $K$ and $\tilde K$,
we pick a torus parametrization for $K'$ and $\tilde K'$, and
show how we can move between these two parametrizations.

Let $U \subset \bbC$ be
a simply connected open neighborhood of $f \subset \bbC$,
with closure $\bar U$,
so that
$\bar U \cap (\partial K \cap \partial \tilde K)
= \partial f \cap (\partial K \cap \partial \tilde K)$,
and so that every $z_i$ and
every $\tilde z_i$ in $\bar U$ also lies in $\partial f$.
Conceptually, we want $U$ to be a very small neighborhood of $f$.
As before, we obtain $K'$ and $\tilde K'$
by modifying $K$ and $\tilde K$ within $U$,
so that $\partial K'$ and $\partial \tilde K'$ do not meet in $U$.
Let $\psi_K : \partial K \cap U \to \partial K' \cap U$
denote an orientation-preserving homeomorphism which agrees with
the identity map on the endpoints of its domain,
and define $\psi_{\tilde K}$ similarly.
Extend $\psi_K$ and $\psi_{\tilde K}$ to all of
$\partial K$ and $\partial \tilde K$ via the identity map.
Define $z'_i = \psi_K(z_i)$
and $\tilde z'_i = \psi_{\tilde K}(\tilde z'_i)$
for $i = 1, 2, 3$.

Suppose that we have fixed a torus parametrization
$\kappa : \partial K \to \bbS^1,
\tilde \kappa : \partial \tilde K \to \bbS^1$
for $K$ and $\tilde K$.
Let $\kappa' = \psi_K^{-1} \circ \kappa$ and
$\tilde \kappa' = \psi_{\tilde K}^{-1} \circ \tilde \kappa$.
Then
$\kappa' : \partial K' \to \bbS^1,
\tilde \kappa' : \partial \tilde K' \to \bbS^1$
is a torus parametrization
for $K'$ and $\tilde K'$
so that $\kappa \equiv \kappa'$
and $\tilde \kappa \equiv \tilde \kappa'$ outside of $U$,
and so that $\kappa(z_i) = \kappa'(z'_i)$
and
$\tilde \kappa(\tilde z_i) = \tilde \kappa'(\tilde z'_i)$
for $i = 1, 2, 3$.

By the induction hypothesis,
there exists an indexable homeomorphism
$\phi' : \partial K' \to \partial \tilde K'$ sending
$z'_i \mapsto \tilde z'_i$ for $i = 1, 2, 3$, satisfying $\eta(\phi') \ge 0$.
Let $\gamma'$ be the graph of $\phi'$ in $\bbT = \bbS^1 \times \bbS^1$
according to the parametrization $\kappa', \tilde \kappa'$.
Our goal is to modify $\gamma'$ to obtain a new simple closed curve
$\gamma \subset \bbT$ which will be the graph,
according to the parametrization $\kappa, \tilde \kappa$,
of an
indexable homeomorphism $\phi : \partial K \to \partial \tilde K$,
in such a way that $\phi$ has the properties required to
complete the proof of the lemma.

We break the conclusion of the proof into two cases
according to Lemma \ref{cpc-lemma}.

\smallskip \noindent{\bf Case 1 of Lemma \ref{cpc-lemma}.}
Without loss of generality,
by interchanging the roles of $K$ and $\tilde K$ if necessary,
let $f$ be a finite bigon in $X$
satisfying $\cpc_{K}(f) = 0$,
with $\cpc_{\tilde K}(f)$ equal to either $0$ or $1$.
Consider the strips
$S_{\tilde K} = \bbS^1 \times \tilde \kappa(U \cap \partial \tilde K)
= \bbS^1 \times \tilde \kappa'(U \cap \partial \tilde K')
\subset \bbT$
and
$S_K = \kappa(U \cap \partial K) \times \bbS^1
= \kappa'(U \cap \partial K') \times \bbS^1
\subset \bbT$,
let $S_\cup = S_K \cup S_{\tilde K}$
denote their union,
and let $S_\cap = S_K \cap S_{\tilde K}$
denote their intersection.
Note that there is no point in
$\partial K' \cap \partial \tilde K'$
whose parametrization under $\kappa', \tilde \kappa'$
lies in $S_\cup$,
more formally that
$\{(\kappa'(z), \tilde \kappa'(z))
: z \in \partial K' \cap \partial \tilde K'\} \cap S_\cup = \varnothing$.
It follows that if we modify $\gamma'$
however we like within $S_\cup$,
leaving the points $\gamma' \cap \partial S_\cup$ fixed,
to obtain a new simple closed curve
$\gamma \subset \bbT$
which
\begin{itemize}
    \item
        is strictly increasing
        in the sense of Section \ref{chap:torus},
    \item
        contains (as is the case with $\gamma'$)
        the points $(\kappa'(z'_i), \tilde \kappa'(\tilde z'_i))$,
        equivalently $(\kappa(z_i), \tilde \kappa(\tilde z_i))$,
        for $i = 1, 2, 3$, and
    \item
        does not contain (as is the case with $\gamma'$)
        any of the points $(\kappa(z), \tilde \kappa(z))$
        for $z \in \partial K \cap \partial \tilde K$,
        noting that $\partial K' \cap \partial \tilde K'
        \subset \partial K \cap \partial \tilde K$,
\end{itemize}
then $\gamma$ will be the graph
of both of two indexable homeomorphisms
$\phi'' : \partial K' \to \partial \tilde K'$
and
$\phi : \partial K \to \partial \tilde K$,
where furthermore
$\eta(\phi'') = \eta(\phi') \ge 0$.

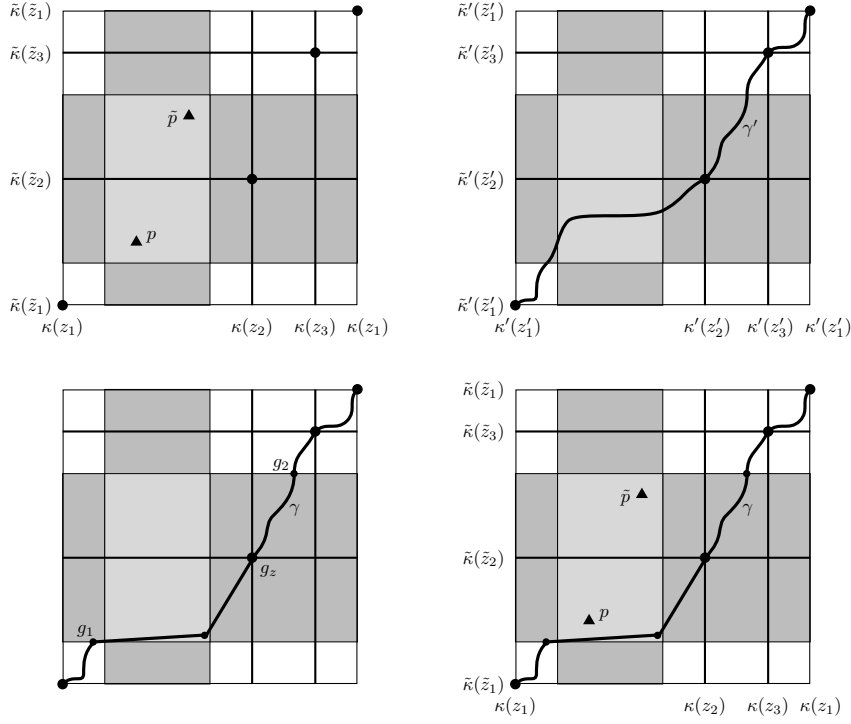
\begin{figure} \centering
    \scalebox{0.7} % Change this value to rescale the drawing.
    {
        \begin{pspicture}(0,-6.775)(15.89,6.775)
            \definecolor{colour0}{rgb}{0.7411765,0.7411765,0.7411765}
            \definecolor{colour1}{rgb}{0.84705883,0.84705883,0.84705883}
            \psframe[linecolor=black, linewidth=0.002, fillstyle=solid,fillcolor=colour0, dimen=outer](6.6,5.025)(1.0,1.825)
            \psframe[linecolor=black, linewidth=0.002, fillstyle=solid,fillcolor=colour0, dimen=outer](3.8,6.625)(1.8,1.025)
            \psframe[linecolor=black, linewidth=0.002, fillstyle=solid,fillcolor=colour1, dimen=outer](3.8,5.025)(1.8,1.825)
            \psframe[linecolor=black, linewidth=0.02, dimen=outer](6.6,6.625)(1.0,1.025)
            \psline[linecolor=black, linewidth=0.04](4.6,6.625)(4.6,1.025)
            \psline[linecolor=black, linewidth=0.04](5.8,1.025)(5.8,6.625)
            \psline[linecolor=black, linewidth=0.04](1.0,5.825)(6.6,5.825)
            \psline[linecolor=black, linewidth=0.04](1.0,3.425)(6.6,3.425)
            \psdots[linecolor=black, dotsize=0.2](1.0,1.025)
            \psdots[linecolor=black, dotsize=0.2](4.6,3.425)
            \psdots[linecolor=black, dotsize=0.2](5.8,5.825)
            \psdots[linecolor=black, dotsize=0.2](6.6,6.625)
            \psdots[linecolor=black, dotstyle=triangle*, dotsize=0.2](2.4,2.225)
            \psdots[linecolor=black, dotstyle=triangle*, dotsize=0.2](3.4,4.625)
            \rput[bl](4.2,0.425){$\kappa(z_2)$}
            \rput[bl](5.4,0.425){$\kappa(z_3)$}
            \rput[bl](0.6,0.425){$\kappa(z_1)$}
            \rput[bl](6.4,0.425){$\kappa(z_1)$}
            \rput[bl](0.0,0.825){$\tilde\kappa(\tilde z_1)$}
            \rput[bl](0.0,3.225){$\tilde\kappa(\tilde z_2)$}
            \rput[bl](0.0,5.625){$\tilde\kappa(\tilde z_3)$}
            \rput[bl](0.0,6.425){$\tilde\kappa(\tilde z_1)$}
            \rput[bl](2.6,2.225){$p$}
            \rput[bl](3.0,4.425){$\tilde p$}
            \psframe[linecolor=black, linewidth=0.002, fillstyle=solid,fillcolor=colour0, dimen=outer](15.2,5.025)(9.6,1.825)
            \psframe[linecolor=black, linewidth=0.002, fillstyle=solid,fillcolor=colour0, dimen=outer](12.4,6.625)(10.4,1.025)
            \psframe[linecolor=black, linewidth=0.002, fillstyle=solid,fillcolor=colour1, dimen=outer](12.4,5.025)(10.4,1.825)
            \psframe[linecolor=black, linewidth=0.02, dimen=outer](15.2,6.625)(9.6,1.025)
            \psline[linecolor=black, linewidth=0.04](13.2,6.625)(13.2,1.025)
            \psline[linecolor=black, linewidth=0.04](14.4,1.025)(14.4,6.625)
            \psline[linecolor=black, linewidth=0.04](9.6,5.825)(15.2,5.825)
            \psline[linecolor=black, linewidth=0.04](9.6,3.425)(15.2,3.425)
            \psdots[linecolor=black, dotsize=0.2](9.6,1.025)
            \psdots[linecolor=black, dotsize=0.2](13.2,3.425)
            \psdots[linecolor=black, dotsize=0.2](14.4,5.825)
            \psdots[linecolor=black, dotsize=0.2](15.2,6.625)
            \rput[bl](12.8,0.425){$\kappa'(z'_2)$}
            \rput[bl](14.0,0.425){$\kappa'(z'_3)$}
            \rput[bl](9.2,0.425){$\kappa'(z'_1)$}
            \rput[bl](15.1,0.425){$\kappa'(z'_1)$}
            \rput[bl](8.5,0.825){$\tilde\kappa'(\tilde z'_1)$}
            \rput[bl](8.5,3.225){$\tilde\kappa'(\tilde z'_2)$}
            \rput[bl](8.5,5.625){$\tilde\kappa'(\tilde z'_3)$}
            \rput[bl](8.5,6.425){$\tilde\kappa'(\tilde z'_1)$}
            \psbezier[linecolor=black, linewidth=0.06](9.6,1.025)(9.8,1.225)(10.0,1.025)(10.0,1.225)(10.0,1.425)(10.0,1.625)(10.2,1.825)
            \psbezier[linecolor=black, linewidth=0.06](10.2,1.825)(10.4,2.025)(10.4,2.425)(10.6,2.625)(10.8,2.825)(12.0,2.625)(12.4,2.825)(12.8,3.025)(12.8,3.225)(13.2,3.425)
            \psbezier[linecolor=black, linewidth=0.06](13.2,3.425)(13.6,3.825)(13.4,4.025)(13.6,4.225)(13.8,4.425)(14.0,4.625)(14.0,5.025)(14.0,5.425)(14.2,5.425)(14.4,5.825)
            \psbezier[linecolor=black, linewidth=0.06](14.4,5.825)(14.6,6.025)(14.8,5.825)(15.0,6.025)(15.2,6.225)(15.0,6.425)(15.2,6.625)
            \rput[bl](13.911111,4.225){$\gamma'$}
            \psframe[linecolor=black, linewidth=0.002, fillstyle=solid,fillcolor=colour0, dimen=outer](6.6,-2.175)(1.0,-5.375)
            \psframe[linecolor=black, linewidth=0.002, fillstyle=solid,fillcolor=colour0, dimen=outer](3.8,-0.575)(1.8,-6.175)
            \psframe[linecolor=black, linewidth=0.002, fillstyle=solid,fillcolor=colour1, dimen=outer](3.8,-2.175)(1.8,-5.375)
            \psframe[linecolor=black, linewidth=0.02, dimen=outer](6.6,-0.575)(1.0,-6.175)
            \psline[linecolor=black, linewidth=0.04](4.6,-0.575)(4.6,-6.175)
            \psline[linecolor=black, linewidth=0.04](5.8,-6.175)(5.8,-0.575)
            \psline[linecolor=black, linewidth=0.04](1.0,-1.375)(6.6,-1.375)
            \psline[linecolor=black, linewidth=0.04](1.0,-3.775)(6.6,-3.775)
            \psdots[linecolor=black, dotsize=0.2](1.0,-6.175)
            \psdots[linecolor=black, dotsize=0.2](4.6,-3.775)
            \psdots[linecolor=black, dotsize=0.2](5.8,-1.375)
            \psdots[linecolor=black, dotsize=0.2](6.6,-0.575)
            \psbezier[linecolor=black, linewidth=0.06](1.0,-6.175)(1.2,-5.975)(1.4,-6.175)(1.4,-5.975)(1.4,-5.775)(1.4,-5.575)(1.6,-5.375)
            \psbezier[linecolor=black, linewidth=0.06](4.6,-3.775)(5.0,-3.375)(4.8,-3.175)(5.0,-2.975)(5.2,-2.775)(5.4,-2.575)(5.4,-2.175)(5.4,-1.775)(5.6,-1.775)(5.8,-1.375)
            \psbezier[linecolor=black, linewidth=0.06](5.8,-1.375)(6.0,-1.175)(6.2,-1.375)(6.4,-1.175)(6.6,-0.975)(6.4,-0.775)(6.6,-0.575)
            \rput[bl](5.311111,-2.975){$\gamma$}
            \rput[bl](4.7555556,-4.1416665){$g_z$}
            \rput[bl](1.28,-5.277222){$g_1$}
            \rput[bl](4.9911113,-2.1216667){$g_2$}
            \psdots[linecolor=black, dotsize=0.14](1.5822222,-5.375)
            \psdots[linecolor=black, dotsize=0.14](5.3955555,-2.175)
            \psdots[linecolor=black, dotsize=0.14](3.6977777,-5.2505555)
            \psline[linecolor=black, linewidth=0.06](1.6,-5.375)(3.7333333,-5.241667)(4.6222224,-3.775)
            \psframe[linecolor=black, linewidth=0.002, fillstyle=solid,fillcolor=colour0, dimen=outer](15.2,-2.175)(9.6,-5.375)
            \psframe[linecolor=black, linewidth=0.002, fillstyle=solid,fillcolor=colour0, dimen=outer](12.4,-0.575)(10.4,-6.175)
            \psframe[linecolor=black, linewidth=0.002, fillstyle=solid,fillcolor=colour1, dimen=outer](12.4,-2.175)(10.4,-5.375)
            \psframe[linecolor=black, linewidth=0.02, dimen=outer](15.2,-0.575)(9.6,-6.175)
            \psline[linecolor=black, linewidth=0.04](13.2,-0.575)(13.2,-6.175)
            \psline[linecolor=black, linewidth=0.04](14.4,-6.175)(14.4,-0.575)
            \psline[linecolor=black, linewidth=0.04](9.6,-1.375)(15.2,-1.375)
            \psline[linecolor=black, linewidth=0.04](9.6,-3.775)(15.2,-3.775)
            \psdots[linecolor=black, dotsize=0.2](9.6,-6.175)
            \psdots[linecolor=black, dotsize=0.2](13.2,-3.775)
            \psdots[linecolor=black, dotsize=0.2](14.4,-1.375)
            \psdots[linecolor=black, dotsize=0.2](15.2,-0.575)
            \rput[bl](12.8,-6.775){$\kappa(z_2)$}
            \rput[bl](14.0,-6.775){$\kappa(z_3)$}
            \rput[bl](9.2,-6.775){$\kappa(z_1)$}
            \rput[bl](15.0,-6.775){$\kappa(z_1)$}
            \rput[bl](8.6,-6.375){$\tilde\kappa(\tilde z_1)$}
            \rput[bl](8.6,-3.975){$\tilde\kappa(\tilde z_2)$}
            \rput[bl](8.6,-1.575){$\tilde\kappa(\tilde z_3)$}
            \rput[bl](8.6,-0.775){$\tilde\kappa(\tilde z_1)$}
            \psbezier[linecolor=black, linewidth=0.06](9.6,-6.175)(9.8,-5.975)(10.0,-6.175)(10.0,-5.975)(10.0,-5.775)(10.0,-5.575)(10.2,-5.375)
            \psbezier[linecolor=black, linewidth=0.06](13.2,-3.775)(13.6,-3.375)(13.4,-3.175)(13.6,-2.975)(13.8,-2.775)(14.0,-2.575)(14.0,-2.175)(14.0,-1.775)(14.2,-1.775)(14.4,-1.375)
            \psbezier[linecolor=black, linewidth=0.06](14.4,-1.375)(14.6,-1.175)(14.8,-1.375)(15.0,-1.175)(15.2,-0.975)(15.0,-0.775)(15.2,-0.575)
            \rput[bl](13.911111,-2.975){$\gamma$}
            \psdots[linecolor=black, dotsize=0.14](10.182222,-5.375)
            \psdots[linecolor=black, dotsize=0.14](13.995556,-2.175)
            \psdots[linecolor=black, dotsize=0.14](12.297778,-5.2505555)
            \psline[linecolor=black, linewidth=0.06](10.2,-5.375)(12.333333,-5.241667)(13.222222,-3.775)
            \psdots[linecolor=black, dotstyle=triangle*, dotsize=0.2](11.0,-4.975)
            \psdots[linecolor=black, dotstyle=triangle*, dotsize=0.2](12.0,-2.575)
            \rput[bl](11.2,-4.975){$p$}
            \rput[bl](11.6,-2.775){$\tilde p$}
        \end{pspicture}
    }
    \caption{\label{tppl-case-1-example-figure}
        {\bf An example construction of $\gamma$
        in the proof of the Three Point Prescription Lemma \ref{tppl}
        in Case 1 of Lemma \ref{cpc-lemma}.}
        Here the lightly shaded region is $S_\cap$,
        and the darker regions comprise $S_\cup \setminus S_\cap$.
        Here $\cpc_{\tilde K}(f) = 1$,
        and in the notation of the proof,
        we have chosen $j = 1$, noting that
        $z_1, \tilde z_1 \not \in U$.
    }
\end{figure}

Next, note that the only points of
$\partial K \cap \partial \tilde K$
whose parametrizations under $\kappa, \tilde \kappa$
lie in $S_\cup$ are the corners of the bigon $f$,
call them $P$ and $\tilde P$.
Let $p = (\kappa(P), \tilde \kappa(P)) \in \bbT$,
and define $\tilde p$ similarly.
Our goal is to obtain $\gamma$ as above
so that the two points
$p$ and $\tilde p$
lie on the same side of $\gamma$.
More formally, let $1 \le j \le 3$ be chosen
so that
$z_j, \tilde z_j \not \in U$.
Then, in the language of Section \ref{chap:torus},
we will arrange so that either
both $p$ and $\tilde p$ lie in
$\Delta_\downarrow(z_j, \gamma)$,
or both lie in
$\Delta_\uparrow(z_j, \gamma)$.
For example,
if $\gamma'$
does not meet $S_\cap \ni p, \tilde p$, then
this condition is satisfied automatically
by taking $\gamma = \gamma'$.
Having found such a $\gamma$, it will follow from
Lemma \ref{prop:computing index from torus}
that $\eta(\phi) = \eta(\phi'')$,
completing the proof in this case.

Without loss of generality,
by enlarging $U$ slightly if necessary,
we may assume that
$\partial S_K \cap \partial S_{\tilde K} \cap \gamma' = \varnothing$.
Then, because $\gamma'$ is strictly increasing,
we get that $\gamma'$ meets $\partial S_\cap$
either $0$ or $2$ times. In the former case,
we are done by our earlier observation, so
suppose $\gamma'$ meets $\partial S_\cap$ twice.
It follows
by the pigeonhole principle
that $\gamma'$ meets $\partial S_\cup$ twice,
since $\gamma'$ meets each of $S_K$ and $S_{\tilde K}$ twice.

Let
$\{v_\nearrow, v_\nwarrow, v_\swarrow, v_\searrow\} = \partial S_K \cap \partial S_{\tilde K}$
denote the
corners of $S_\cap$ in the natural way.
Let $g_1, g_2 \in \gamma' \cap \partial S_\cup$ denote the points
where $\gamma'$ enters and exits $S_\cup$, respectively.
If there is an $i$ so that
$(\kappa(z_i), \tilde \kappa(\tilde z_i)) \in S_\cup$, then
let $g_z$ denote $(\kappa(z_i), \tilde \kappa(\tilde z_i))$,
otherwise let $g_z$ remain undefined.
Note that if $g_z$ is defined then
$g_z \in S_{\tilde K}$
by our assumption that $\cpc_K(f) = 0$.
We are now ready to describe the construction of $\gamma$.
We refer the reader to Figure \ref{tppl-case-1-example-figure} for
an example.

First, suppose that $g_z$ is undefined
(equivalently that $\cpc_{\tilde K}(f) = 0$).
Then, let $\gamma$ be obtained by replacing
$[g_1 \to g_2]_{\gamma'}$ by
a strictly increasing arc
starting at $g_1$,
progressing first to
a point located infinitesimally northwest of $v_\searrow$,
then to $g_2$.
This $\gamma$ satisfies our requirements and we are done.

Next suppose that $g_z$ is defined.
One of two cases occurs:
starting at $g_1$,
either $\gamma'$ passes through $g_z$ and then through $S_\cap$,
or it passes through $S_\cap$ and then through $g_z$.
In the first case,
construct $\gamma$ from $\gamma'$ by
replacing $[g_z \to g_2]_{\gamma'}$ by
a strictly increasing arc
starting at $g_z$,
progressing first to
a point located infinitesimally southeast of $v_\nwarrow$,
then to $g_2$.
In the second case,
construct $\gamma$ from $\gamma'$ by
replacing $[g_1 \to g_z]_{\gamma'}$ by
a strictly increasing arc
starting at $g_1$,
progressing first to
a point located infinitesimally northwest of $v_\searrow$,
then to $g_z$.
In both cases $\gamma$ satisfies our requirements and we are done.

\smallskip \noindent{\bf Case 2 of Lemma \ref{cpc-lemma}.}
Without loss of generality,
by interchanging the roles of $K$ and $\tilde K$ if necessary,
let $f$ be a finite bigon of $X$
satisfying $f \subset K \cap \tilde K$
with $\cpc(f)$ equal to either
$(1, 1)$ or $(0, 2)$.

Note that if $\cpc(f) = (0, 2)$, then
\begin{itemize}
    \item
        $U \cap \partial K$ is a strict subset of
        exactly one of
        $[z_1 \to z_2]_{\partial K}$,
        $[z_2 \to z_3]_{\partial K}$,
        or
        $[z_3 \to z_1]_{\partial K}$,
        and
    \item
        $U \cap \partial \tilde K$ is a strict superset of
        exactly one of
        $[\tilde z_1 \to \tilde z_2]_{\partial \tilde K}$,
        $[\tilde z_2 \to \tilde z_3]_{\partial \tilde K}$,
        or
        $[\tilde z_3 \to \tilde z_1]_{\partial \tilde K}$.
\end{itemize}
By cyclically and simultaneously permuting the indices
of the $z_i$ and $\tilde z_i$ if necessary,
we assume from now on without loss of generality that
if $\cpc(f) = (0, 2)$, then
$U \cap \partial \tilde K
\supset [\tilde z_2 \to \tilde z_3]_{\partial \tilde K}$.

First, suppose that either
\begin{itemize}
    \item 
        $\cpc(f) = (1, 1)$, and
        there does not exist an $i$ so that
        $(\kappa(z_i), \tilde \kappa(\tilde z_i)) \in S_\cap$,
        or
    \item
        $\cpc(f) = (0, 2)$, and
        $U \cap \partial K \not \subset [z_2 \to z_3]_{\partial K}$.
\end{itemize}
In both of these cases, it is possible to obtain
$\gamma$ from $\gamma'$ via a construction
very similar to that given in Case 1. This is left to the reader.

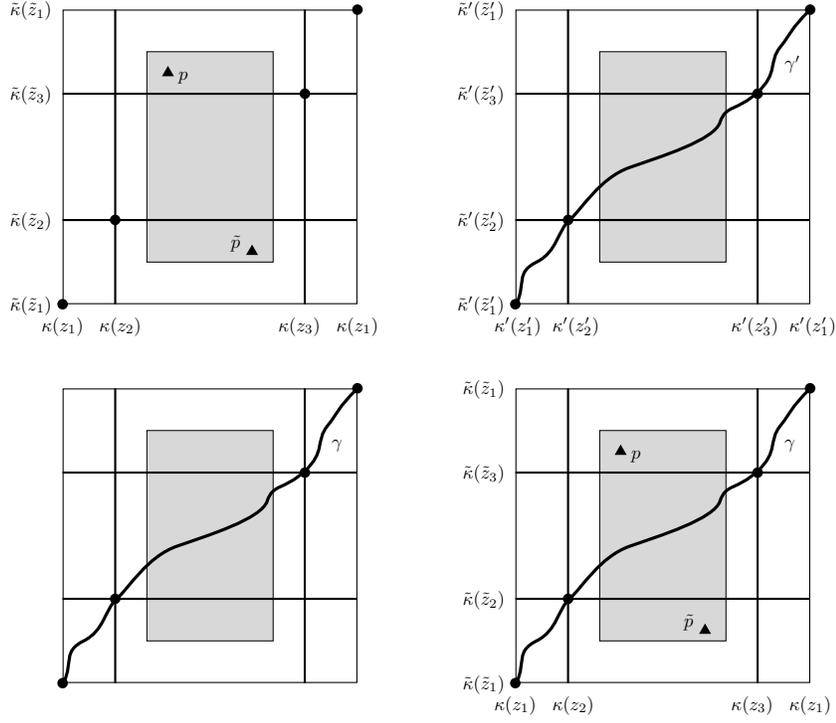
\begin{figure} \centering
    % \usepackage[usenames,dvipsnames]{pstricks}
    % \usepackage{epsfig}
    % \usepackage{pst-grad} % For gradients
    % \usepackage{pst-plot} % For axes
    % \usepackage[space]{grffile} % For spaces in paths
    % \usepackage{etoolbox} % For spaces in paths
    % \makeatletter % For spaces in paths
    % \patchcmd\Gread@eps{\@inputcheck#1 }{\@inputcheck"#1"\relax}{}{}
    % \makeatother
    % % User Packages:
    % 
    % 
    \scalebox{0.7} % Change this value to rescale the drawing.
    {
        \begin{pspicture}(0,-6.775)(15.89,6.775)
            \definecolor{colour1}{rgb}{0.84705883,0.84705883,0.84705883}
            \psframe[linecolor=black, linewidth=0.002, fillstyle=solid,fillcolor=colour1, dimen=outer](5.0,5.825)(2.6,1.825)
            \psframe[linecolor=black, linewidth=0.02, dimen=outer](6.6,6.625)(1.0,1.025)
            \psline[linecolor=black, linewidth=0.04](2.0,6.625)(2.0,1.025)
            \psline[linecolor=black, linewidth=0.04](5.6,1.025)(5.6,6.625)
            \psline[linecolor=black, linewidth=0.04](1.0,5.025)(6.6,5.025)
            \psline[linecolor=black, linewidth=0.04](1.0,2.625)(6.6,2.625)
            \psdots[linecolor=black, dotsize=0.2](1.0,1.025)
            \psdots[linecolor=black, dotsize=0.2](2.0,2.625)
            \psdots[linecolor=black, dotsize=0.2](5.6,5.025)
            \psdots[linecolor=black, dotsize=0.2](6.6,6.625)
            \psdots[linecolor=black, dotstyle=triangle*, dotsize=0.2](3.0,5.425)
            \psdots[linecolor=black, dotstyle=triangle*, dotsize=0.2](4.6,2.025)
            %\rput[bl](0.6,0.425){$\kappa(z_1)$} 0
            %\rput[bl](4.2,0.425){$\kappa(z_2)$} -2.5
            %\rput[bl](5.4,0.425){$\kappa(z_3)$} -0.3
            %\rput[bl](6.4,0.425){$\kappa(z_1)$} -0.2
            \rput[bl](0.6,0.425){$\kappa(z_1)$}
            \rput[bl](1.7,0.425){$\kappa(z_2)$}
            \rput[bl](5.1,0.425){$\kappa(z_3)$}
            \rput[bl](6.2,0.425){$\kappa(z_1)$}
            %\rput[bl](0.0,0.825){$\tilde\kappa(\tilde z_1)$} 0
            %\rput[bl](0.0,3.225){$\tilde\kappa(\tilde z_2)$} -0.8
            %\rput[bl](0.0,5.625){$\tilde\kappa(\tilde z_3)$} -0.8
            %\rput[bl](0.0,6.425){$\tilde\kappa(\tilde z_1)$} 0
            \rput[bl](0.0,0.825){$\tilde\kappa(\tilde z_1)$}
            \rput[bl](0.0,2.425){$\tilde\kappa(\tilde z_2)$}
            \rput[bl](0.0,4.825){$\tilde\kappa(\tilde z_3)$}
            \rput[bl](0.0,6.425){$\tilde\kappa(\tilde z_1)$}
            \rput[bl](3.2,5.225){$p$}
            \rput[bl](4.2,2.025){$\tilde p$}
            \psframe[linecolor=black, linewidth=0.02, dimen=outer](15.2,6.625)(9.6,1.025)
            \psdots[linecolor=black, dotsize=0.2](9.6,1.025)
            \psdots[linecolor=black, dotsize=0.2](15.2,6.625)
            %\rput[bl](9.2,0.425){$\kappa'(z'_1)$}
            %\rput[bl](12.8,0.425){$\kappa'(z'_2)$}
            %\rput[bl](14.0,0.425){$\kappa'(z'_3)$}
            %\rput[bl](15.2,0.425){$\kappa'(z'_1)$}
            %\rput[bl](8.5,0.825){$\tilde\kappa'(\tilde z'_1)$}
            %\rput[bl](8.5,3.225){$\tilde\kappa'(\tilde z'_2)$}
            %\rput[bl](8.5,5.625){$\tilde\kappa'(\tilde z'_3)$}
            %\rput[bl](8.5,6.425){$\tilde\kappa'(\tilde z'_1)$}
            \rput[bl](9.2,0.425){$\kappa'(z'_1)$}
            \rput[bl](10.3,0.425){$\kappa'(z'_2)$}
            \rput[bl](13.7,0.425){$\kappa'(z'_3)$}
            \rput[bl](14.8,0.425){$\kappa'(z'_1)$}
            \rput[bl](8.5,0.825){$\tilde\kappa'(\tilde z'_1)$}
            \rput[bl](8.5,2.425){$\tilde\kappa'(\tilde z'_2)$}
            \rput[bl](8.5,4.825){$\tilde\kappa'(\tilde z'_3)$}
            \rput[bl](8.5,6.425){$\tilde\kappa'(\tilde z'_1)$}
            \rput[bl](14.711111,5.425){$\gamma'$}
            \psframe[linecolor=black, linewidth=0.02, dimen=outer](6.6,-0.575)(1.0,-6.175)
            \psdots[linecolor=black, dotsize=0.2](1.0,-6.175)
            \psdots[linecolor=black, dotsize=0.2](6.6,-0.575)
            \psframe[linecolor=black, linewidth=0.02, dimen=outer](15.2,-0.575)(9.6,-6.175)
            \psdots[linecolor=black, dotsize=0.2](9.6,-6.175)
            \psdots[linecolor=black, dotsize=0.2](15.2,-0.575)
            %\rput[bl](9.2,-6.775){$\kappa(z_1)$}
            %\rput[bl](12.8,-6.775){$\kappa(z_2)$}
            %\rput[bl](14.0,-6.775){$\kappa(z_3)$}
            %\rput[bl](15.0,-6.775){$\kappa(z_1)$}
            %\rput[bl](8.6,-6.375){$\tilde\kappa(\tilde z_1)$}
            %\rput[bl](8.6,-3.975){$\tilde\kappa(\tilde z_2)$}
            %\rput[bl](8.6,-1.575){$\tilde\kappa(\tilde z_3)$}
            %\rput[bl](8.6,-0.775){$\tilde\kappa(\tilde z_1)$}
            \rput[bl](9.2,-6.775){$\kappa(z_1)$}
            \rput[bl](10.3,-6.775){$\kappa(z_2)$}
            \rput[bl](13.7,-6.775){$\kappa(z_3)$}
            \rput[bl](14.8,-6.775){$\kappa(z_1)$}
            \rput[bl](8.6,-6.375){$\tilde\kappa(\tilde z_1)$}
            \rput[bl](8.6,-4.775){$\tilde\kappa(\tilde z_2)$}
            \rput[bl](8.6,-2.375){$\tilde\kappa(\tilde z_3)$}
            \rput[bl](8.6,-0.775){$\tilde\kappa(\tilde z_1)$}
            \psframe[linecolor=black, linewidth=0.002, fillstyle=solid,fillcolor=colour1, dimen=outer](13.6,5.825)(11.2,1.825)
            \psline[linecolor=black, linewidth=0.04](10.6,6.625)(10.6,1.025)
            \psline[linecolor=black, linewidth=0.04](14.2,1.025)(14.2,6.625)
            \psline[linecolor=black, linewidth=0.04](9.6,5.025)(15.2,5.025)
            \psline[linecolor=black, linewidth=0.04](9.6,2.625)(15.2,2.625)
            \psdots[linecolor=black, dotsize=0.2](10.6,2.625)
            \psdots[linecolor=black, dotsize=0.2](14.2,5.025)
            \psbezier[linecolor=black, linewidth=0.06](9.6285715,1.0535715)(9.8,1.425)(9.6,1.625)(10.0,1.825000000000001)(10.4,2.025)(10.4,2.425)(10.6285715,2.625)(10.857142,2.825)(11.2,3.425)(11.771428,3.625)(12.342857,3.825)(13.3714285,4.1392856)(13.485714,4.482143)(13.6,4.825)(13.8,4.6821427)(14.2,5.053571)(14.6,5.425)(14.4,5.625)(14.6285715,5.910714)(14.857142,6.196429)(14.8,6.225)(15.2,6.625)
            \rput[bl](6.111111,-1.775){$\gamma$}
            \psframe[linecolor=black, linewidth=0.002, fillstyle=solid,fillcolor=colour1, dimen=outer](5.0,-1.375)(2.6,-5.375)
            \psline[linecolor=black, linewidth=0.04](2.0,-0.575)(2.0,-6.175)
            \psline[linecolor=black, linewidth=0.04](5.6,-6.175)(5.6,-0.575)
            \psline[linecolor=black, linewidth=0.04](1.0,-2.175)(6.6,-2.175)
            \psline[linecolor=black, linewidth=0.04](1.0,-4.575)(6.6,-4.575)
            \psdots[linecolor=black, dotsize=0.2](2.0,-4.575)
            \psdots[linecolor=black, dotsize=0.2](5.6,-2.175)
            \psbezier[linecolor=black, linewidth=0.06](1.0285715,-6.1464286)(1.2,-5.775)(1.0,-5.575)(1.4,-5.375)(1.8,-5.175)(1.8,-4.775)(2.0285714,-4.575)(2.2571428,-4.375)(2.6,-3.775)(3.1714287,-3.575)(3.7428572,-3.375)(4.7714286,-3.0607142)(4.885714,-2.7178571)(5.0,-2.375)(5.2,-2.517857)(5.6,-2.1464286)(6.0,-1.775)(5.8,-1.575)(6.0285716,-1.2892857)(6.257143,-1.0035714)(6.2,-0.975)(6.6,-0.575)
            \rput[bl](14.711111,-1.775){$\gamma$}
            \psframe[linecolor=black, linewidth=0.002, fillstyle=solid,fillcolor=colour1, dimen=outer](13.6,-1.375)(11.2,-5.375)
            \psline[linecolor=black, linewidth=0.04](10.6,-0.575)(10.6,-6.175)
            \psline[linecolor=black, linewidth=0.04](14.2,-6.175)(14.2,-0.575)
            \psline[linecolor=black, linewidth=0.04](9.6,-2.175)(15.2,-2.175)
            \psline[linecolor=black, linewidth=0.04](9.6,-4.575)(15.2,-4.575)
            \psdots[linecolor=black, dotsize=0.2](10.6,-4.575)
            \psdots[linecolor=black, dotsize=0.2](14.2,-2.175)
            \psbezier[linecolor=black, linewidth=0.06](9.6285715,-6.1464286)(9.8,-5.775)(9.6,-5.575)(10.0,-5.375)(10.4,-5.175)(10.4,-4.775)(10.6285715,-4.575)(10.857142,-4.375)(11.2,-3.775)(11.771428,-3.575)(12.342857,-3.375)(13.3714285,-3.0607142)(13.485714,-2.7178571)(13.6,-2.375)(13.8,-2.517857)(14.2,-2.1464286)(14.6,-1.775)(14.4,-1.575)(14.6285715,-1.2892857)(14.857142,-1.0035714)(14.8,-0.975)(15.2,-0.575)
            \psdots[linecolor=black, dotstyle=triangle*, dotsize=0.2](11.6,-1.775)
            \psdots[linecolor=black, dotstyle=triangle*, dotsize=0.2](13.2,-5.175)
            \rput[bl](11.8,-1.975){$p$}
            \rput[bl](12.8,-5.175){$\tilde p$}
        \end{pspicture}
    }
    \caption{\label{tppl-case-2-example-figure}
        {\bf An example construction of $\gamma$
        in the proof of the Three Point Prescription Lemma \ref{tppl}
        in Case 2 of Lemma \ref{cpc-lemma}.}
        Here the shaded region is $S_\cap$.
        In this example $\cpc(f) = (0, 2)$
        with $U \cap \partial K \subset [z_2 \to z_3]_{\partial K}$.
    }
\end{figure}

Finally, we have that either
\begin{itemize}
    \item 
        $\cpc(f) = (1, 1)$, and
        there exists an $i$ so that
        $(\kappa(z_i), \tilde \kappa(\tilde z_i))\allowbreak \in\allowbreak S_\cap$,
        or
    \item
        $\cpc(f) = (0, 2)$, and
        $U \cap \partial K \subset [z_2 \to z_3]_{\partial K}$.
\end{itemize}
In both of these cases, we set $\gamma = \gamma'$.
We refer the reader to Figure \ref{tppl-case-2-example-figure} for an example.
Let $P, \tilde P$ denote the points
of $\partial K \cap \partial \tilde K \cap f$
where $\partial K$ enters $\tilde K$,
respectively where
$\partial \tilde K$ enters $K$,
using notation analogous to that of Section \ref{chap:torus}.
Because of the orientations
on $\partial f \cap \partial K$
and $\partial f \cap \partial \tilde K$
forced by the condition that $f \subset K \cap \tilde K$,
we have that
$p = (\kappa(P), \tilde \kappa(P))$
is in the northwest corner of $S_\cap$,
and that $\tilde p = (\kappa(\tilde P), \tilde \kappa(\tilde P))$
is in the southeast corner.
Furthermore, in both cases
(whether $\cpc(f)$ is $(1, 1)$ or is $(0, 2)$),
the curve $\gamma$
separates $S_\cap$
so that $p \in \Delta_\uparrow(z_1, \gamma)$
and $\tilde p \in \Delta_\downarrow(z_1, \gamma)$.
It follows from Lemma \ref{prop:computing index from torus}
that $\eta(\phi) = \eta(\phi') + 1$, and we are done.
\qed

\begin{remark} \label{closing-remark}
    The conditions of Lemma \ref{lem:3p}
    that $\tilde z_1, \tilde z_2, \tilde z_3 \not \in \partial K$,
    that $z_1, z_2, z_3 \allowbreak \not \in \partial \tilde K$, and
    that $K$ and $\tilde K$ are in transverse position
    may appear to the reader to be too strong.
    There are several reasons that the lemma is stated as it is.

    First, it may seem that even if those conditions are violated,
    it might be be possible to
    modify $K$ and $\tilde K$ ``very slightly''
    so that the stated hypotheses of Lemma \ref{lem:3p}
    \emph{do} hold, and
then, for example, apply the lemma to the modified Jordan domains,
hoping to recover the desired conclusion
for our original $K$ and $\tilde K$
by reversing our modifications to them.
A major issue with this approach is that
Lemma \ref{lem:3p} as stated does not give any
lower bound on $|\phi(z) - z|$;
the author is not aware of a simple amendment to the proof
that would provide one.
Because of this, it is \emph{a priori} not trivial to arrange the modifications
to $K$ and $\tilde K$ mentioned earlier
so that they are guaranteed to be reversible
without altering $\eta(\phi)$.

Next, it is not clear what a good set of hypotheses is
for a strengthened version of
Lemma \ref{lem:3p}. For the conclusion
of the lemma to hold, it is necessary that there be \emph{some}
indexable $\phi : \partial K \to \partial \tilde K$
identifying $z_i \mapsto \tilde z_i$, and
without the transverse position hypothesis,
it is not trivial to give necessary and sufficient conditions for the
existence of such a $\phi$. For example, no such $\phi$ exists if
$[z_1 \to z_2]_{\partial K}
\subset [\tilde z_1 \to \tilde z_2]_{\partial \tilde K}$.
For a more involved example,
treating real numbers as points in the complex plane, let
$z_1 = 2, z_2 = 0, \tilde z_1 = 1, \tilde z_2 = 3$, so that
$[z_1 \to z_2]_{\partial K} \subset \bbR$ and
$[\tilde z_1 \to \tilde z_2]_{\partial \tilde K} \subset \bbR$.
Then no indexable
$\phi : \partial K \to \partial \tilde K$ exists which identifies
$z_1 \mapsto \tilde z_1$ and $z_2 \mapsto \tilde z_2$.

It is possible that Lemma \ref{lem:3p} holds under the hypothesis that there
exists \emph{any} indexable $\phi : \partial K \to \partial \tilde K$
identifying $z_i \mapsto \tilde z_i$.
However,
this formulation would be somewhat unsatisfying absent an understanding
of when such a $\phi$ exists, and
the author does not know of a modification of the proof given in this article
which could prove this stronger statement.
\end{remark}

\end{document}